\documentclass[11pt]{article}
\usepackage[margin=1in]{geometry}

\usepackage{amsthm,amsmath}
\usepackage{amssymb,amsfonts,latexsym,mathtools}
\usepackage{graphicx,color}
\usepackage{enumitem}
\usepackage{ifthen}

\newtheorem{theorem}{Theorem}
\numberwithin{theorem}{section}

\newtheorem{lemma}[theorem]{Lemma}

\newtheorem{corollary}[theorem]{Corollary}

\def\Halmos{\mbox{\quad$\square$}}

\usepackage{commands}

\begin{document}
	\title{Iterative Collaborative Filtering for Sparse Matrix Estimation}
	\author{Christian Borgs \footnote{UC Berkeley, Berkeley, CA;
			\texttt{borgs@berkeley.edu}} 
			\,\, Jennifer T. Chayes  \footnote{UC Berkeley, Berkeley, CA;
			\texttt{jchayes@berkeley.edu}}
			\,\, Devavrat Shah  \footnote{Massachusetts Institute of Technology, Cambridge, MA;
			\texttt{devavrat@mit.edu}}
			\,\, Christina Lee Yu \footnote{Cornell University, Ithaca NY;
			\texttt{cleeyu@cornell.edu}}}
	\date{}
	\maketitle

	\begin{abstract}
We consider sparse matrix estimation where the goal is to estimate an $n\times n$ matrix from noisy observations of a small subset of its entries. We analyze the estimation error of the popularly utilized collaborative filtering algorithm for the sparse regime. Specifically, we propose a novel iterative variant of the algorithm, adapted to handle the setting of sparse observations. We establish that as long as the fraction of entries observed at random scale as $\frac{\log^{1+\kappa}(n)}{n}$ for any fixed $\kappa > 0$, the estimation error with respect to the $\max$-norm decays to $0$ as $n\to\infty$ assuming the underlying matrix of interest has constant rank $r$. Our result is robust to model mis-specification in that if the underlying matrix is approximately rank $r$, then the estimation 
error decays to the approximate error with respect to the $\max$-norm.  In the process, we establish algorithm's ability  
to handle arbitrary bounded noise in the observations. 
	\end{abstract}

	\section{Introduction}
\label{sec:intro}

We consider the task of sparse matrix estimation given noisy observations. Let $F$ be an $n \times n$ matrix which we 
would like to estimate, and let $Z$ be a noisy signal of matrix $F$ such that $\E[Z] = F$. Let $\cE \subset [n] \times [n]$ denote the subset of indices that are observed. In particular, we observe matrix $M$ where $M(u,v) = Z(u,v)$ for $(u,v) \in \cE$, and $M(u,v) = 0$ for $(u,v) \notin \cE$. 
We assume that the entries of $Z$ are independent random variables, and we assume a Bernoulli sampling model; each $(u, v) \in [n] \times [n]$ is in $\cE$
with probability $p \in (0,1]$ independently. The goal is to estimate $F$. 

As a prototype for such a problem,  consider a noisy observation of a social network where observed interactions are 
signals of true underlying connections. We might want to predict the probability that two users would choose to connect 
if recommended by the platform, e.g. LinkedIn. As a second example, consider a recommendation system where we observe 
movie ratings provided by users, and we may want to predict the probability distribution over ratings for specific movie-user pairs.  
A popular collaborative filtering approach suggests using ``similarities'' between pairs of users to estimate the probability that a 
connection is formed or the probability a user likes a particular movie. Traditionally, the similarities between pair of users in a social network is computed 
by comparing the set of their friends, or in the context of movie recommendation, by comparing commonly rated movies. In the 
sparse setting, most pairs of users have no common friends, or most pairs of users have no commonly rated movies; 
thus there is insufficient data to compute the traditional similarity metrics.

In this work, the primary interest is to provide a principled way to extend the simple, intuitive approach of computing similarities
between pair of users or items in order to perform sparse matrix estimation via nearest neighbor collaborative filtering. We propose to do so by incorporating information within a larger radius 
neighborhood of the data graph rather than restricting only to immediate neighbors. This variation of collaborative filtering and its analysis in this work
can be viewed as a natural extension of the work by \cite{AbbeSandon15a, AbbeSandon16} in the context of stochastic block model
and \cite{ZhangLevinaZhu15, LeeLiShahSong16} for traditional collaborative filtering.

\subsection{Summary of Contributions}
The primary contribution of this work is an analysis of an iterative collaborative filtering algorithm in the sparse regime. We consider 
the setting of a latent variable model where the matrix $F = [F(u,v)]$ can be described by a latent function $f$ evaluated over latent variables associated to the coordinates. In particular, we assume that $F(u, v) = f(\theta_u, \theta_v)$ where $f$ is a piece-wise Lipschitz function, and
$\theta_u, \theta_v \in [0,1]$ are coordinate latent variables sampled uniformly at random. Details of the model
are described in Section \ref{sec:model}.

As the main result of this work, we establish that with high probability the max entry-wise error associated with the resulting 
estimate converges to $0$ as long as the latent function $f$ when regarded as an integral operator has finite spectrum with 
constant rank $r$ and $p = \Omega(n^{-1 + \kappa})$ for $\kappa > 0$. In addition, if we have knowledge of the spectrum, the algorithm 
can be improved so that the max entry-wise error of the estimate converges to zero as long as $p = \Omega(n^{-1} \ln^{1 + \kappa} n)$ 
for any $\kappa > 0$. 
We also establish robustness of our result with respect to the low rank requirement of $f$. In particular, 
we provide a robust version of our result that holds when $f$ has $\eps$-approximate rank $r$, i.e. 
there exists a rank $r$ function that approximates $f$ within $\eps$ with respect to the $\ell_\infty$ 
norm. We establish it by arguing that if all the observed entries are perturbed arbitrarily or adversarially 
within $\eps$, then the algorithm estimates for each entry are perturbed by at most 
$O(\max(\sqrt{\eps}, \eps))$. The efficacy of the proposed algorithm with respect to arbitrary 
noise is an interesting result on its own.

Algorithmically and methodologically, our work builds on \cite{AbbeSandon15a, AbbeSandon15b, AbbeSandon16}, which estimates clusters 
of the stochastic block model by computing distances from local neighborhoods around vertices. We improve upon their algorithm and
analysis to provide bounds on the maximum entrywise estimation error for the general latent variable model with finite spectrum.
This includes a larger class of generative models such as mixed membership stochastic block models, in contrast to their work which focuses on the stochastic 
block model with non-overlapping communities. We note that the algorithm considered in this work, uses the knowledge of which 
entries are observed and which are not, in line with the literature on matrix estimation. In the setting of clustering cf. 
\cite{AbbeSandon15a, AbbeSandon15b, AbbeSandon16}, such a knowledge is absent from the purview of the algorithm. 

Withthe  exception of a few recent results, by and large the literature on matrix estimation has focused on providing estimation error bounds with
respect to the normalized Frobenius norm. In contrast, we provide bounds on the max entry-wise estimation error which is
a lot more challenging. {Our bounds are restricted to the latent variable model, while the traditional matrix estimation literature
considers the underlying matrix to be an arbitrary instance from the family of (approximately) low-rank matrices with `incoherence'-like
conditions. Indeed, understanding the relationship between these two seemingly different model classes remains an important direction
for future work. }

A weaker version of this result was published in the NeurIPS conference as \cite{NIPS2017_7cc23420}. In contrast, this paper provides sharper bounds for both the MSE and max-norm error that improves the exponent in the convergence rates. We have also included a perturbation analysis of the algorithm that shows under “adversarial” bounded noise, the error scales gracefully with the bound on the noise. This enables analysis of our work for the approximately low-rank setting. We have also included a modified algorithm that achieves the same rates with a reduced computational complexity, and we have shown extensions of our results to relaxed modeling assumptions on the latent variable model. We have added empirical evaluation of our method compared with state-of-art methods. 

\subsection{Related Work}
\label{sec:related_work}
The related work includes that of matrix estimation or completion, collaborative filtering, and graphon estimation arising from the asymptotic theory of graphs. We provide a brief overview of prior works for each of these topics. 

In the context of matrix estimation or completion, there has been much progress under the low-rank assumption and additive noise model. Most theoretically founded methods are based on spectral decompositions or minimizing a loss function with respect to spectral constraints, c.f. \cite{KeshavanMontanariOh10a, KeshavanMontanariOh10b, CandesRecht09, CandesTao10, Recht11, NegahbanWainwright11, DavenportPlanBergWootters14, ChenWainwright15, Chatterjee15, Xu18}. In a nutshell, this collection of works establishes that if the underlying matrix
has rank $r$, then it can be estimated so that the estimator has normalized Mean Squared Error (MSE) going to $0$ as $n\to \infty$ as long as $p = \Omega(r n^{-1} \log n)$. Furthermore, \cite{KeshavanMontanariOh10b, CandesPlan10} showed that $\omega(r n^{-1})$ 
samples are necessarily required for such a guarantee. These near optimal sample complexity results hold when the noise in each entry of the matrix is independent and identically distributed. For the setting of generic noise and the general latent variable model where the latent function is analytic, \cite{Chatterjee15, Xu18} provide an estimator for which the MSE decays to $0$ as $n\to\infty$ as long as $p = \Omega(n^{-1} {\sf poly}(\log n))$. 

The guarantee with respect to MSE does not necessarily guarantee recovery of {\em all} entries accurately. Indeed, bounding max entrywise error provides such a guarantee as established by our result. In parallel with our work, there has been recent progress 
on developing matrix estimation methods that provide max entrywise bounds for matrices with rank $r$. In particular, for sufficiently `nice' rank $r$ matrices, 
\cite{abbe2020entrywise} establish that a simple spectral algorithm can recover the matrix with max entrywise error decaying to $0$ 
as long as $p = \Omega(\log n / n)$.  Indeed, improving such max entrywise guarantee has been actively pursued over the past few years 
witnessed in the growing body of works, cf. \cite{chen2019spectral}, \cite{zhong2018near}, \cite{cai2019subspace}, \cite{ma2018implicit} and \cite{DingChen20}.


The collaborative filtering method has been successfully employed across industry applications (Netflix, Amazon, Youtube) due to its simplicity and scalability, c.f. \cite{goldberg92, Linden03,korenHandbook,Ning2015}; however the theoretical results have been relatively sparse. We call special attention to the recent works by \cite{ZhangLevinaZhu15, LeeLiShahSong16, LiShahSongYu20}
which provide a non-parametric statistical perspective for the traditional collaborative filtering method. In particular, they suggest that the practical success of these methods across a variety 
of applications may be due to its ability to capture local structure like the classical nearest neighbor or kernel regression method. They establish that as long as the latent function $f$ 
is Lipschitz, the MSE of the resulting estimator decays to $0$ as $n\to\infty$ as long as $p = \omega(n^{-\frac12})$. A key limitation of this approach is that it requires a dense dataset with 
sufficient entries in order to compute similarity metrics, requiring that each pair of rows or columns has a growing number of overlapped observed entries, which does not hold when $p = o(n^{-1/2})$. 

Graphons emerged as the limiting object of a sequence of large dense graphs, c.f. \cite{BorgsChayesLovaszSosVestergombi08, DiaconisJanson07,Lovasz12}, with recent work extending the theory to sparse graphs, c.f. \cite{BorgsChayesCohnZhao14a, BorgsChayesCohnZhao14b, BorgsChayesCohnHolden16, VeitchRoy15}. In the graphon estimation problem, one observes a single instance of a random graph sampled from an underlying latent variable model, and the goal is to estimate the function that governs the edge probabilities of the graph. \cite{GaoLuZhou15,KloppTsybakovVerzelen15} provide minimax optimal rates for graphon estimation; however a majority of the proposed estimators are not computable in polynomial time, since they require optimizing over an exponentially large space (e.g. least squares or maximum likelihood), c.f. \cite{WolfeOlhede13, BorgsChayesCohnGanguly15, BorgsChayesSmith15, GaoLuZhou15, KloppTsybakovVerzelen15}. \cite{BorgsChayesCohnGanguly15} provides a polynomial time method based on 
degree sorting in the special case when the expected degree function is monotonic. 
\cite{Xu18} analyzes universal singular value thresholding (USVT) for graphon estimation in settings that the spectrum decays quickly, showing convergence rates which matches the minimax optimal rate for low dimensional smooth functions. 

Stochastic block model (SBM) parameter estimation is an instance of graphon estimation, where the underlying function has a specific structure. Under the SBM, each vertex is associated to one of $r$ community types, and the probability of an edge is a function of the community types of both endpoints. This implies that the edge probability function is block constant. Estimating the $n\times n$ parameter matrix becomes an instance of matrix estimation with
a technical distinction -- all entries are fully observed, i.e. each edge is present (1) or absent (0). In SBM, the expected matrix is at most rank $r$ due to its block structure. Precise thresholds for cluster detection (better than random) and estimation have been established by \cite{AbbeSandon15a, AbbeSandon15b, AbbeSandon16}. As mentioned before, our work, both algorithmically and methodically is closely
related to their work. The mixed membership stochastic block model (MMSBM) allows each vertex to be associated to a length $r$ vector, which represents its weighted membership in each of the $r$ communities. The probability of an edge is a function of the weighted community memberships vectors of both endpoints, resulting in an expected matrix with rank at most $r$. Recent work by \cite{SteurerHopkins17} provides an algorithm for weak detection for MMSBM with sample complexity $r^2 n$, when the community membership vectors are sparse and evenly weighted. 
They provide partial results to support a conjecture that $r^2 n$ is a computational lower bound, separated by a gap of $r$ from the information theoretic lower bound of $r n$. This gap was 
first shown in the simpler context of the stochastic block model \cite{DecelleKrzakalaMooreZdeborova11}. \cite{XuMassoulieLelarge14} proposed a spectral clustering method for inferring 
the edge label distribution for a network sampled from a generalized stochastic block model. When the expected function has a finite spectrum decomposition, i.e. low rank, then they 
provide a consistent estimator for the sparse data regime, with $\Omega(n \log n)$ samples. 

In the above discussion, we have focused primarily on the sample complexity required for consistent estimation, 
i.e. the scaling of the number of samples required ($pn$) such that the normalized estimation 
error such as the MSE or max-norm goes to $0$. 
When consistent estimation is feasible, we can further consider the rate of decay of the error guarantees. To that end, we provide a brief overview of the minimax scaling with respect to boudns on the MSE. 
 \cite{Chatterjee15} identifies a minimax lower bound on the scaling of the MSE for a generic matrix estimation task characterized by the nuclear norm of the target matrix. 
In particular, for symmetric matrices with nuclear norm bounded by $\delta$, the minimax MSE scaling is lower bounded 
by $\min\big(\frac{\delta}{\sqrt{n^3 p}}, \frac{\delta^2}{n^2}, 1 \big)$; furthermore \cite{Chatterjee15} argues that the universal singular value
thresholding achieves this scaling. This bound holds even in the scenario where observed entries are noiseless. This characterization however is loose for the setting of low-rank matrices. 
Observe that for rank $r$ symmetric matrices with entries bounded in $[-1,1]$, the nuclear norm can scale as $n \sqrt{r}$; resulting in a bound of $\sqrt{\frac{r}{np}}$ (for small enough $p$) \cite{Chatterjee15}.
For the setting of rank $r$ matrices with noiseless observations, \cite{KeshavanMontanariOh10a, KeshavanMontanariOh10b} provide an estimator with MSE scaling as $\frac{r}{np}$ for $p = \Omega(1/n)$. This points to the fact that the class of matrices with bounded nuclear norm is more complex than the class of rank $r$ matrices with bounded entries. 
In the setting of low rank graphon estimation (i.e. binary observations), \cite{gao2015rate, klopp2017oracle} show a minimax lower bound 
on the MSE scaling as $\frac{\log r}{pn}$ for small enough $p = \Omega(\log r / n)$; however the existence of a computationally efficient estimator that achieves this lower bound under the more general noise setting of graphon estimation is still an open research direction.

\section{Setup}
\label{sec:model}

\subsection{Model and Assumptions}

Recall that our goal is to estimate the $n \times n$ matrix $F$; $Z$ is a noisy signal of matrix $F$ such that $\E[Z] = F$. The available data is denoted by $(\cE, M)$, where $\cE \subset [n] \times [n]$ denotes the subset of indices for which data is observed, and $M$ is the $n \times n$ data matrix where $M(u,v) = Z(u,v)$ for $(u,v) \in \cE$,  and $M(u,v) = 0$ for $(u,v) \notin \cE$. The observations can be equivalently represented by an directed weighted graph $\cG$ with vertex set  $[n]$, edge set $\cE$, and edge weights given by $M$. We assume that $\{Z(u,v)\}_{(u,v) \in [n]^2}$ are independent random variables across all indices with 
$\E[Z(u,v)] = F(u,v)$, and that the underlying matrix and observations are bounded, i.e. $F(u,v), Z(u,v) \in [0,1]$. We assume a uniform Bernoulli sampling model, where each entry is observed independently with probability $p$, i.e. $\{\Ind((u,v) \in \cE)\}_{(u,v) \in [n]^2}$ are independent Bernoulli$(p)$ random variables.

\paragraph{Latent Variable Model.} Assume that each $u \in [n]$ is associated to a latent feature variable $\theta_u \sim U[0,1]$, which is drawn independently across indices $[n]$ uniformly on the unit interval. We assume that the expected data matrix can be described by the latent function $f$, i.e. $F(u,v) = f(\theta_u, \theta_v)$, where $f: [0,1]^2 \to [0,1]$ is a symmetric bounded function. The symmetry assumption can be easily relaxed but is assumed for ease of notation in the analysis. The latent function $f$ is assumed to be fixed and independent of the dimension $n$. We additionally impose local neighborhood properties that are primarily used in the nearest neighbor portion of the analysis. We will assume that $f$ is Lipschitz, but this assumption can be relaxed as discussed in Section \ref{sec:lvm_disc}.


\paragraph{Low Rank.} We assume that the latent function $f$ has finite spectrum with rank $r$ when regarded as an integral operator, i.e. for any $ \theta_u, \theta_v \in [0,1]$,
\[f(\theta_u, \theta_v) = \sum_{k=1}^r \lambda_k q_k(\theta_u) q_k(\theta_v),\]
where $\lambda_k \in \mathbb{R}$ for $1\leq k\leq r$, and $q_k$ are orthonormal $\ell_2$ functions for $1\leq k \leq r$ such that 
\[\int_0^1 q_k(y)^2 dy = 1 \text{ and } \int_0^1 q_k(y) q_h(y) dy = 0 \text{ for } k \neq h \in [r]. \]
We assume there exists some $B$ such that $\sup_{y \in [0,1]} |q_k(y)| \leq B$ for all $k \in [r]$.
Let $\Lambda$ denote the $r \times r$ diagonal matrix with $\{\lambda_k\}_{k \in [r]}$ as the diagonal entries, 
and let $Q$ denote the $r \times n$ matrix where $Q(k,u) = q_k(\theta_u)$. Since $Q$ is a random matrix depending on 
the sampled $\theta$, it is not guaranteed to be an orthonormal matrix (even though $q_k$ are orthonormal functions). By 
definition, it follows that $F = Q^T \Lambda Q$. Let $r' \leq r$ be the number of distinct valued eigenvalues amongst 
$\{\lambda_k\}_{k\in [r]}$. Let $\tLambda$ denote the $r \times r'$ matrix where $\tLambda(a,b) = \lambda_a^{b-1}$.

The finite spectrum assumption also implies that the model can be represented by latent variables in the $r$ dimensional Euclidean space, where the latent variable for node $i$ would be the vector $(q_1(\theta_i), \dots q_r(\theta_i))$, and the latent function would be bilinear, having the form $$f(\vec{q}, \vec{q}') = \sum_k \lambda_k q_k q'_k = q^T \Lambda q'.$$ This condition also implies that the expected matrix $F$ is low rank, which includes scenarios such as the mixed membership stochastic block model and finite degree polynomials. The function $f$ is fixed with respect to $n$, the rank $r$ is assumed to be finite in the low rank setting.

The mixed membership model for network data can be represented with a finite spectrum latent variable model. Each coordinate is associated to a vector $\pi \in \Delta_r$, sampled iid from a distribution $P$. For two nodes with respective types $\pi$ and $\pi'$, the observed interaction is $f(\pi, \pi') = \sum_{ij} \pi_i \pi'_j B_{ij} = \pi^T B \pi'$,
where $B \in [0,1]^{r \times r}$ and assumed to be symmetric. Since $B$ is symmetric, there exists a diagonal decomposition $B = U \tilde\Lambda U^T$ with $u_k$ denoting the eigenvectors, such that $f(\pi, \pi') = \sum_{k=1}^r \tilde\lambda_k u_k^T \pi u_k^T \pi'$. It follows from this decomposition that the Hilbert-Schmidt integral operator associated to function $f: \Delta_r \times \Delta_r \to [0,1]$ has finite spectrum with rank at most $r$.

Interaction data arising from symmetric finite degree polynomials also leads to finite spectrum latent variable models. Let $f(x, y)$ be a finite degree symmetric polynomial, represented by $f(x,y) = \sum_{i=0}^r \sum_{j=0}^r c_{ij} x^i y^j$, where $c_{ij} = c_{ji}$ for all $ij$. Let $\bx = (1, x, x^2, \dots x^r)$ and $\by = (1, y, y^2, \dots y^r)$, and let $C$ denote the $(r+1)\times(r+1)$ matrix with entries $[c_{ij}]$, so that $f(x,y) = \bx^T C \by$. Since $C$ is symmetric, there exists a diagonal decomposition $B = U \tilde\Lambda U^T$ with $u_k$ denoting the eigenvectors, such that $f(x, y) = \sum_{k=1}^r \tilde\lambda_k u_k^T \bx u_k^T \by$. It follows from this decomposition that the Hilbert-Schmidt integral operator associated to function $f$ has finite spectrum with rank at most $r$.

\paragraph{Approximately Low Rank.}
More generally, we shall consider approximately low-rank $f$ cf. \cite{udell2019big}. 
Specifically, for a given $\varepsilon > 0$, a symmetric function $f$ is said to have $\varepsilon$-approximate rank $r$ if 
\begin{align}\label{eq:approx.lowrank}
\sup_{\theta_u, \theta_v \in [0,1]} \big|f(\theta_u, \theta_v) -\sum_{k=1}^r \lambda_k q_k(\theta_u) q_k(\theta_v) \big| \leq \varepsilon,
\end{align}
where $\lambda_k \in \mathbb{R}$ for $1\leq k\leq r$, and $q_k$ are orthonormal $\ell_2$ functions for $1\leq k \leq r$.
In this case, it follows that $F = Q^T \Lambda Q + \beps$ where $\beps = [\varepsilon_{ij}] \in \Reals^{n \times n}$ is such that $\max_{ij} |\varepsilon_{ij}| \leq \varepsilon$. That is, the matrix $F$ is approximately rank $r$. Functions $f$ which do not have finite spectrum, but for which the eigenvalues decay quickly can be shown to have approximately low rank. \cite{Chatterjee15,Xu18} use this observation to analyze the USVT algorithm for latent variable model estimation with Lipschitz, Holder, and Sobolev functions. \cite{udell2019big} also show that any analytic function with bounded derivatives has approximately low rank. Recall again that we assume the function $f$ is fixed with respect to $n$, but we can consider the choice of $\varepsilon$ to be dependent on $n$, so that the $\varepsilon$ approximate rank $r$ would grow with respect to $n$.

\subsection{Discussion on Latent Variable Model} \label{sec:lvm_disc}


The latent variable model assumes a random generative model on the underlying matrix $F$, as opposed to the typical deterministic incoherence style conditions found in the literature. The generative model assuming i.i.d. sampled latent variables and boundedness of the eigenfunctions of $f$ guarantee similar properties as incoherence with high probability, as any single row or column will not dominate the signal in a way that deviates too much from the typical values of $f$. The i.i.d. sampling assumption on the latent variables is used in analyzing the local neighborhoods of the observation graph, however this assumption can likely be replaced by regularity assumptions over the empirical distribution of the latent factors for large $n$, e.g. if the latent factors are close to a typical sample set from a well-behaved underlying distribution.

The Lipschitzness assumption of $f$ together with the assumption that $\theta_u \sim U[0,1]$, guarantees that for any given $u \in [n]$ there are sufficiently many other coordinates $v \in [n]$ such that the observed entries are similar across both rows or columns. These assumptions can be relaxed as long as the key property of ``sufficiently many similarly behaving coordinates'' is maintained. As examples, a piecewise Lipschitz function $f$ or a setting with finite latent types would also satisfy the needed local neighborhood properties. Similarly, the scalar assumption on the latent variables and the uniform distribution $U[0,1]$ are not crucial and can be relaxed to i.i.d. sampled random latent vectors from a larger class of distributions. The critical conditions to maintain are the finite spectrum of $f$, boundedness of eigenfunctions, and local neighborhood properties. The local measure needs to be concentrated enough relative to the rate of change in the function $f$ so that when $n$ points are sampled from the space, there are sufficently many ``nearby neighbors'' for whom the function behaves similarly for any given point we would want to estimate. This primarily affects the nearest neighbor portion of the algorithm and analysis. \cite{LiShahSongYu20} also provides a formal discussion and results for extending the nearest neighbor analysis to accommodate settings beyond scalar Lipschitz functions. Our model can also be extended to asymmetric matrix settings and categorical data. Section \ref{sec:thm_extensions} discuss how our theorem extend to some of these model variations.

\subsection{Goal}

The goal is to produce $\hat{F}$, an estimate of $F$, using observation matrix $M$ and knowledge of $\cE$. We measure the estimation error through the maximum entry-wise error and the mean squared error. The maximum entry-wise error or $\infty$-norm of the error matrix $\hat{F} - F$ is defined as
\begin{align}
\| \hat{F} - F\|_{\max} & = \max_{u, v} | \hat{F}(u, v) - F(u,v)|. 
\end{align}
We will provide bounds on this that hold with high probability, that is, with probability converging to $1$ as $n\to \infty$. The mean squared error (MSE) is defined as
\begin{align}
\MSE(\hat{F}) & = \frac{1}{n^2} \mathbb{E}\Big[ \sum_{u, v} (\hat{F}(u, v) - F(u,v))^2\Big].
\end{align}
In measuring error either with high probability or in expectation, the randomness is considered over the data generation process.


\section{Algorithm}
\label{sec:alg}

We propose and analyze a variation of the similarity based collaborative filtering algorithm. At its core, the collaborative filtering algorithm attempts to produce
the estimate $\hat{F}(u,v)$ by averaging over observed entries $F(u',v')$ for a subset of tuples $(u',v')$ such that $u'$ is ``similar" to $u$ and $v'$ is ``similar" to $v$. 

\medskip
\noindent{\em Sample Splitting.} To state the precise algorithm, for technical reasons, we shall use sample splitting. Recall that $\cE \subset [n]^2$ denotes
the set of indices for which we observe noisy signals of $F(u,v)$, i.e. for each $(u,v) \in \cE$, $M(u,v) = Z(u,v)$ where $\E[Z(u,v)] = F(u,v)$. We assumed that $\cE$ is generated 
according to a Bernoulli($p$) sampling model, i.e. for each $(u, v) \in [n]^2$, it belongs to $\cE$ with probability $p$ independently. We split the samples $\cE$ into three subsets as follows: for each tuple or edge $(u,v) \in \cE$, with probability $1/4$ it is placed in $\cEp$, with probability $1/4$ it is placed in $\cEpp$, and
with the remaining $1/2$ probability it is placed in $\cEppp = \cE \backslash (\cEp \cup \cEpp)$. 

We will use additional ``virtual'' edges that will aid in estimating the distance as part of the algorithm. To that end, 
note that conditioned on the edge set $\cEp$, for some $(u,v) \notin \cEp$, $\Prob{(u,v) \in \cEpp | \cEp} = \frac{p}{4-p} = \pp$.
Furthermore, conditioned on $\cEp$, $\Ind((u,v) \in \cEpp)$ are independent random variables. Conditioned on $\cEp$, we generate a random subset $\cEpind \subseteq \cEp$ such that 
each $(u,v)\in \cEp$ is included in $\cEpind$ independently with probability $\pp = \frac{p}{4-p}$. 
Therefore, conditioned on $\cEp$, the set $\cEpind \cup \cEpp$ is distributed according to a Bernoulli($\pp$) sampling model, where each $(u,v) \in [n]^2$ are included in $\cEpind \cup \cEpp$ independently with probability $\pp$.


For each $u, v \in [n]$, define $\Mp(u,v) = \Ind((u,v) \in \cEp) M(u,v)$, $\Mpind(u,v) = \Ind((u,v) \in \cEpind) M(u,v)$, ~$\Mpp(u,v) = \Ind((u,v) \in \cEpp) M(u,v)$, and 
$\Mppp(u,v) = \Ind((u,v) \in \cEppp) M(u,v)$; let $\Mp = [\Mp(u,v)]$, $\Mpind = [\Mpind(u,v)]$, $\Mpp = [\Mpp(u,v)]$ and $\Mppp = [\Mppp(u,v)]$
denote the associated $n\times n$ matrices. Note that $\Mpind$ is strictly contained within $\Mp$ as $\cEpind \subseteq \cEp$. The algorithm will use observations $\Mp$ and $\Mpp$ to producing distance estimates $\hat{d}$, and it uses observations $\Mppp$
 to produce the final estimate $\hat{F}$ given $\hat{d}$. 

\medskip
\noindent
{\em Noisy Nearest Neighbor Algorithm.} We consider the following noisy nearest neighbor algorithm described below, followed by three different subroutines to compute distances depending on the sparsity regime of the dataset. \\
(1) Compute distances $\hat{d}(u,v)$ between pairs of coordinates $u,v \in [n]^2$ using  $\Mp$ and $\Mpp$. \\
(2) For each $u, v \in [n]^2$, produce an estimate
\begin{align}
\hat{F}(u,v) & = \tfrac{1}{|\cEppp_{uv}|} \textstyle\sum_{(a,b) \in \cEppp_{uv}} \Mppp(a,b), \label{eq:estimate}
\end{align}
where $\cEppp_{uv} = \{(a,b) \in \cEppp ~:~ \hat{d}(u,a) < \eta, \hat{d}(v,b) < \eta\}$ for some small enough $\eta > 0$.

We will choose the threshold $\eta = \eta(n)$ depending on the local geometry of the latent feature space with 
respect to $\hat{d}(u,v)$, in order to guarantee that $\eta(n)$ is small enough to drive the bias to zero, yet large enough to 
ensure $|\cEppp_{uv}|$ diverges so that the variance due to observation noise is small.  The key part of the algorithm is determining how to estimate 
the distances $\hat{d}(u,v)$. In what follows, we describe three variations depending upon the observation density, $p$. 

\subsection{Estimating Distance $\hat{d}$}

\medskip
\noindent
{\em Dense Regime.} When $p = \omega(n^{-\frac12})$, it is feasible to compute distances by simply looking 
at the overlapping entries; this is popularly done in practice \cite{goldberg92} as well as analyzed theoretically in the recent works 
\cite{ZhangLevinaZhu15, LeeLiShahSong16}. For any $(u,a) \in [n]^2$, 
\begin{align}
\hat{d}(u,a) = \tfrac{1}{|\cO_{ua}|} \textstyle\sum_{y \in \cO_{ua}} (M(u,y) - M(a,y))^2,
\end{align}
where $\cO_{ua} = \{y \in [n] : (u,y), (a,y) \in \cEp\}$. This is a finite sample approximation of $\int_0^1 (f(\theta_u, y) - f(\theta_v,y))^2~dy$. 
When $p = \omega(n^{-\frac12})$, it follows that $|\cO_{ua}| = \omega(1)$ for all $u, a \in [n]^2$ with high probability, so that $\hat{d}(u,a) \approx \int_0^1 (f(\theta_u, y) - f(\theta_v,y))^2~dy$. \cite{LeeLiShahSong16} subsequently prove that for any Lipschitz latent function $f$ the MSE decays to $0$ as $n\to\infty$ as long as $p = \omega(n^{-\frac12})$. The arguments of \cite{LeeLiShahSong16} 
can be adapted to show that the maximum entry-wise error decays to $0$ with high probability as well.  However, for $p = o(n^{-\frac12})$, for 
most $u, a \in [n]^2$, $\cO_{ua} = \emptyset$ with high probability and hence a different approach is needed -- overcoming the sparse regime is the primary interest of this work. 

\medskip
\noindent
{\em Sparse Regime.} Consider the sparse regime where $p = n^{-1 + \kappa}$ for any $\kappa \in (0, \frac12)$; in this regime the overlap is small and thus new distance estimates are required. Recall that the function $f$ has finite spectrum, i.e. $f(\theta_u, \theta_v) = \sum_k \lambda_{k=1}^r q_k(\theta_u) q_k(\theta_v)$. We 
propose an estimator which approximates $d(u,v) = \|\Lambda^{t} Q (e_u - e_v)\|_2^2$ by comparing depth $t$ neighborhoods of $u$ 
and $v$ in the data graph ${\cal G} = ([n], \cEp)$. Specifically, let the weight of an edge $(a, b) \in \cEp$ in graph $\cG$ be the observed value $M(a, b)$ ($=M^\prime(a,b)$). 
By 
assumption, in expectation this weight equals $F(a, b) = f(\theta_a, \theta_b)$. Therefore, the product of weights along a path from $u$ to $y$, 
of length $t$, denoted as $(u, x_1, \dots, x_{t-1}, y)$ with $(u,x_1), (x_1, x_2), \dots, (x_{t-1}, y) \in \cEp$, in expectation equals 
\begin{align}
&\mathbb{E}_{X_1,\dots, X_{t-1}} \Big[ f(\theta_u, X_1) \times \prod_{s=1}^{t-2} f(X_{s}, X_{s+1}) \times f(X_{t-1}, \theta_y)    | \theta_u, \theta_y \Big] \nonumber \\
&\qquad= \textstyle\sum_{k=1}^r \lambda_k^t q_k(\theta_u) q_k(\theta_y) \nonumber \\
&\qquad= e_u^T Q^T \Lambda^t Q e_y.
\end{align}
Therefore, the product of weights along the path connecting $u$ to $y$ is a good proxy of quantity $e_u^T Q^T \Lambda^t Q e_y$. Recall that 
each entry is observed independently with probability $p$ due to our assumed Bernoulli sampling model. Therefore, for any $u \in [n]$, the number of neighbors of $u$ in $\cG$ 
scale as $pn = n^{\kappa}$. More generally, for $1\leq t \leq 1/\kappa$, the number of nodes at distance $t$ from $u$ scale as $n^{ \kappa t}$. 
We choose $t$ large enough to guarantee that for any two nodes $u$ and $v$, there is a sufficient overlap between the two subset of nodes at distance $y$ from nodes $u$ and $v$ respectively. This suggests that we choose $t$ so that $n^{\kappa t} \approx n^{\frac12}$, which in effect aggregates enough data in the sparse regime to match the expected number of observations per row in the dense regime. We formalize this intuition in the following construction of the distance estimates. 

Let $\cS_{u,s}$ denote the set of vertices which are at distance $s$ from vertex $u$ in the graph defined by edge set $\cEp$. Specifically, $i \in \cS_{u,s}$ if the shortest path in $\cG = ([n], \cEp)$ from $u$ to $i$ has a length of $s$.  Let $\cT_u$ denote a breadth-first tree in $\cG$ rooted at vertex $u$. The breadth-first property ensures that the length of the path from $u$ to $i$ within $\cT_u$ is equal to the length of the shortest path from $u$ to $i$ in $\cG$. Let $\cT_u^t \subset \cT_u$ denote the sub-tree containing all
nodes and edges in $\cT_u$ up to and including depth $t$. If there is more than one valid breadth-first tree rooted at $u$, choose one uniformly at random. Let $N_{u,t} \in [0,1]^n$ denote the following vector with support on the boundary of the depth-$t$ neighborhood of vertex $u$ (we also call $N_{u,t}$ the neighborhood boundary):
\begin{align*}
N_{u,t}(i) = \begin{cases} \textstyle\prod_{(a,b) \in \text{path}_{\cT_u}(u,i)} M^\prime(a,b) &\mbox{if } i \in \cS_{u,t}, \\
0 & \mbox{if } i \notin S_{u,t}, \end{cases}
\end{align*}
where $\text{path}_{\cT_u}(u,i)$ denotes the set of edges along the path from $u$ to $i$ in the tree $\cT_u$. The sparsity of $N_{u,t}(i)$ is equal to $|\cS_{u,t}|$, and the value of the coordinate $N_{u,t}(i)$ is equal to the product of weights along the path from $u$ to $i$. Let $\tN_{u,t}$ denote the normalized neighborhood boundary such that $\tN_{u,t} = N_{u,t} / |\cS_{u,t}|$. 
For each tuple $(u,v) \in [n]^2$, compute $\hat{d}(u,v)$ according to
\begin{align}
\hat{d}(u,v) & = \big(\tfrac{1}{\pp}\big) \big(\tN_{u,t} - \tN_{v,t}\big)^T (\Mpp + \Mpind) \big(\tN_{u,t + 1} - \tN_{v,t + 1}\big). \label{eq:dist1}
\end{align}

\medskip
\noindent
{\em Sparser Regime.} Consider the even sparser regime where $p =  n^{-1} \ln^{1 +\kappa} n$ for some $\kappa > 0$. Let us assume that the algorithm knows the eigenvalues $\{\lambda_k\}_{k \in [r]}$. Recall that $r' \leq r$ denotes the number of distinct valued eigenvalues amongst 
$\{\lambda_k\}_{k\in [r]}$. Recall that $\Lambda$ is the diagonal matrix with $\Lambda_{kk} = \lambda_k$, and $\tLambda$ is the $r \times r'$ Vandermonde matrix where $\tLambda(a,b) = \lambda_a^{b-1}$. Let $z \in \Reals^{r'}$ be the vector that satisfies $\Lambda^{2 t + 2} \tLambda z = \Lambda^2 \bOne$; $z$ always exists and is unique because $\tLambda$ is a Vandermonde matrix, and $\Lambda^{-2t} \bOne$ lies within the span of its columns. For every $(u, v) \in [n]^2$, compute distance according to
\begin{align}\label{eq:dist2}
\hat{d}(u,v) & = \big(\tfrac{1}{\pp}\big) \textstyle\sum_{\ell \in [r']} z_{\ell} \big(\tN_{u,t} - \tN_{v,t}\big)^T (\Mpp + \Mpind) \big(\tN_{u,t + \ell} - \tN_{v,t + \ell}\big).
\end{align}

\subsection{Reducing computation by subsampling vertices} \label{sec:alg_subsampling}


The pairwise distances can only be estimated up to a limited precision depending on the sparsity of the data and amount of noise in the observations, and furthermore we tune the nearest neighbor threshold to tradeoff between bias and variance. As a result, the performance of the algorithm can be maintained with reduced computation by clustering the coordinates so that not all $n^2$ pairwise distances need to be computed. This would involve adding an extra step at the beginning of the algorithm that samples sufficiently many ``anchor'' vertices $\cK \subset [n]$ that cover the space well. $|\cK|$ should be chosen large enough such that for any vertex $u \in [n]$, there exists some anchor vertex $i \in \cK$ which is ``close'' to $u$ in the sense that $\|\Lambda Q (e_u - e_i)\|_2^2$ is small. For all $n$ vertices, we only compute the distances to each of the $|\cK|$ anchor vertices, and we let $\pi: [n] \to \cK$ be a mapping from each vertex to the anchor vertex that minimizes the estimated distance $\hat{d}$ as computed in the original algorithm statement, $\pi(u) = \argmin_{i \in \cK} \hat{d}(u,i)$. The final estimate then is given by 
\[\hat{F}(u,v) = \hat{F}(\pi(u), \pi(v)) = \frac{1}{|\cE_{\pi(u)\pi(v)}|} \sum_{(a,b) \in \cE_{\pi(u)\pi(v)}} \Mppp(a,b),\]
where $\cE_{\pi(u)\pi(v)}$ denotes the set of undirected edges $(a,b)$ such that $(a,b) \in \cE_3$ and both $\hat{d}(\pi(u),a)$ and $\hat{d}(\pi(v),b)$ are less than some threshold $\eta$. We can compute $\cE_{\pi(u) \pi(v)}$ by the clustering assignments and distances of all vertices to the anchor vertices.

\subsection{Computational Complexity}


To analyze the computational complexity of the algorithm, we consider each step. Growing local neighborhoods around each vertex costs at most $n |\cE|$, since there are $n$ vertices and the BFS trees visit each edge at most once. Computing the inner product for all pairs of vertices given the local neighborhood vectors costs at most $n^2 |\cE|$, since there are $n^2$ vertex pairs and $|\cE|$ entries in the data matrix $M$. The final nearest neighbor estimator involves a (weighted) average of the datapoints, which costs at most $n^2 |\cE|$, as there are $n^2$ entries in the matrix to estimate, and at worst the estimate would involve averaging over $|\cE|$ datapoints. This extremely crude bound leads to a computational complexity of $O(pn^4)$. The bottleneck of the algorithm is the final nearest neighbor estimate, which may be reduced by using approximate nearest neighbor methods.

If we instead used the modified algorithm that subsamples $|\cK|$ anchor vertices at random and treats them as ``cluster centers'', there are only $(|\cK|^2 + n|\cK|)$ pairwise distances computed, for a computational cost of $(|\cK|^2 + n|\cK|) |\cE|$ instead of $n^2 |\cE|$. Once we cluster the vertices, the final estimate is only computed for the pairwise cluster blocks, as the final estimate is a block constant matrix with only $|\cK|^2$ distinct valued estimates. This results in $|\cK|^2 |\cE|$ computation for the final step of the estimation. The computational complexity reduces from $O(n^2 |\cE|)$ to $O((|\cK|^2 + n|\cK|) |\cE|)$. The choice of $|\cK|$ depends on the distribution of latent variables, the shape of the latent function, and the error tolerance. In a setting with finitely many latent types, then $|\cK|$ would be roughly linear in the number of latent types.

A practical benefit of our algorithm is that it is amenable to a distributed and parallelized implementation. The key computational step of our algorithm involves comparing the expanded local neighborhoods of pairs of vertices to find the ``nearest neighbors''. As the algorithm is inherently local with respect to the data graph, it can be easily implemented for large scale datasets where the data may be stored in a distributed fashion optimized for local graph computations. The local neighborhoods can be computed in parallel, as they are independent computations. Using approximate nearest neighbor techniques and subsampling vertices to cluster will additionally reduce the computation.


\subsection{Discussion}


In practice, we may not know the model parameters, and we would use cross validation to tune the BFS tree depth $t$ and nearest neighbor threshold $\eta$. If the depth $t$ is either too small or too large, then the vector $N_{u,t}$ will be too sparse, and will not optimally aggregate the datapoints. The threshold $\eta$ trades off between bias and variance of the final estimate. When the sampled observations are not uniform across entries, the algorithm may require more modifications to properly normalize for high degree hub vertices, as the optimal choice of depth $t$ may differ depending on the local sparsity.

In our algorithm, we assumed that we observed the edge set $\cE$. Specifically, this means that we are able to distinguish between entries of the matrix that have value zero because they are not observed, i.e. $(i,j) \notin \cE$, or if the entry was observed to be value zero, i.e. $(i,j) \in \cE$ and $M(i,j) = Z(i,j) = 0$. This fits well for applications such as recommendations, where the system does know the information of which entries are observed or not. Some social network applications contain this information (e.g. facebook would know if they have recommended a link which was then ignored) but other network information may lack this information, e.g. we do not know if link does not exist because observations are sparse, or because observations are dense but the probability of an edge is small. The absence of this knowledge would primarily affect the normalization of the neighborhood vectors as well as the normalization in the final averaging step.

The idea of comparing vertices by looking at larger radius neighborhoods was introduced in \cite{AbbeSandon15a}, and has connections to belief propagation \cite{DecelleKrzakalaMooreZdeborova11, AbbeSandon16} and the non-backtracking operator \cite{KrzakalaMooreMosselNeemanSlyZdeborovaZhang13, SaadKrzakalaZdeborova14, MosselNeemanSly14, Massoulie13, BordenaveLelargeMassoulie15}. The non-backtracking operator was introduced to overcome the issue of sparsity. For sparse graphs, vertices with high-degree dominate the spectrum, such that the informative components of the spectrum get hidden  behind the high degree vertices. The non-backtracking operator avoids paths that immediately return to the previously visited vertex in a similar manner as belief propagation, and its spectrum has been shown to be more well-behaved, perhaps adjusting for the high degree vertices, which get visited very often by paths in the graph. In our algorithm, the neighborhood paths are defined by first selecting a rooted tree at each vertex, thus enforcing that each vertex along a path in the tree is unique. This is important in our analysis, as it guarantees that the distribution of vertices at the boundary of each subsequent depth of the neighborhood is unbiased, since the sampled vertices are freshly visited.
\section{Results}
\label{sec:results}

In all of the results below, we assume the latent variable model assumptions laid out in Section \ref{sec:model}. As a reminder, we assume uniform Bernoulli sampling with density $p$, independent bounded observation noise, and a generative latent variable model where coordinates are associated to i.i.d. sampled latent variables and the underlying matrix behaves according to a bounded latent function $f$ that is Lipschitz and low rank (or approximately low rank) with bounded eigenfunctions.

\subsection{$f$ has rank $r$} We first provide theoretical bounds for the estimation error in both sparse regimes mentioned above when $f$ has finite spectrum with rank $r$.  

\medskip 
\noindent
{\em Sparse Regime.} Theorem \ref{lemma:case1_assumption} shows that the maximum entrywise error of the collaborative filtering algorithm using distance function \eqref{eq:dist1} converges to zero in the sparse regime when $p = n^{-1 + \kappa}$ for some $\kappa \in (0,\frac12)$. 
\begin{theorem} \label{lemma:case1_assumption}
Let $f$ have rank $r$, $p = n^{-1 + \kappa}$ for some $\kappa \in (0,\frac12)$ so that $1/\kappa$ is not an integer. 
Consider the estimates produced by the nearest neighbor algorithm using the distance defined in \eqref{eq:dist1} for $t = \lfloor\frac{\ln(1/p)}{\ln(np)}\rfloor$ and selecting the nearest neighbor distance threshold to satisfy $\eta = \Theta(n^{ - \frac12 (\kappa - \rho)})$ for any 
$\rho \in (0, \kappa)$. Let $C_f = |\lambda_1|/|\lambda_r|$ denote the condition number of the latent function $f$. With probability $1-o(1)$, 
\begin{align}
\|\hat{F} - F\|_{\max} & = O\Big(r C_f^{1/\kappa} n^{ - \frac14 (\kappa - \rho)}\Big).
\end{align}
Furthermore, 
\begin{align}
\MSE(\hat{F}) & = \frac{1}{n^2} \|\hat{F} - F\|_{Fr}^2 = O\Big(r^2 C_f^{2/\kappa} n^{ - \frac12 (\kappa - \rho)}\Big).
\end{align}
\end{theorem}

\medskip 
\noindent
\medskip
\noindent
{\em Sparser Regime.} Theorem \ref{lemma:case2_assumption} shows that the maximum entrywise error of the collaborative filtering algorithm using distance function \eqref{eq:dist2} converges to zero in the sparser regime when $p =  n^{-1} \ln^{1 +\kappa} n$ for some $\kappa > 0$. 
\begin{theorem} \label{lemma:case2_assumption}
Let $f$ have rank $r$, $p =  n^{-1} \ln^{1 +\kappa} n$ for some $\kappa > 0$. Consider the estimates produced by the nearest neighbor algorithm using the distance defined in \eqref{eq:dist2} for $t = \lceil \frac{\ln (0.08/p) }{\ln (0.275 np)} - r' \rceil$ and selecting the nearest neighbor distance threshold to satisfy $\eta = \Theta\Big((\ln n)^{ - \frac12 (\kappa - \rho)}\Big)$ for any 
$\rho \in (0, \kappa)$. With probability $1-o(1)$, 
\begin{align}
\|\hat{F} - F\|_{\max} & = O\Big((\ln n)^{-\frac{1}{4}(\kappa -\rho)}\Big). 
\end{align}
Furthermore, 
\begin{align}
\MSE(\hat{F}) & = O\Big((\ln n)^{-\frac{1}{2}(\kappa - \rho)} n\Big).
\end{align}
\end{theorem}

Theorems \ref{lemma:case1_assumption} and \ref{lemma:case2_assumption} show that for symmetric sparse matrix estimation, as long as the fraction of entries observed at random scale as {$\frac{\log^{1+\kappa}(n)}{n}$} for any fixed $\kappa > 0$, the estimation error of our proposed iterative variant of the classical collaborative filtering algorithm with respect to the $\max$-norm decays to $0$ as $n\to\infty$ assuming the underlying matrix of interest has {constant} rank $r$.  

\subsection{$f$ has $\varepsilon$-approximate rank $r$}


We extend the above stated result to the setting when the latent function $f$ has $\varepsilon$-approximate rank $r$; this captures settings where $f$ may have infinite but quickly decaying spectrum. We formally state the extension in the sparse regime ($p = n^{-1+\kappa}$), but we believe that a similar result is likely to hold for the sparser regime ($p = n^{-1} \log^{1+\kappa}(n)$) as well, which we omit for simplicity of presentation.

\begin{theorem} \label{lemma:case3_assumption}
Let $f$ have $\varepsilon$-approximate rank $r$ for some $\varepsilon > 0$,  
$p = n^{-1 + \kappa}$ for some $\kappa \in (0,\frac12)$ so that $1/\kappa$ is not an integer. 
Consider the estimates produced by the nearest neighbor algorithm using the distance defined in \eqref{eq:dist1} for 
$t = \lfloor\frac{\ln(1/p)}{\ln(np)}\rfloor < \frac{1}{\kappa} - 1$ and selecting the nearest neighbor distance threshold to satisfy 
$\eta = \Theta(n^{ - \frac12 (\kappa - \rho)})$ for any $\rho \in (0, \kappa)$. Let $C_{f,r} = |\lambda_1|/|\lambda_r|$ denote the condition number of the rank $r$ approximation to the latent function $f$. With probability $1-o(1)$, 
\begin{align}\label{eq:case3.max}
\|\hat{F} - F\|_{\max} & = O\Big(r C_{f,r}^{1/\kappa} 
 n^{ - \frac14 (\kappa - \rho)}\Big)  
+  
O\Big(|\lambda_r|^{-\frac{1}{\kappa}} \sqrt{r} \left(\sqrt{\frac{\eps}{\kappa}} (1+\eps)^{\frac{1}{2\kappa} - \frac12} + \frac{\eps}{\kappa} (1+\eps)^{\frac{1}{\kappa}-\frac32}\right)\Big)
\end{align}
Furthermore, 
\begin{align}\label{eq:case3.mse}
\MSE(\hat{F}) & = O\Big(r^2 C_{f,r}^{\frac{2}{\kappa}}
 n^{ - \frac12 (\kappa - \rho)}\Big) + O\Big(|\lambda_r|^{-\frac{2}{\kappa}} r \left(\frac{\eps}{\kappa} (1+\eps)^{\frac{1}{\kappa} - 1} + \frac{\eps^2}{\kappa^2} (1+\eps)^{\frac{2}{\kappa}-3}\right)\Big).
\end{align}
\end{theorem}

As we assume the function values are bounded in $[0,1]$, we can assume that $\eps \in [0,1]$, such that the dominating terms in \eqref{eq:case3.max} are $O\Big(r C_{f,r}^{1/\kappa} n^{ - \frac14 (\kappa - \rho)}\Big) + O\Big(\sqrt{\eps r |\lambda_r|^{-\frac{2}{\kappa}} \kappa^{-1}}\Big)$, and the dominating terms in \eqref{eq:case3.mse} are $O\Big(r^2 C_{f,r}^{2/\kappa} n^{ - \frac12 (\kappa - \rho)}\Big) + O\Big(\eps r |\lambda_r|^{-\frac{2}{\kappa}} \kappa^{-1}\Big)$.
While we assume the function $f$ is fixed with respect to $n$, when the function $f$ has infinite spectrum, we can choose $\varepsilon$ to decrease with $n$ in order to tradeoff between the two terms in the error bound. Note that the approximate rank $r$ and the approximate condition number $C_{f,r}$ also depend on the choice of $\varepsilon$. In particular the relationship between $\varepsilon$, $r$, and $C_{f,r}$ will depend on the spectrum of $f$ and how quickly the tail decays to zero. Choosing a larger value of $r$ will increase the condition number as $|\lambda_r|$ will be smaller, and it will decrease the approximation error $\varepsilon$. 
Below we present a specific example as a consequence of Theorem \ref{lemma:case3_assumption}. 

\begin{corollary}\label{cor:case3}
Let $p = n^{-1 + \kappa}$ for some $\kappa \in (0,\frac12)$ so that $1/\kappa$ is not an integer. 
Consider $f$ such that for any $r \geq 1$, it has $\eps_r$-approximate rank $r$ with
$|\lambda_r|$ corresponding to rank $r$ approximation with $C_{f, r} = |\lambda_1|/|\lambda_r|$ being the condition 
number such that $|\lambda_1| = O(1)$ and 
\begin{align}\label{eq:cond.aprx}
\lim_{r\to\infty} \eps_r |\lambda_r|^{-2/\kappa} r & = 0.
\end{align}
Then, for any $\delta > 0$, for all $n$ large enough, with probability $1-o(1)$, $\|\hat{F} - F\|_{\max}  = O\Big(\sqrt{\delta}\Big)$.
Further,  $\MSE(\hat{F}) = O\Big(\delta \Big)$. 
\end{corollary}
\paragraph{Proof of Corollary \ref{cor:case3}.} 
For any $\delta > 0$, by \eqref{eq:cond.aprx}, there exists large enough $r = r(\delta)$ such that 
$\eps_r |\lambda_r|^{-2/\kappa} r \leq \delta$.  Due to $|\lambda_1| = O(1)$, $C_{f, r} = O(|\lambda_r|^{-1})$. 
Given choice of $r = r(\delta)$, for $n$ large enough we have 
$ r C_{f,r}^{1/\kappa} n^{ - \frac14 (\kappa - \rho)} \leq \sqrt{\delta}$. By \eqref{eq:case3.max}
of Theorem \ref{lemma:case3_assumption} it follows that $\|\hat{F} - F\|_{\max}  = O(\sqrt{\delta})$ with 
probability at least $1-o(1)$. By \eqref{eq:case3.mse} of Theorem \ref{lemma:case3_assumption}, 
it follows that $\MSE(\hat{F}) = O(\delta)$.
 \hfill \Halmos
 
From Corollary \ref{cor:case3}, it follows that $\|\hat{F} - F\|_{\max}=o(1)$ with probability $1-o(1)$ and $\MSE(\hat{F}) = o(1)$ when the spectrum decays in such a way that $\lim_{r\to\infty} \eps_r |\lambda_r|^{-2/\kappa} r = 0.$

\subsection{Discussion}


In our latent variable model, the latent function $f$ is fixed with respect to $n$, so the max norm of the truth matrix is constant $\|F\|_{\max} = \Theta(1)$, and the Frobenius norm of the truth matrix scales linear with the matrix dimension so that $\frac{1}{n^2} \|F\|_{Fr}^2 = \Theta(1)$. As a result the above stated results also show the convergence rates with respect to the relative errors of the max norm and normalized Frobenius norm. 

The overall proof sketch can be split into two parts. First we prove that the estimated pairwise distances concentrate to a metric computed with respect to the true latent function $f$. Second we prove that given well behaved estimated distances, the nearest neighbor estimate with properly chosen thresholds to balance mean and variance will converge at the above stated rate. This second part of the proof is straightforward and follows the standard proof for any nearest neighbor style algorithm. The crux of the proof is arguing that in sparse settings the computed distances concentrate well. This relies on the uniform sampling assumption, independence of the observation noise, regularity of the latent feature variables, and the finite spectrum assumption of the latent function. The assumptions on the specific distribution of the latent variables and the Lipschitzness of the latent function are in fact primarily used for the second nearest neighbor portion of the proof, and thus can be relaxed. The key property needed is that there are sufficiently many ``nearest neighbor'' coordinates; the precise distribution of the latent variables and shape of the latent function will affect the tuning of the threshold parameter to tradeoff between bias and variance. We provide formal statements for a few variations of the model in Section \ref{sec:thm_extensions}.

In addition to providing bounds on the MSE, our theorem also provides bounds on the maximum entrywise error of the estimate. The rate of our maximum entrywise error is the square root of the MSE rate, which suggests that the error is uniformly spread across all entries. This is a stronger guarantee that the typical MSE bounds found in the literature, and it can be useful for downstream results that use the estimates for decision making such as ranking and recommendations. 


Thus far, we have focused on finding conditions on $p$ that allow for consistent estimation with respect to both the MSE and max entrywise error. Our results also provide the rate at which the error decays. Specifically, our bound for the mean squared error (MSE) scales as $O((pn)^{-1/2 + \rho})$ for any arbitrarily small constant $\rho > 0$,  and our bound for the max entrywise error is $O((pn)^{-1/4 + \rho})$ for any small $\rho$.  



\section{Proof Sketch for Analyzing Noisy Nearest Neighbors}
\label{sec:proofsketch}

As the algorithm uses a fixed radius nearest neighbor estimate, the analysis boils down to arguing that the distance functions as defined in \eqref{eq:dist1} and \eqref{eq:dist2} have certain desired properties that enable the classical nearest neighbor algorithm to be effective. In this section we characterize the needed properties for the convergence of noisy nearest neighbors.

Our algorithm estimates $F(u, v)$, i.e. $f(\theta_u, \theta_v)$, according to \eqref{eq:estimate}, which simply averages over datapoints $M(u', v')$ corresponding to tuples $(u',v')$ for which $u'$ is close to $u$ and $v'$ is close to $v$ according to the estimated distance function $\hat{d}$. This simple nearest neighbor averaging estimator suggests that the last step of the analysis involves choosing the threshold $\eta$ to tradeoff between bias and variance.

The primary desired property is that the data-driven distance estimates $\hat{d}(u,v)$ concentrate around some ideal data-independent distance $d(\theta_u,\theta_v)$ for $d: [0,1]^2 \to \mathbb{R}_+$. We can then subsequently argue that the nearest neighbor estimate produced by \eqref{eq:estimate} using $d(\theta_u,\theta_v)$ in place of $\hat{d}(u,v)$ will yield a good estimate by properly choosing the threshold $\eta$ to tradeoff between bias and variance. The bias will depend on the local geometry of the function $f$ relative to the distances defined by $d$. The variance depends on the measure of the latent variables $\{\theta_u\}_{u \in [n]}$ relative to the distances defined by $d$, i.e. the number of observed tuples $(u',v') \in \cEppp$ such that $d(\theta_u,\theta_{u'}) \leq \eta$ and $d(\theta_v,\theta_{v'}) \leq \eta$ needs to be sufficiently large. We formalize the above stated desired properties. 

\begin{properties}[Good Distance]\label{ass:good_distances.1}
We call an ideal distance function $d: [0,1]^2 \to \RealsP$ to be a $\bias$-good distance function for some $\bias: \RealsP \to \RealsP$ if for any given $\eta > 0$ it follows that 
$|f(\theta_a, \theta_b) - f(\theta_u, \theta_v)| \leq \bias(\eta)$
for all $(\theta_a,\theta_b,\theta_u,\theta_v) \in [0,1]^4$ such that $d(\theta_u,\theta_a) \leq \eta$ and $d(\theta_v,\theta_b) \leq \eta$.
\end{properties}

Property \ref{ass:good_distances.1} follows from choosing an appropriate ideal distance function $d$. In particular we will choose $d$ with respect to the spectral representation of $f$, and the desired property and the expression for $\bias(\eta)$ will follow from the low rank assumption as well as the boundedness of the eigenfunctions.

\begin{properties}[Good Distance Estimation]\label{ass:good_distances.2}
For some $\Delta > 0$, we call distance $\hat{d}: [n]^2 \to \RealsP$ a $\Delta$-good estimate for ideal distance $d: [0,1]^2 \to \RealsP$, if $|d(\theta_u,\theta_a) - \hat{d}(u,a)| \leq \Delta$ for all $(u, a) \in [n]^2$.
\end{properties}

Showing property \ref{ass:good_distances.2} is the crux of the proof and follows from the design of the algorithm along with the assumptions of uniform sampling and the latent variable model. It essentially uses all the model assumptions except for Lipschitzness of $f$.

\begin{properties}[Sufficient Representation]\label{ass:good_distances.3}
The collection of coordinate latent variables $\{\theta_u\}_{u \in [n]}$ is called $\meas$-represented for some
$\meas: \RealsP \to \RealsP$ if for any $u \in [n]$ and $\eta' > 0$, 
$\frac{1}{n} \sum_{a \in [n]} \Ind(d(u, a) \leq \eta') \geq \meas(\eta')$. 
\end{properties}

Property \ref{ass:good_distances.3} is only used for the final step of the nearest neighbor analysis. In particular, as the estimate averages datapoints within an estimated nearby region of the target coordinates, there is a bias variance tradeoff that depends on how the datapoints are locally distributed. In particular, we need to guarantee that for any $(a,b) \in [n]^2$, there exists sufficiently many observed pairs $(u,v) \in [n]^2$ such that the function behaves similarly, i.e. $f(a,b)$ is close to $f(u,v)$. This property follows from our assumption that the latent variables are sampled i.i.d. from $U[0,1]$, and that the function $f$ is $L$-Lipschitz. As discussed in section \ref{sec:model}, these assumptions can be relaxed, but alternative assumptions would need to guarantee property \ref{ass:good_distances.3} for some reasonable local measure function $\meas(\eta)$.

Given the above three properties, we can then prove Lemma \ref{lemma:nearest_neighbor}, which characterizes the error of the noisy nearest neighbor algorithm as a function of the $\bias$ function, $\meas$ function, and estimation error $\Delta$. Section \ref{sec:thm_proofs} uses Lemma \ref{lemma:nearest_neighbor} to establish Theorems \ref{lemma:case1_assumption}, \ref{lemma:case2_assumption}, and \ref{lemma:case3_assumption} by simply showing the three properties for suitable choices of $\bias, \meas,$ and $\Delta$, and tuning $\eta$ accordingly to balance between different terms of the error. Proving that the distance estimates concentrate well, i.e. property \ref{ass:good_distances.2}, is the most involved part of the analysis, which we defer to sections \ref{sec:distance_estimates_proof} and \ref{sec:distance_estimate_approx_proof}. Property \ref{ass:good_distances.1} follows from the low rank assumption and property \ref{ass:good_distances.3} arises from the latent variable model assumptions, in particular the distribution of the latent variables and shape of the latent function.

\begin{lemma} \label{lemma:nearest_neighbor}
Assume that properties \ref{ass:good_distances.1}-\ref{ass:good_distances.3} hold with probability $1-\alpha$ for some $\eta, \Delta,$ and $\eta' = \eta - \Delta$; in particular $d$ is a $\bias$-good distance function, $\hat{d}$ as estimated from $\Mp$ and $\Mpp$ is a $\Delta$-good distance estimate for $d$, and $\{\theta_u\}_{u \in [n]}$ is $\meas$-represented. The noisy nearest neighbor estimate $\hat{F}$ computed according to \eqref{eq:estimate} satisfies 
\begin{align*}
\MSE(\hat{F}) &\leq \bias^2(\eta + \Delta) + \frac{2 }{(1 - \delta) p \left(\meas(\eta-\Delta) n\right)^2 } + \exp\left(-\frac{\delta^2 p \left(\meas(\eta-\Delta) n\right)^2}{4}\right) + \alpha,
\end{align*}
for any $\delta \in (0,1)$. Furthermore, for any $\delta' \in (0,1)$, 
\[\max_{(u,v)\in[n]^2} |\hat{F}(u, v) - f(\theta_u,\theta_v)| \leq \bias(\eta + \Delta) + \delta',\]
with probability at least 
\[1 - n^2 \exp\left(-\tfrac{1}{4}\delta^2 p \left(\meas(\eta-\Delta) n\right)^2\right)
- n^2 \exp\left(-\delta'^2 (1 - \delta) p \left(\meas(\eta-\Delta) n\right)^2\right) - \alpha.\]
\end{lemma}

\paragraph{Proof of Lemma \ref{lemma:nearest_neighbor}.}
Recall that the algorithm uses sample splitting, where $\hat{d}$ is computed using $\Mp$ and $\Mpp$, and the final estimate $\hat{F}$ is computed using $\Mppp$. Therefore, for some $(a,b) \in \cEppp$, the observation $M(a, b) = Z(a, b)$ is independent of $\hat{d}$, and $\E[M(a, b)] = f(\theta_a, \theta_b)$. 
Conditioned on $\cEppp$, by definition of $\hat{F}$ and by assuming properties \ref{ass:good_distances.1} and \ref{ass:good_distances.2}, it follows that
\begin{align*}
\E[(\hat{F}(u, v) - f(\theta_u,\theta_v))^2] 
&= \left(\frac{1}{|\cEppp_{uv}|}\sum_{(a,b) \in \cEppp_{uv}} f(\theta_a,\theta_b) - f(\theta_u,\theta_v) \right)^2 \\
&\qquad + \frac{1}{|\cEppp_{uv}|^2} \sum_{(a,b) \in \cEppp_{uv}} \Var[M(a, b)] \\
&\stackrel{(a)}{\leq} \text{bias}^2(\eta + \Delta) + \frac{1}{|\cEppp_{uv}|}.
\end{align*}
Inequality $(a)$ follows from Properties \ref{ass:good_distances.1}-\ref{ass:good_distances.2}: $|d(u,a) - \hat{d}(u, a)| \leq \Delta$ and 
$\hat{d}(u, a) \leq \eta \implies d(u, a) \leq \eta + \Delta$. By definition $M(a, b) \in [0,1]$ for all $(a, b)$, which implies $\Var[M(a, b)] \leq 1$ for all $(a, b) \in \cEppp$. 
Define ${\cV}_{uv} = \{ (a,b) \in [n]^2 ~:  {d}(u,a) < \eta -\Delta, ~{d}(v, b) < \eta-\Delta\}$.
Assuming property \ref{ass:good_distances.3},
\begin{align*}
|\cV_{uv}| & = |\{a \in [n]: {d}(u,a) < \eta-\Delta\}| ~|\{b \in [n]: {d}(v, b) < \eta-\Delta\}| \\
&\geq \left(\meas(\eta-\Delta) n\right)^2.
\end{align*}
By the Bernoulli sampling model and sample splitting process, each tuple $(a, b) \in [n]^2$ belongs to $\cEppp$ with probability $p/2$ independently. By a straightforward application of Chernoff's bound, it follows that for any $\delta \in (0,1)$, 
\begin{align}
\Prob{|\cEppp \cap \cV_{uv}| \leq \frac{(1 - \delta) p}{2} \left(\meas(\eta-\Delta) n\right)^2} \leq \exp\left(- \frac{\delta^2 p \left(\meas(\eta-\Delta) n\right)^2 }{4}\right).
\label{eq:omega3_V}
\end{align}
Therefore, by assuming property \ref{ass:good_distances.2}, it follows that with probability at least $1 - \exp\left(- \frac{\delta^2 p \left(\meas(\eta-\Delta) n\right)^2 }{4}\right)$,
\begin{align*}
|\cEppp_{uv}|  & = |\{ (a,b) \in \cEppp ~: \hat{d}(u,a) < \eta, ~\hat{d}(v, b) < \eta\}| \\
& \geq |\{ (a,b) \in \cEppp ~: {d}(u,a) < \eta-\Delta, ~{d}(v, b) < \eta-\Delta\}| \\
& = | \cEppp \cap \cV_{uv}| \\
&\geq \frac{(1 - \delta) p}{2} \left(\meas(\eta-\Delta) n\right)^2.
\end{align*}

Define the event $\cH = \{  |\cEppp_{uv}|  \geq \frac{(1 - \delta) p}{2} \left(\meas(\eta-\Delta) n\right)^2| \}$. It follows that 
$\Prob{\cH^c} \leq \exp\left(-\frac{1}{4}\delta^2 p \left(\meas(\eta-\Delta) n\right)^2\right)$. By definition, $F(u,v) = f(\theta_u, \theta_v) \in [0,1]$ 
for all $u, v \in [n]$. Therefore, assuming properties \ref{ass:good_distances.1}-\ref{ass:good_distances.3} hold,
\begin{align*}
&\E[(\hat{F}(u,v) - f(\theta_u,\theta_v))^2] \\
&\leq \E[(\hat{F}(u,v) - f(\theta_u,\theta_v))^2 ~\Big|~\cH ] + \Prob{\cH^c} \\
&\leq \bias^2(\eta + \Delta) + \frac{2}{(1 - \delta) p \left(\meas(\eta-\Delta) n\right)^2 } + \exp\left(-\frac{1}{4}\delta^2 p \left(\meas(\eta-\Delta) n\right)^2\right).
\end{align*}
We add an additional $\alpha$ in the final MSE bound to account for the probability that properties \ref{ass:good_distances.1}-\ref{ass:good_distances.3} are violated.

To obtain the high-probability bound on the maximum entry-wise error, 
note that $M(a, b)$ are independent across indices $(a, b) \in \cEppp$ as well as independent of observations in $\cEp \cup \cEpp$. 
Additionally, the model assumes that $M(a, b), F(a, b) \in [0,1]$, and $\E[M(a, b)] = F(a, b)$ for observed tuples $(a,b)$. By an application of Hoeffding's inequality for bounded, zero-mean independent variables, for any $\delta' \in (0,1)$ it follows that assuming properties \ref{ass:good_distances.1}-\ref{ass:good_distances.3} hold,
\begin{align*}
\Prob{\left. \tfrac{\Big|\sum_{(a,b) \in \cEppp_{uv}} (M(a, b) - F(a,b)) \Big|}{|\cEppp_{uv}|}\geq \delta' ~\right|~ \cH} 
&\leq \exp\left(-\delta'^2 (1 - \delta) p \left(\meas(\eta-\Delta) n\right)^2\right).
\end{align*}
By union bound it follows that 
\[\max_{(u,v)\in[n]^2} |\hat{F}_{uv} - f(\theta_u,\theta_v)| \leq \bias(\eta + \Delta) + \delta', \]
with probability at least 
\[1 - n^2 \exp\left(-\frac{1}{4}\delta^2 p \left(\meas(\eta-\Delta) n\right)^2\right)
- n^2 \exp\left(-\delta'^2 (1 - \delta) p \left(\meas(\eta-\Delta) n\right)^2\right) - \alpha.\]
This completes the proof of Lemma \ref{lemma:nearest_neighbor}.
 \hfill \Halmos

\section{Extensions} \label{sec:thm_extensions}


\subsection{Subsampled Anchor Vertices}

As mentioned in Section \ref{sec:alg_subsampling}, we can reduce the computational complexity of the algorithm by subsampling a set of anchor vertices $\cK$ and only computing pairwise distances relative to the anchor vertices, equivalent to computing a clustering amongst vertices and using that to estimate. For pairs of anchor vertices $(a,b) \in \cK^2$ which we also refer to as cluster centers, the algorithm estimates $\hat{F}(a,b)$ according to the original stated algorithm with no modifications. For $u \notin \cK$, we denote $\pi(u) = \argmin_{i \in \cK} \hat{d}(u,i)$ to be a clustering that maps from $u$ to the closest anchor vertex in $\cK$. The final estimate for $(u,v) \notin \cK^2$ is then given by the estimate of the associated anchor vertices, which act as cluster centers, $\hat{F}(u,v) = \hat{F}(\pi(u), \pi(v))$.

The original argument provides high probability bounds on $|\hat{F}(u,v) - F(u,v)|$ for cluster centers $(u,b) \in \cK^2$, as nothing changed in the algorithm for the cluster centers. The only additional part of the proof is to bound the additional bias for non cluster centers, as $|\hat{F}(u,v) - F(u,v)| \leq |\hat{F}(\pi(u), \pi(v)) - F(\pi(u), \pi(v))| + |F(\pi(u), \pi(v)) - F(u,v)|$. The first term is directly bounded by the current analysis, and the bias from the second term will depend on the size of $|\cK|$. Recall our latent variable model assumption that each vertex $u$ is associated to a latent variable $\theta_u \sum U[0,1]$ such that $F(u,v) = f(\theta_u, \theta_v)$ and $f$ is $L$-Lipschitz with respect to the latent variables. For $|\cK| = \frac{2}{\delta} \log(\frac{1}{\delta})$, with probability at least $1 - \delta$, each interval $[(i-1)\delta, i\delta]$ for $i \in [1/\delta]$ contains at least one anchor point in $\cK$, as the latent variables of these anchor points are chosen at random. Under this good event, then $\max_{u \in [n]} \min_{i\in\cK} |\theta_u - \theta_i| \leq \delta$.

We discuss the results and analysis for the sparse setting when $p = n^{-1 + \kappa}$ for some $\kappa \in (0,\frac12)$, however a similar argument applies for the sparser setting of $p =  n^{-1} \ln^{1 +\kappa} n$ as well. Equation \eqref{eq:ideal.dist1.a} will show that $d(\theta_u,\theta_v) \leq |\lambda_1|^{2t} L^2 |\theta_u - \theta_v|^2$, so that for some $u \in [n]$, the closest anchor point $a \in \cK$ with respect to the latent representation will also satisfy $d(\theta_u, \theta_a) \leq |\lambda_1|^{2t} L^2 \delta^2$. As Property \ref{ass:good_distances.2} guarantees $|d(\theta_u,\theta_a) - \hat{d}(u,a)| \leq \Delta$ for all estimated distances, it follows that $d(\theta_u, \theta_{\pi(u)}) \leq |\lambda_1|^{2t} L^2 \delta^2 + 2 \Delta$ for all $u \in [n]$. By Property \ref{ass:good_distances.1}, $|F(\pi(u), \pi(v)) - F(u,v)| \leq \bias(|\lambda_1|^{2t} L^2 \delta^2 + 2 \Delta)$. We choose $|\cK|$ so that $\delta = \frac{\sqrt{\Delta}}{L |\lambda_1|^{t}}$, and we plug in the choice of $\Delta$ and $t$ from Theorem \ref{lemma:case1_assumption}, resulting in $\delta = B r |\lambda_1|^{(\kappa + 1)/\kappa} L^{-1} n^{ - \frac14 (\kappa - \rho)} = o(1)$ so that $|\cK| = \Theta(n^{\frac14 (\kappa - \rho)})$. This choice of $|\cK|$ will guarantee that the extra added bias does not change the existing guarantees in Theorem \ref{lemma:case1_assumption} by more than a constant.

\subsection{Local Geometry}

We can generalize the latent variable model beyond scalar valued latent variables and Lipschitz latent functions. These assumptions only affect the function $\meas$ in Property \ref{ass:good_distances.3}, and thus it only changes the last portion of the nearest neighbor proof in which we tune the threshold $\eta$ to tradeoff between the bias and variance terms. We present two examples of extending our results to a different local geometry, illustrating the modifications for the sparse setting when $p = n^{-1 + \kappa}$ for some $\kappa \in (0,\frac12)$. 

If there were only $m$ distinct latent types such that $\theta_u \in [m]$ and $p_{\min} = \min_{i \in [m]} \Prob{\theta_u = i} > 0$, then $\meas(\eta')$ could be chosen to be a constant slightly less than $p_{\min}$ for every value of $\eta' > 0$. If the minimum distance measured by $d(\theta_u, \theta_v)$ between any two distinct types $\theta_u \neq \theta_v$ is larger than $2\Delta$, then we can choose $\eta$ to be $\Delta$ so that by Property \ref{ass:good_distances.2}, the algorithm will achieve perfect clustering. In particular, if Property \ref{ass:good_distances.2} holds then no vertex of a different type will have estimated distance less than $\eta$ and $\bias(\eta + \Delta) = 0$. Given this, each type has at least $p_{\min} n$ instances realized, on average. Therefore, for a given $u, v \in [n]$, there are roughly $(p_{\min} n)^2$ entries $(u', v') \in [n] \times [n]$
such that $u, u'$ and $v, v'$ are of the same type. Each of these $(p_{\min} n)^2$ is observed with probability $p$. Therefore, by taking average over these observed entries, the Mean Squared Error should scale as $1/(p (p_{\min} n)^2)$ and the max entry-wise error would scale as $(p (p_{\min} n)^2)^{-1/2}$. In the case that the minimum distance between any two distinct types is less than $\Delta$, then the bias term will still be there and the limiting term is still $\bias(\Delta)$, and thus the convergence rate would be limited by the same rate as stated in Theorem \ref{lemma:case1_assumption}.

Next we discuss a higher dimensional setting. Assume the latent variables are sampled uniformly over a $m$-dimensional hypercube such that $\theta_u \sim U([0,1]^m)$ and the latent function $f$ is $L$-Lipschitz with respect to an underlying metric $d_m$, such that the measure of a ball with radius $\delta$ is $\Theta(\delta^m)$. Property \ref{ass:good_distances.3} would instead hold for $\meas(\eta') = \Theta((\frac{\sqrt{\eta'}}{\lambda^t L})^m)$, resulting in a different choice of threshold $\eta$ to balance between bias and variance. If $m \leq (\kappa + 2)/\kappa$, then the current $\bias(\Delta)$ term dominates such that we would choose $\eta = \Theta(\Delta)$, and the error convergence rate will be the same as that stated in Theorem \ref{lemma:case1_assumption}. For high dimension $m > (\kappa + 2)/\kappa$, we choose the threshold $\eta = \Theta((pn^2)^{-1/(m+1)})$ such that the MSE bound will scale as $\Theta((pn^2)^{-1/(m+1)}) = \Theta(n^{-(1+\kappa)/(m+1)})$ and the max entrywise error bound will scale as $\Theta(n^{-(1+\kappa)/2(m+1)})$.

\subsection{Asymmetric Matrix}

Even though our stated results are for symmetric models, we can transform an asymmetric latent variable model to a symmetric model as long as the row and column dimensions grow proportionally to one another. Consider an $n \times m$ matrix $F$ which we would like to learn, where $F(u,v) = f(\alpha_u, \beta_v) \in [0,1]$, and $f$ has finite spectrum. We can construct a $(n+m)\times(n+m)$ matrix where $F$ is placed on the off-diagonal blocks and the diagonal $n\times n$ and $m\times m$ blocks are set to zero. We can argue that this constructed matrix is sampled form a symmetric latent model, so that we can apply our algorithm and analysis directly.

\subsection{Categorical Valued Data}

If the edge labels are categorical instead of real-valued, then the goal is instead to estimate the distribution over the different categories or labels. This is particularly suitable for a setting in which there is no obvious metric between the categories such that an aggregate statistic such as the expected label would not be meaningful. If the edge labels take values within $m$ category types, we can split the data is split into $m$ different matrices, each containing the information for a separate category (or edge label). For each category or label $\ell \in [m]$, the associated matrix $F_\ell$ represents the probability that each datapoint is labeled with $\ell$, such that $\bP(Z(u,v) = \ell) = F_\ell(u,v) = f_\ell(\alpha_u, \alpha_v)$, where $f$ is a symmetric function having finite spectrum. The algorithm can then be applied to each matrix separately to estimate the probability of each category across the different entries. Since we need the estimates across different categories for the same entry to sum to 1, we can simply let the estimate for the $m$-th category one minus the sum of the estimates for the first $m-1$ categories. To obtain an error bound, we can simply use union bound across the $m-1$ applications of the algorithm, which simply multiplies the error probability by $m-1$. 

\subsection{Non-Uniform Sampling}

We assumed a uniform sampling model, where each entry is observed independently with probability $p$. However, in reality the probability that entries are observed may not be uniform across all pairs $(i,j)$. Our results can be extended to a setting where the sampling probability is instead a function of the latent variable, i.e. entry $(i,j)$ is observed with probability $c_n g(\theta_i,\theta_j)$ where $g$ is a Lipschitz low rank function independent of $n$ and $c_n$ is a scaling factor governing the density. The observed data $M(i,j)$ would then be sampled according to
\begin{align*}
M(i,j) = \begin{cases}
0 &\text{ with probability } 1 - c_n g(\theta_i,\theta_j) \\
Z(i,j) &\text{ with probability } c_n g(\theta_i,\theta_j).
\end{cases}
\end{align*}
for $\E[Z(i,j)] = f(\theta_i,\theta_j)$. Whereas previously we had $\E[M(i,j)] = p f(\theta_i,\theta_j)$, in this modified model, $\E[M(i,j)] = c_n g(\theta_i,\theta_j) f(\theta_i,\theta_j)$. A limitation of this model is that we need the sampling probabilities to all scale at the same order with respect to $n$. The model is not fully identifiable as we could multiply $c_n$ by a constant and divide $g$ by the same constant and obtain the same data distribution, and thus we can only estimate up to a constant scaling factor.

We can essentially then apply our algorithm twice, first using data matrix $M$ to estimate the product $g(\theta_i,\theta_j) f(\theta_i,\theta_j)$ up to a scaling factor. Second we apply our algorithm to the binary adjacency matrix representing the sparsity of the observation set $\Omega$ in order to estimate $g(\theta_i,\theta_j)$ up to scaling factor. The one nuance one would have to handle is that since the set of observed entries is not uniformly sampled, the constructed BFS trees will grow non-uniformly, which will affect the normalization and scaling terms. As the model is only recoverable up to scaling, this is the best we can do. If we had data from a two-step sampling process in which we first observe binary edges sampled uniformly with probability $c_n$, and then subsequently observed datapoints sampled with an additional probability $g(\theta_i, \theta_j)$, then the model would exactly fall into our assumptions and the results could directly be applied to estimating $g(\theta_i, \theta_j)$ and the product $g(\theta_i,\theta_j) f(\theta_i,\theta_j)$.
\section{Experiments}


We show results on synthetic data to illustrate the performance of our algorithm. We did not do sample splitting as it is primarily introduced for the purpose of the analysis. We computed distances according to equation \eqref{eq:dist1} (but again without sample splitting) for fixed radius parameters of $t \in \{0,1,2,3,4\}$. Note that the depth for expanding the BFS tree is until $t+1$. We did not specifically tune the nearest neighbor threshold $\eta$, but simply chose it to be the 70th percentile amongst all estimated distances. As a result, the expected number of entries used to compute the final weighted average estimate is $0.49 p n^2$. We compare against a naive baseline which predicts using the column-wise mean. And we compare against the softimpute implementation in python's fancyimpute package and alternating least squares with rank 2 from parafac algorithm in the python tensorly package (higher rank performed more poorly in the sparse setting as it overfit to noise). Nuclear norm minimization was too slow for the size of instances that we show and thus was omitted.

The matrix $F$ is generated as follows. For rank $r = 10$, we first sample two Gaussian $n \times r$ latent factor matrices $U \in \Reals^{n \times r}$ and $V \in \Reals^{n\times r}$. Each entry of the latent factor matrices is sampled from an independent Gaussian distribution with mean 10 and standard deviation 10. Next we compute $F$ according to 
\begin{align*}
F = \frac{(UV^T - \text{mean}(UV^T))}{\max(\text{abs}(UV^T))}.
\end{align*}
For a $\kappa \in (0,1]$, the density is chosen to be $p = n^{-1 + \kappa}$, and each entry is observed (and thus included in sample set $\Omega$) with probability $p$ independently of all other entries. For each observed entry $(u,v) \in \Omega$, there is an added independent Gaussian noise $M(u,v) = F(u,v) + \eps(u,v)$, where $\eps(u,v) \sim N(0,\sigma^2)$ where $\sigma$ is chosen to be the 40th percentile of the magnitude of entries in $F$. We show results for $n = 500, 1000,$ and $5000$.

We compute an adjusted mean squared error (MSE), limited to the error in predicting missing entries, and we normalize by the squared error of predicting with zeros. When the adjusted MSE is larger than 1, it means the estimate is worse than predicting all zeros.
\begin{align*}
\text{adjusted MSE} = \frac{\sum_{(u,v) \notin \Omega} (\hat{F}(u,v)-F(u,v))^2}{\sum_{(u,v) \notin \Omega} F(u,v)^2}
\end{align*}

Figure \ref{fig:mse_p} shows the adjusted MSE of the algorithms with respect to the sampling probability $p$. When $p$ is very small, then our algorithm with the optimal choice of the depth parameter $t$ performs better than ALS and SoftImpute, however when it is too sparse than either the simple mean estimate or predicting with all zeros is best. Note that we did not do any tuning of the nearest neighbor parameter $\eta$, and thus there may be additional gains possible for our algorithm. If we consider the minimum density for which the algorithm performs better than the simple mean, SoftImpute requires the most dense observation. The minimum density required for our algorithm depends on optimally choosing the depth parameter $t$, but for an optimal choice of $t$, our algorithm requires less data than ALS before it performance better than the simple mean.

\begin{figure*}[t!]
	\centering
	\begin{subfigure}[t]{0.32\textwidth}
		\centering
		\includegraphics[height=1.7in]{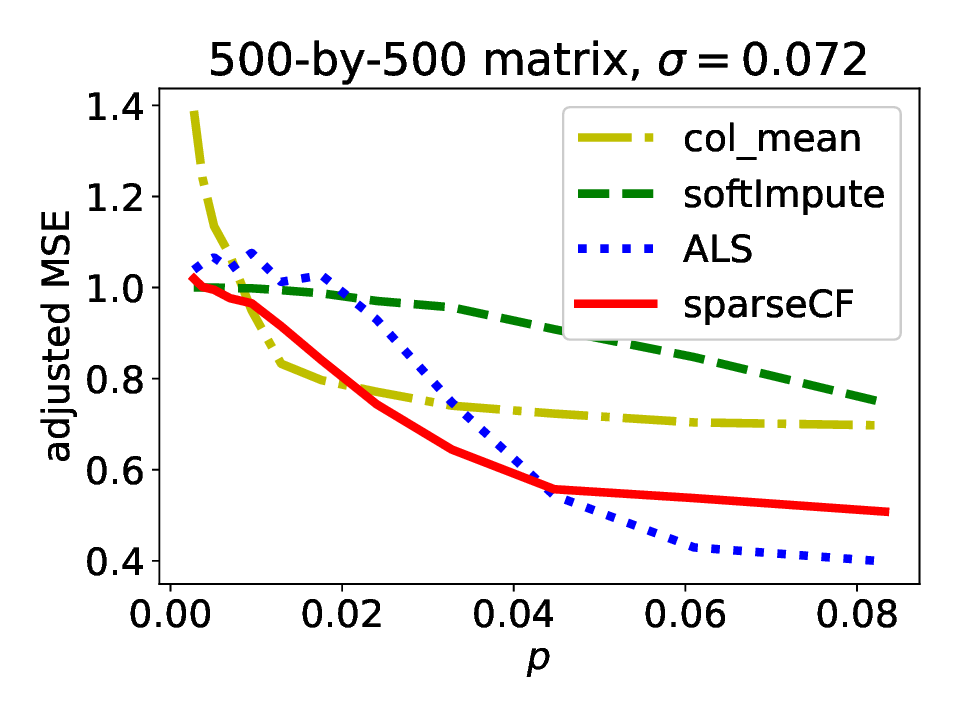}
	\end{subfigure}%
	~ 
	\begin{subfigure}[t]{0.32\textwidth}
		\centering
		\includegraphics[height=1.7in]{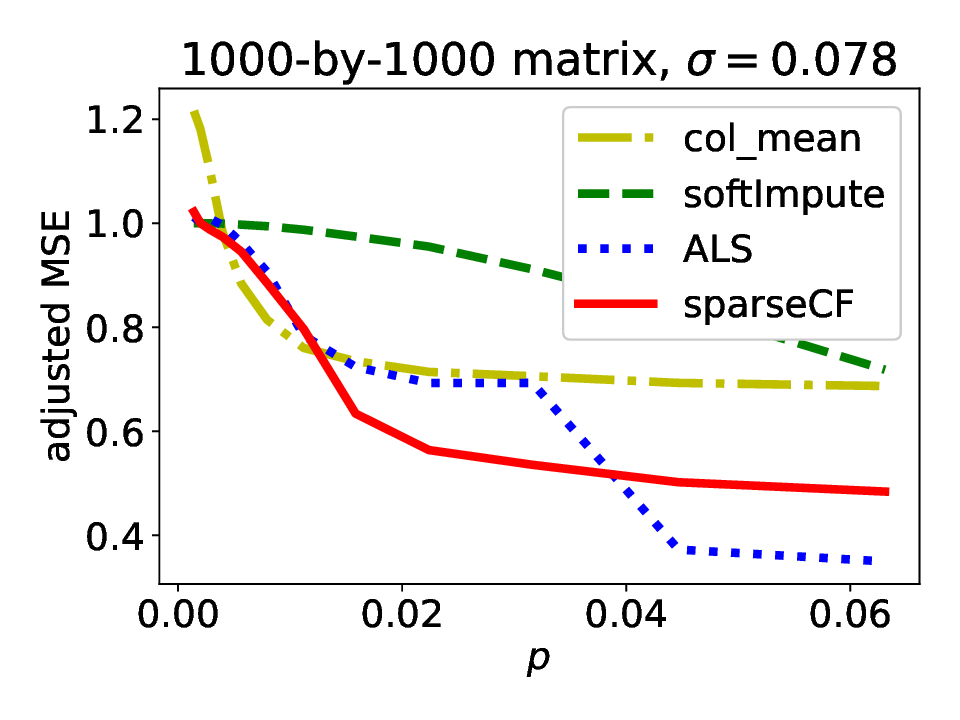}
	\end{subfigure}
~ 
\begin{subfigure}[t]{0.32\textwidth}
\centering
\includegraphics[height=1.7in]{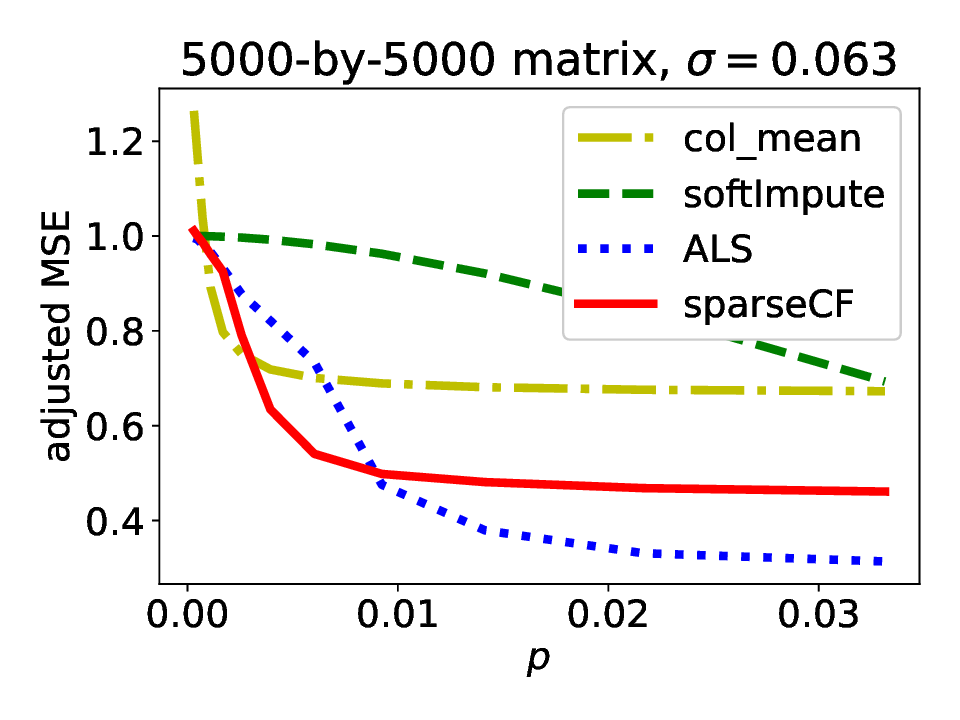}
\end{subfigure}
	\caption{Adjusted MSE of missing entries vs. sampling probability $p$. The rank of the ground truth matrix is 10, and the observations are perturbed with mean zero additive Gaussian noise with variance $\sigma^2$. Results are shown left to right for matrices of sizes 500-by-500, 1000-by-1000, and 5000-by-5000.} \label{fig:mse_p}
\end{figure*}

Figure \ref{fig:mse_kappa} shows the adjusted MSE of the algorithms with respect to the exponent of the density parameter $\kappa$ where $p = n^{-1 + \kappa}$. This rescales the $x$-axis so that the small values of $p$ are more visible. We plot only up to $\kappa = 0.6$ as we are focusing on the sparse regime with little overlaps in entries between pairs of rows and columns. Notice the dependence on the performance of our algorithm with respect to the radius parameter $t$ as illustrated best in Figure \ref{fig:5000_mse_kappa}. For too small values of $t$ the alg is suboptimal as it does not aggregate data sufficiently, but for too large values of $t$ the algorithm again is suboptimal as it simply estimates zeros due to the BFS trees running out of vertices. 

\begin{figure*}[t!]
	\centering
	\begin{subfigure}[t]{0.32\textwidth}
		\centering
		\includegraphics[height=1.7in]{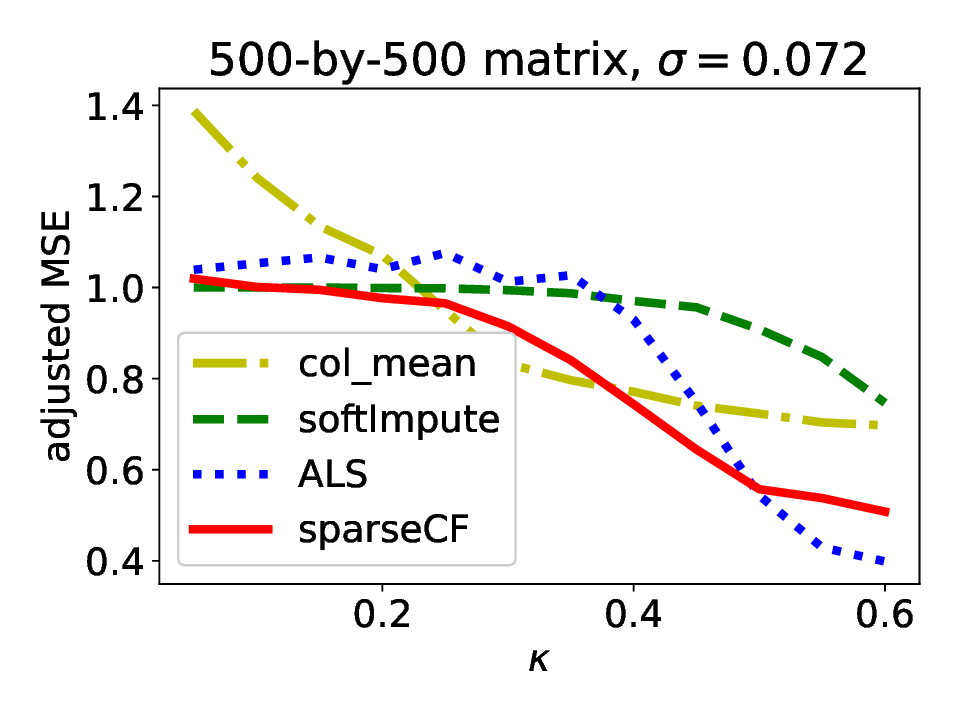}
	\end{subfigure}%
	~ 
	\begin{subfigure}[t]{0.32\textwidth}
		\centering
		\includegraphics[height=1.7in]{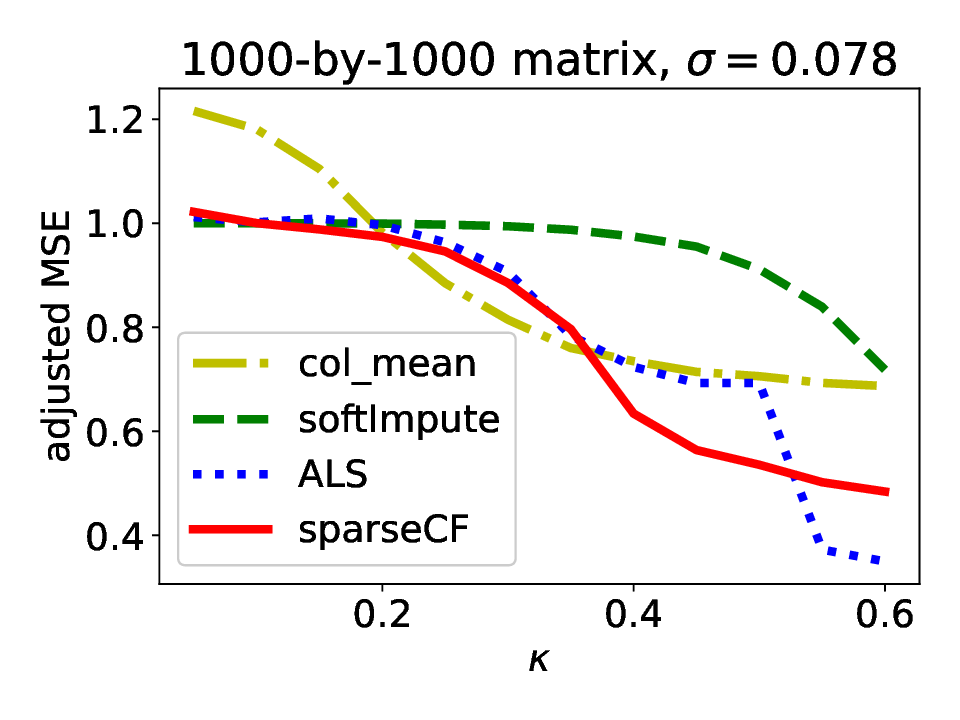}
	\end{subfigure}
	~ 
	\begin{subfigure}[t]{0.32\textwidth}
		\centering
		\includegraphics[height=1.7in]{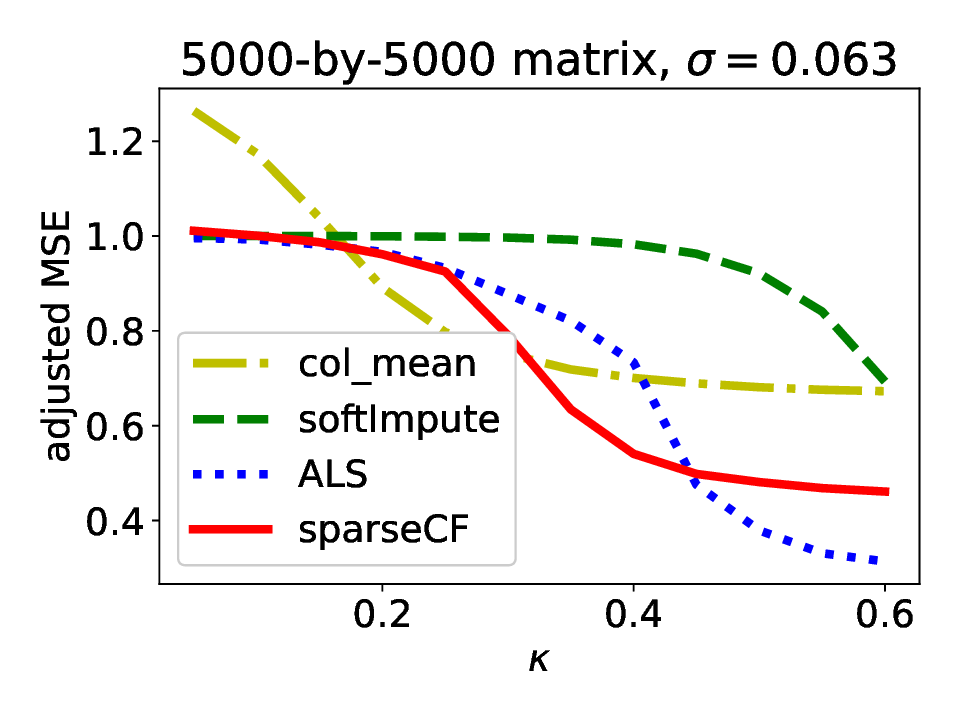}
	\end{subfigure}
	\caption{Adjusted MSE of missing entries vs. sampling probability exponent $\kappa = \ln(pn)/\ln(n)$. The rank of the ground truth matrix is 10, and observations are perturbed with mean zero additive Gaussian noise with variance $\sigma^2$. Results shown left to right for matrices of sizes 500-by-500, 1000-by-1000, and 5000-by-5000.} \label{fig:mse_kappa}
\end{figure*}

Figure \ref{fig:time} shows the time each of the algorithms took to run. We can see that our proposed algorithm is faster than SoftImpute, and this gap in speed is amplified with large $n$. Alternating Least Squares (ALS) is very fast, nearly as fast as the simple mean. Nuclear norm minimization was too slow to run on the size of instances in our example and thus was not included.

\begin{figure*}[t!]
	\centering
	\begin{subfigure}[t]{0.32\textwidth}
		\centering
		\includegraphics[height=1.7in]{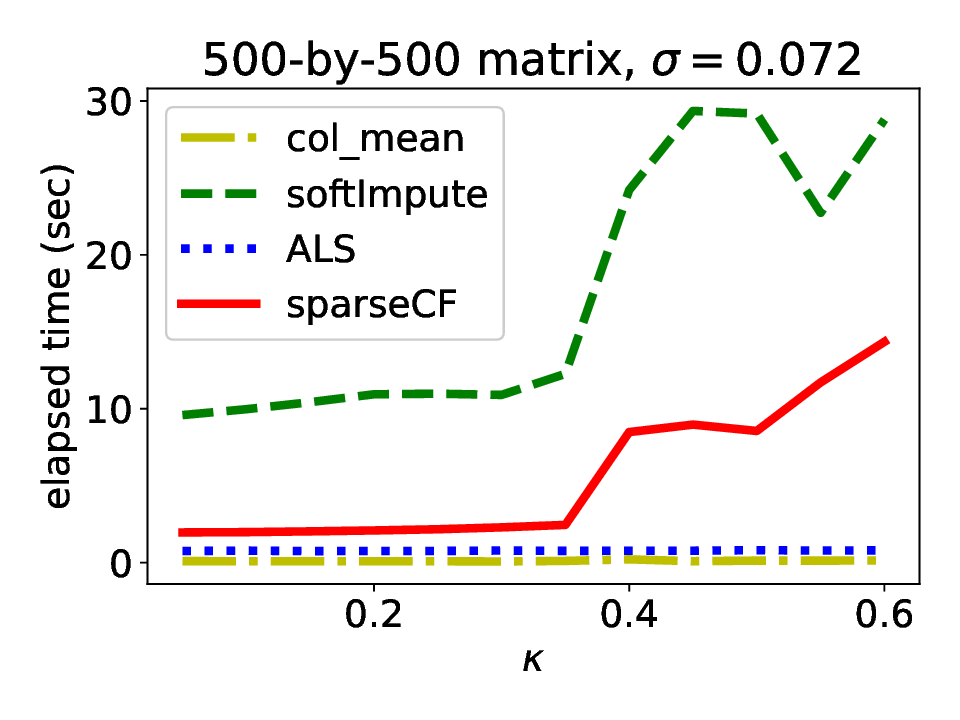}
	\end{subfigure}%
	~ 
	\begin{subfigure}[t]{0.32\textwidth}
		\centering
		\includegraphics[height=1.7in]{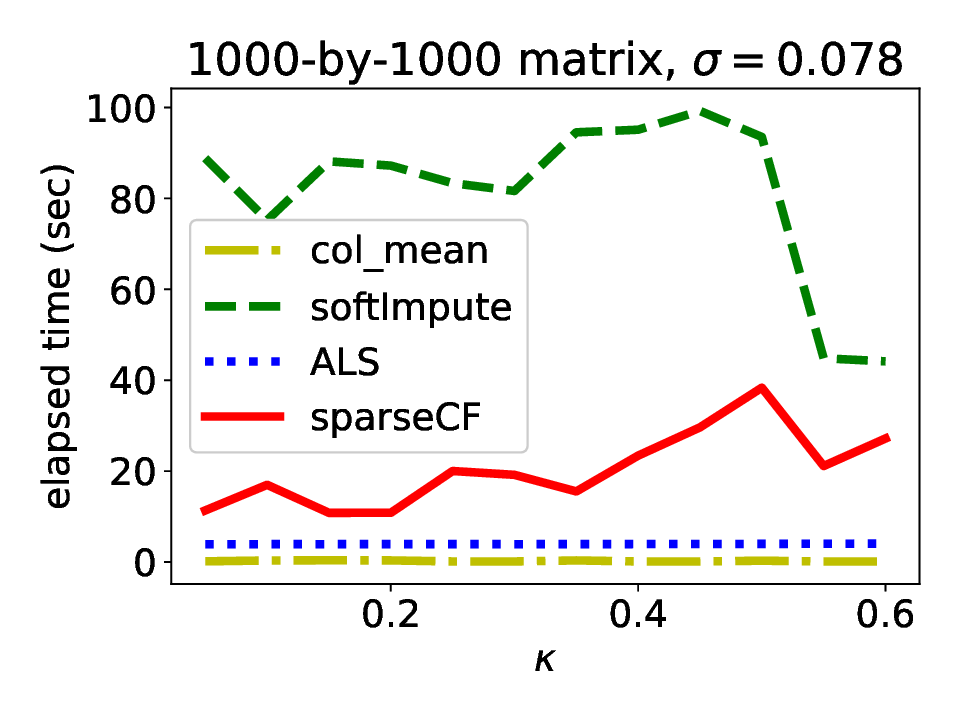}
	\end{subfigure}
	~ 
	\begin{subfigure}[t]{0.32\textwidth}
		\centering
		\includegraphics[height=1.7in]{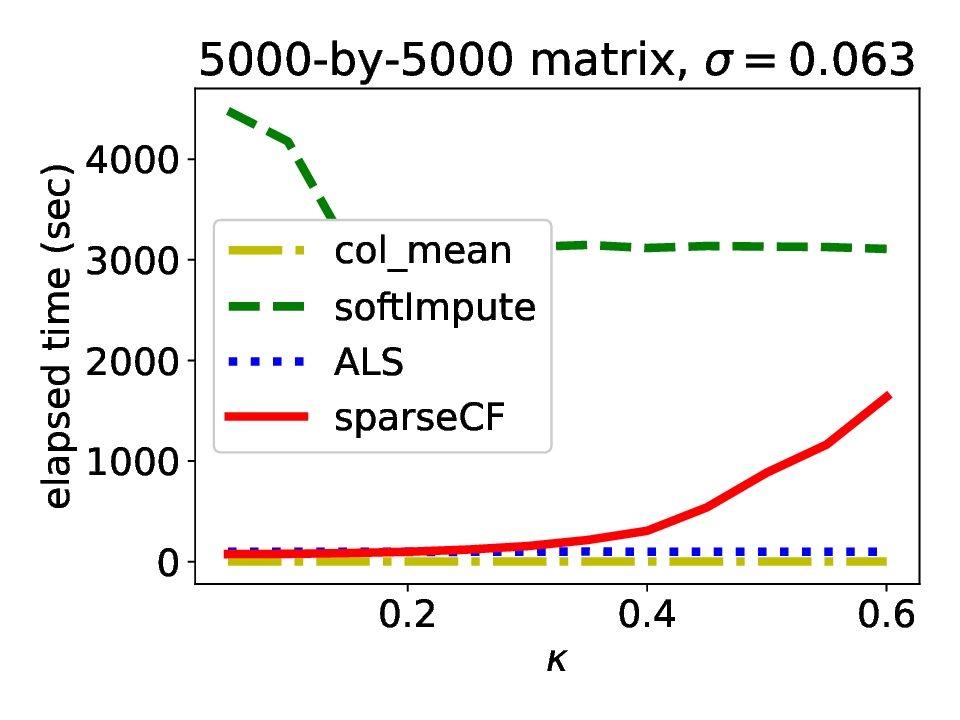}
	\end{subfigure}
	\caption{Computation time vs. sampling probability exponent $\kappa = \ln(pn)/\ln(n)$. The rank of the ground truth matrix is 10, and the observations are perturbed with mean zero additive Gaussian noise with variance $\sigma^2$. Results are shown left to right for matrices of sizes 500-by-500, 1000-by-1000, and 5000-by-5000.} \label{fig:time}
\end{figure*}

\section{Proofs for Theorems \ref{lemma:case1_assumption}, \ref{lemma:case2_assumption}, and \ref{lemma:case3_assumption}} \label{sec:thm_proofs}

In this section, we use the noisy nearest neighbor lemma \ref{lemma:nearest_neighbor} along with to establish Theorems \ref{lemma:case1_assumption}, \ref{lemma:case2_assumption}, and \ref{lemma:case3_assumption}. Proofs of the concentration of distance estimates is deffered to sections \ref{sec:distance_estimates_proof} and \ref{sec:distance_estimate_approx_proof}. 

\subsection{Analyzing Sparse Regime: Proofs of Theorem \ref{lemma:case1_assumption} and \ref{lemma:case3_assumption}} 
We prove that as long as $p = n^{-1+\kappa}$ for 
any $\kappa \in (0,\frac12)$, with high probability, properties \ref{ass:good_distances.1}-\ref{ass:good_distances.3} hold for an appropriately chosen function $d$, and for distance estimates $\hat{d}$ computed according to \eqref{eq:dist1} with $t = \lfloor\frac{\ln(1/p)}{\ln(np)}\rfloor$. We subsequently use Lemma \ref{lemma:nearest_neighbor} to conclude Theorem \ref{lemma:case1_assumption}. The most involved part in the proof is establishing
that property \ref{ass:good_distances.2} holds with high probability for an appropriately chosen $\Delta$, which is delegated to Lemma \ref{lemma:dist_1}.

\medskip
\noindent{\em Good distance $d$ and Property \ref{ass:good_distances.1}.} 
We start by defining the ideal distance $d$ as follows. For all $(u, v) \in [n]^2$, let
\begin{align}
d(\theta_u,\theta_v) & = \|\Lambda^{t+1} Q (e_u - e_v)\|_2^2 = \sum_{k=1}^r \lambda_k^{2(t+1)} (q_k(\theta_u) - q_k(\theta_v))^2. \label{eq:ideal.dist1} 
\end{align}
Recall that $t = \lfloor\frac{\ln(1/p)}{\ln(np)}\rfloor$. Assuming $p = n^{-1 + \kappa}$, 
$\kappa \in (0,\frac12)$
\begin{align}
t & = \Bigg\lfloor\frac{\ln(1/p)}{\ln(np)}\Bigg\rfloor ~=~\Bigg\lfloor \frac{1}{\kappa} - 1 \Bigg\rfloor.
\end{align}

We want to show that there exists $\bias: \RealsP \to \RealsP$ so that $|f(\theta_a, \theta_b) - f(\theta_u, \theta_v) |  \leq \bias(\eta)$for any $\eta > 0$ and $(u,a,v,b) \in [n]^4$ such that $d(\theta_u,\theta_a) \leq \eta$ and $d(\theta_v,\theta_b) \leq \eta$.
By the finite spectrum characterization of the function $f$, it follows that
\begin{align}
|f(\theta_u,\theta_v) - f(\theta_a, \theta_b)| &= |e_u^T Q^T \Lambda Q e_v - e_a^T Q^T \Lambda Q e_b|  \nonumber \\
&= |e_u^T Q^T \Lambda Q (e_v - e_b) - (e_a - e_u)^T Q^T \Lambda Q e_b| \nonumber \\
&\stackrel{(a)}{\leq} B \sqrt{r} \|\Lambda Q (e_v - e_b)\|_2 + B \sqrt{r} \|\Lambda Q (e_u - e_a)\|_2  \nonumber \\
&\leq B \sqrt{r} |\lambda_r|^{-t} \|\Lambda^{t+1} Q (e_v - e_b)\|_2
+ B \sqrt{r} |\lambda_r|^{-t} \|\Lambda^{t+1} Q (e_u - e_a)\|_2 \nonumber \\
& = B |\lambda_r|^{-t} \sqrt{r} \Big( \sqrt{d(\theta_v,\theta_b)} + \sqrt{d(\theta_u,\theta_a)} \Big) \nonumber \\
& \leq 2 B |\lambda_r|^{-t} \sqrt{r \eta}~\equiv \bias(\eta), \label{eq:ideal.bias1}
\end{align}
where (a) follows from assuming that $|q_k(\theta)| \leq B$ for all $k \in [r]$ and $\theta \in [0,1]$. In summary, 
property \ref{ass:good_distances.1} is satisfied for distance function $d$ defined according to \eqref{eq:ideal.dist1} and $\bias(\eta) = 2 B |\lambda_r|^{-t} \sqrt{r \eta}$. 

\medskip
\noindent{\em Good distance estimate $\hat{d}$ and Property \ref{ass:good_distances.2}.} We state the following Lemma when $f$ has rank $r$, whose proof is delegated to Section \ref{sec:distance_estimates_proof}. 
\begin{lemma}\label{lemma:dist_1}
Let $f$ has rank $r$, $p = n^{-1+\kappa}$ for $\kappa \in (0,\frac12)$ such that $1/\kappa$ is not an integer. Consider $\hat{d}$ as computed in \eqref{eq:dist1} with $t = \lfloor\frac{\ln(1/p)}{\ln(np)}\rfloor$. For any 
$\rho \in (0, \kappa)$
\begin{align*}
\max_{u,a \in [n]^2} |d(\theta_u,\theta_a) - \hat{d}(u,a)|
&= O(B r |\lambda_1|^{2/\kappa} n^{ - \frac12 (\kappa - \rho)}), 
\end{align*}
with probability at least 
$1 - O\Big(n^2 \exp\big(-\Theta(n^{\min(\rho, \kappa(t - \frac12))})\big)\Big)$.
\end{lemma}
Lemma \ref{lemma:dist_1} implies that property \ref{ass:good_distances.2} holds with probability $1 - o(1)$ for some $\Delta = \Theta(B r |\lambda_1|^{2/\kappa} n^{-(\kappa - \rho)/2})$ and any $\rho \in (0,\kappa)$. The distance error bound $\Delta$ is minimized by choosing $\rho$ arbitrarily close to 0 so that $\Delta$ can be arbitrarily close to $\Theta(B r |\lambda_1|^{2/\kappa} n^{-\kappa/2}) = \Theta(B r |\lambda_1|^{2/\kappa} (pn)^{-1/2})$.

The corresponding statement for $f$ that has $\varepsilon$-approximate rank $r$ is stated below.
\begin{lemma}\label{lemma:dist_1_approx}
Let $f$ have $\varepsilon$-approximate rank $r$, $p = n^{-1+\kappa}$ for $\kappa \in (0,\frac12)$ such that $1/\kappa$ is not an integer. Consider $\hat{d}$ as computed in \eqref{eq:dist1} with $t = \lfloor\frac{\ln(1/p)}{\ln(np)}\rfloor$. For any 
$\rho \in (0, \kappa)$
\begin{align*}
\max_{u,a \in [n]^2} |d(\theta_u,\theta_a) - \hat{d}(u,a)|
&= O(B r |\lambda_1|^{2/\kappa} n^{ - \frac12 (\kappa - \rho)}) + O\Big(t \eps (1+\eps)^t + t^2 \eps^2 (1+\eps)^{2t-1}\Big), 
\end{align*}
with probability at least 
$1 - O\Big(n^2 \exp\big(-\Theta(n^{\min(\rho, \kappa(t - \frac12))})\big)\Big)$.
\end{lemma}

\medskip
\noindent{\em Sufficient representation and Property \ref{ass:good_distances.3}.} Since $f$ is $L$-Lipschitz, the distance $d$ as defined in \eqref{eq:ideal.dist1} is bounded above by the squared $\ell_2$ distance: 
\begin{align}
d(\theta_u,\theta_v) & = \|\Lambda^{t+1} Q (e_u - e_v)\|_2^2 \nonumber \\
& \leq |\lambda_1|^{2t} \|\Lambda Q (e_u - e_v)\|_2^2 \nonumber \\
& = |\lambda_1|^{2t} \int_0^1 (f(\theta_u, y) - f(\theta_v, y))^2 dy \nonumber \\
&\leq |\lambda_1|^{2t} L^2 |\theta_u - \theta_v|^2. \label{eq:ideal.dist1.a}
\end{align}
We assumed that the latent parameters $\{\theta_u\}_{u \in [n]}$ are sampled i.i.d. uniformly over $[0,1]$. Therefore, for any $\theta_u \in [0,1]$, for any $v \in [n]$ and $\eta' > 0$,  
\begin{align*}
\Prob{d(\theta_u,\theta_v) \leq \eta' ~\big|~ \theta_u} &\geq \Prob{|\lambda_1|^{2t} L^2 |\theta_u - \theta_v|^2 \leq \eta' ~\big|~ \theta_u} \\
& = \Prob{|\theta_u - \theta_v| \leq \frac{\sqrt{\eta'}}{|\lambda_1|^{t} L} ~\big|~ \theta_u} \\
& \geq \min\Big(1, \frac{\sqrt{\eta'}}{|\lambda_1|^{t} L}\Big). 
\end{align*}
Let us define 
\begin{align}
\meas(\eta') =  \frac{(1-\delta) \sqrt{\eta'}}{|\lambda_1|^{t} L}\label{eq:meas_def}
\end{align}
for all $\eta' \in (0, |\lambda_1|^{2t} L^2)$. By an application of Chernoff's bound and a simple majorization argument, it follows that for all $\eta' \in (0,|\lambda_1|^{2t} L^2)$ and $\delta \in (0,1)$,   
\begin{align*}
\Prob{\frac{1}{n-1} \sum_{a \in [n] \setminus u} \Ind\left(d(u,a) \leq \eta' \right) \leq \meas(\eta') ~\big|~ \theta_u } &\leq \exp\left(- \frac{\delta^2 (n-1) \sqrt{\eta'}}{2 |\lambda_1|^{t} L}\right).
\end{align*}
By using union bound over all $n$ indices, it follows that for any $\eta' \in (0, |\lambda_1|^{2t} L^2)$, with probability at least $1 -  n \exp\left(- \frac{\delta^2 (n-1) \sqrt{\eta'}}{2 |\lambda_1|^{t} L}\right)$, property \ref{ass:good_distances.3} is
satisfied with $\meas$ as defined in \eqref{eq:meas_def}.

\medskip
\noindent{\em Concluding Proof of Theorem \ref{lemma:case1_assumption}.} In summary, with probability at least $1-\alpha$ for 
\[\alpha = O(n^2 \exp\left(- \Theta(n^{\min(\rho, \kappa(t - \frac12))}) \right)) + n \exp\left(- \frac{\delta^2 (n-1) \sqrt{\eta - \Delta}}{2 |\lambda_1|^{t} L}\right),\]
properties \ref{ass:good_distances.1}-\ref{ass:good_distances.3} are satisfied for the estimate $\hat{d}$ computed from \eqref{eq:dist1} with $t = \lfloor\frac{\ln(1/p)}{\ln(np)}\rfloor$, and the choices of
\begin{align}
d(\theta_u, \theta_v) &= \|\Lambda^{t+1} Q (e_u - e_v)\|_2^2, \nonumber \\
\bias(\eta) & = 2 B |\lambda_r|^{-t} \sqrt{r \eta}, \nonumber \\
\Delta & = \Theta(B r |\lambda_1|^{2/\kappa} n^{ - \frac12 (\kappa - \rho)}), \nonumber \\
\meas(\eta') & = \frac{(1-\delta) \sqrt{\eta'}}{|\lambda_1|^{t} L}, \label{eq:summary.1}
\end{align}
for any $\eta > 0$, $\rho \in (0, \kappa)$, $\delta \in (0,1)$ and $\eta' = \eta - \Delta \in (0,|\lambda_1|^{2t} L^2)$.
By substituting the expressions for $\bias$, $\meas$, and $\alpha$ into Lemma \ref{lemma:nearest_neighbor}, it follows that 
\begin{align}
\MSE(\hat{F}) &\leq 4B^2 |\lambda_r|^{-2t} r (\eta + \Delta) + \frac{2|\lambda_1|^{2t} L^2}{(1-\delta)^3 p n^2 (\eta-\Delta)} + \exp\Bigg( -\frac{\delta^2 p n^2 (1-\delta)^2 (\eta-\Delta)}{4L^2 |\lambda_1|^{2t}}\Bigg) \nonumber \\
& \quad  + O(n^2 \exp\left(- \Theta(n^{\min(\rho, \kappa(t - \frac12))}) \right)) + n \exp\left(- \frac{\delta^2 (n-1) \sqrt{\eta - \Delta}}{2 |\lambda_1|^{t} L}\right). \label{eq:mse.1.1}
\end{align}
Additionally, for any $\delta' \in (0,1)$, 
\begin{align}
\max_{(u,v)\in[n]^2} |\hat{F}(u, v) - f(\theta_u,\theta_v)| \leq 2B |\lambda_r|^{-t} \sqrt{r (\eta + \Delta)} + \delta'  \label{eq:high_prob_bd1}
\end{align}
with probability at least
\begin{align*}
& 1 - n^2 \exp\left(-\tfrac{\delta^2 (1-\delta)^2 p n^2  (\eta-\Delta)}{4 |\lambda_1|^{2t}L^{2}} \right)- n^2 \exp\left(-\tfrac{\delta'^2 (1 - \delta)^3 p n^2 (\eta-\Delta)}{|\lambda_1|^{2t}L^{2}}\right) \\
& \qquad - O(n^2 \exp\left(- \Theta(n^{\min(\rho, \kappa(t - \frac12))}) \right))- n \exp\left(- \frac{\delta^2 (n-1) \sqrt{\eta'}}{2 |\lambda_1|^{t} L}\right).
\end{align*}
By selecting $\eta = \Theta\big(B r |\lambda_1|^{2/\kappa} n^{ - \frac12 (\kappa - \rho)}\big)$ with a large enough constant, it follows that
\begin{align*}
\eta \pm \Delta &= \Theta(\eta) = \Theta(\Delta), \\
pn^2 \eta &= \Theta(B r |\lambda_1|^{2/\kappa}n^{1 + \kappa - \frac12 (\kappa - \rho)}) = \Omega(B r |\lambda_1|^{2/\kappa}n^{1 + \kappa/2}), \\
n \sqrt{\eta} &= \omega(\sqrt{B r} |\lambda_1|^{1/\kappa} {n^{\frac{7}{8}}}).
\end{align*}
By substituting this choice of $\eta$ and $\delta = \frac12$ into \eqref{eq:mse.1.1}, it follows that 
\begin{align}
\MSE(\hat{F}) & = O\Big(r^2 B^3 |\lambda_r|^2 (|\lambda_1|/|\lambda_r|)^{2/\kappa} n^{ - \frac12 (\kappa - \rho)}\Big).
\end{align}
By choosing $\delta' = 2B |\lambda_r|^{-t} \sqrt{r (\eta + \Delta)}$, it follows that $\delta'^2 pn^2 \eta = \Omega(n)$. Therefore, by substituting into \eqref{eq:high_prob_bd1}, it follows that with probability $1 - o(1)$, 
\begin{align}
\max_{(u,v)\in[n]^2} |\hat{F}(u, v) - f(\theta_u,\theta_v)| & = O\Big(r B^{3/2} |\lambda_r| (|\lambda_1|/|\lambda_r|)^{1/\kappa} n^{ - \frac14 (\kappa - \rho)}\Big).
\end{align}
This completes the proof of Theorem \ref{lemma:case1_assumption}. \hfill \Halmos

\medskip
\noindent{\em Concluding Proof of Theorem \ref{lemma:case3_assumption}.} Like Proof of 
Theorem \ref{lemma:case1_assumption},  with probability at least $1-\alpha$ for 
\[\alpha = O(n^2 \exp\left(- \Theta(n^{\min(\rho, \kappa(t - \frac12))}) \right)) + n \exp\left(- \frac{\delta^2 (n-1) \sqrt{\eta - \Delta}}{2 |\lambda_1|^{t} L}\right),\]
properties \ref{ass:good_distances.1}-\ref{ass:good_distances.3} are satisfied for the estimate $\hat{d}$ computed from \eqref{eq:dist1} with $t = \lfloor\frac{\ln(1/p)}{\ln(np)}\rfloor$, and the choices of
\begin{align}
d(\theta_u, \theta_v) &= \|\Lambda^{t+1} Q (e_u - e_v)\|_2^2, \nonumber \\
\bias(\eta) & = 2 B |\lambda_r|^{-t} \sqrt{r \eta}, \nonumber \\
\Delta & = \Theta(B r |\lambda_1|^{2/\kappa} n^{ - \frac12 (\kappa - \rho)}) + \Theta\Big(t \eps (1+\eps)^t + t^2 \eps^2 (1+\eps)^{2t-1}\Big), \nonumber \\
\meas(\eta') & = \frac{(1-\delta) \sqrt{\eta'}}{|\lambda_1|^{t} L}, \label{eq:summary.1.eps}
\end{align}
for any $\eta > 0$, $\rho \in (0, \kappa)$, $\delta \in (0,1)$ and $\eta' = \eta - \Delta \in (0,|\lambda_1|^{2t} L^2)$. 
Note that the only difference is in choice of $\Delta$ due to Lemma \ref{lemma:dist_1_approx} for $f$ that has $\varepsilon$-approximate
rank $r$.
By substituting the expressions for $\bias$, $\meas$, and $\alpha$ into Lemma \ref{lemma:nearest_neighbor}, it follows that 
\begin{align}
\MSE(\hat{F}) &\leq 4B^2 |\lambda_r|^{-2t} r (\eta + \Delta) + \frac{2|\lambda_1|^{2t} L^2}{(1-\delta)^3 p n^2 (\eta-\Delta)} + \exp\Bigg( -\frac{\delta^2 p n^2 (1-\delta)^2 (\eta-\Delta)}{4L^2 |\lambda_1|^{2t}}\Bigg) \nonumber \\
& \quad  + O(n^2 \exp\left(- \Theta(n^{\min(\rho, \kappa(t - \frac12))}) \right)) + n \exp\left(- \frac{\delta^2 (n-1) \sqrt{\eta - \Delta}}{2 |\lambda_1|^{t} L}\right). \label{eq:mse.3.1}
\end{align}
Additionally, for any $\delta' \in (0,1)$, 
\begin{align}
\max_{(u,v)\in[n]^2} |\hat{F}(u, v) - f(\theta_u,\theta_v)| \leq 2B |\lambda_r|^{-t} \sqrt{r (\eta + \Delta)} + \delta'  \label{eq:high_prob_bd3}
\end{align}
with probability at least
\begin{align*}
& 1 - n^2 \exp\left(-\tfrac{\delta^2 (1-\delta)^2 p n^2  (\eta-\Delta)}{4 |\lambda_1|^{2t}L^{2}} \right)- n^2 \exp\left(-\tfrac{\delta'^2 (1 - \delta)^3 p n^2 (\eta-\Delta)}{|\lambda_1|^{2t}L^{2}}\right) \\
& \qquad - O(n^2 \exp\left(- \Theta(n^{\min(\rho, \kappa(t - \frac12))}) \right))- n \exp\left(- \frac{\delta^2 (n-1) \sqrt{\eta'}}{2 |\lambda_1|^{t} L}\right).
\end{align*}
By selecting $\eta = \Theta\big(B r |\lambda_1|^{2/\kappa} n^{ - \frac12 (\kappa - \rho)}\big) + \Theta(t \eps (1+\eps)^t + t^2 \eps^2 (1+\eps)^{2t-1}\Big)$ with appropriately large enough constants, it follows that
\begin{align*}
\eta \pm \Delta &= \Theta(\eta) = \Theta(\Delta), \\
pn^2 \eta & 
= \Omega(n^{1 + \kappa/2}), \\
n \sqrt{\eta} &= \omega({n^{\frac{7}{8}}}).
\end{align*}
By substituting this choice of $\eta$ and $\delta = \frac12$ into \eqref{eq:mse.1.1}, and using $t < 1/\kappa - 1$, it follows that 
\begin{align}
\MSE(\hat{F}) 
& = O\Big(r^2 (|\lambda_1|/|\lambda_r|)^{2/\kappa} n^{ - \frac12 (\kappa - \rho)}\Big) + O\Big(|\lambda_r|^{-2/\kappa} r \left(\frac{\eps}{\kappa} (1+\eps)^{1/\kappa - 1} + \frac{\eps^2}{\kappa^2} (1+\eps)^{2/\kappa-3}\right)\Big).
\end{align}
By choosing $\delta' = \Theta(B |\lambda_r|^{-t} \sqrt{r (\eta + \Delta)})$, it follows that $\delta'^2 pn^2 \eta = \Omega(n)$. Therefore, by substituting into \eqref{eq:high_prob_bd1}, it follows that with probability $1 - o(1)$, 
\begin{align}
&\max_{(u,v)\in[n]^2} |\hat{F}(u, v) - f(\theta_u,\theta_v)| \nonumber \\
&\qquad= O\Big(r (|\lambda_1|/|\lambda_r|)^{1/\kappa} n^{ - \frac14 (\kappa - \rho)}\Big)  
+ O\Big(|\lambda_r|^{-\frac{1}{\kappa}} \sqrt{r} (\sqrt{\frac{\eps}{\kappa}} (1+\eps)^{\frac{1}{2\kappa} - \frac12} + \frac{\eps}{\kappa} (1+\eps)^{\frac{1}{\kappa}-\frac32})\Big).
\end{align}
This completes the proof of Theorem \ref{lemma:case3_assumption}. \hfill \Halmos

\subsection{Analyzing Sparser Regime: Proof of Theorem \ref{lemma:case2_assumption}} Similar to the proof of Theorem \ref{lemma:case1_assumption}, we prove that as long as $p = \frac{\log n^{1+\kappa}}{n}$ for any $\kappa > 0$, with high probability, properties \ref{ass:good_distances.1}-\ref{ass:good_distances.3} are satisfied for an appropriately chosen function $d$ and for distance estimates $\hat{d}$ computed according to \eqref{eq:dist2} with $t = \lceil \frac{\ln(0.08/p)}{\ln(0.275pn)} - r' \rceil$.
We subsequently use Lemma \ref{lemma:nearest_neighbor} to conclude Theorem \ref{lemma:case2_assumption}. The most involved part in the proof is establishing that property \ref{ass:good_distances.2} holds with high probability for an appropriately chosen $\Delta$, which is delegated to Lemma \ref{lemma:dist_2}.

\medskip
\noindent{\em Good distance $d$ and Property \ref{ass:good_distances.1}.} We start by defining the ideal distance $d$ as follows. For all $(u, v) \in [n]^2$, 
\begin{align}  \label{eq:ideal.dist2} 
d(\theta_u,\theta_v) &= \|\Lambda Q (e_u - e_v)\|_2^2
~= \int_0^1 (f(\theta_u, y) - f(\theta_v, y))^2 dy. 
\end{align}
For any $u, v, a, b \in [n]$ with corresponding $\theta_u, \theta_v, \theta_a, \theta_b \in [0,1]$, 
\begin{align*}
|f(\theta_u,\theta_v) - f(\theta_a, \theta_b)|  &= |e_u^T Q^T \Lambda Q e_v - e_a^T Q^T \Lambda Q e_b| \\
&= |e_u^T Q^T \Lambda Q (e_v - e_b) - (e_a - e_u)^T Q^T \Lambda Q e_b| \\
&\stackrel{(a)}{\leq} B \sqrt{r} \|\Lambda Q (e_v - e_b)\|_2 + B \sqrt{r} \|\Lambda Q (e_u - e_a)\|_2, \\
& = B \sqrt{r} (\sqrt{d(\theta_v, \theta_b)} + \sqrt{d(\theta_u, \theta_a)}), 
\end{align*}
where (a) follows from assuming that $|q_k(\theta)| \leq B$ for all $k \in [r]$ and $\theta \in [0,1]$. It follows
that for any $\eta > 0$, if $d(\theta_u, \theta_a) \leq \eta$ and $d(\theta_v, \theta_b) \leq \eta$, then $|f(\theta_u, \theta_v) - f(\theta_a, \theta_b)| \leq 2B \sqrt{r \eta}$.
In summary, property \ref{ass:good_distances.1} is satisfied for distance $d$ defined in \eqref{eq:ideal.dist2} with 
$\bias: \RealsP \to \RealsP$ defined as $\bias(\eta) = 2B \sqrt{r\eta}$. 

\medskip
\noindent{\em Good distance estimation $\hat{d}$ and Property \ref{ass:good_distances.2}.} We state the following Lemma whose proof is
delegated to Section \ref{sec:distance_estimates_proof}.
\begin{lemma}\label{lemma:dist_2}
Assume that $p =  n^{-1} \ln^{1 +\kappa} n$ for some $\kappa > 0$. 
Consider $\hat{d}$ as computed in \eqref{eq:dist2} with 
$$ t = \Bigg\lceil \frac{\ln (0.08/p) }{\ln (0.275 np)} - r'  \Bigg \rceil.$$
For any $\rho \in (0, \kappa)$, 
\begin{align*}
\max_{u,a \in [n]^2} |d(\theta_u,\theta_a) - \hat{d}(u,a)| &\leq c (\ln n)^{ - \frac12 (\kappa - \rho)}\,
\end{align*}
with probability at least 
\begin{align*}
1 - O\Big(n^2 \exp(-\Theta((\ln n)^{1+\rho})) \Big),
\end{align*}
where $c = c(\lambda_1, \lambda_r, \lambda_{gap}, r, B)$ is independent of $n$ and $\lambda_{gap} = \min_{1 \leq s < s' \leq r} |\lambda_s - \lambda_s'|$.
\end{lemma}
Therefore, property \ref{ass:good_distances.2} is satisfied with probability $1 - o(1)$ 
for some $\Delta =  \Theta\left( (\ln n)^{ - \frac12 (\kappa - \rho)}\right)$ for any $\rho \in (0, \kappa)$. 

\medskip
\noindent{\em Sufficient representation and Property \ref{ass:good_distances.3}.} Since $f$ is $L$-Lipschitz, the distance $d$ as defined in \eqref{eq:ideal.dist1} is bounded
above by squared $\ell_2$ distance: 
\begin{align}
d(\theta_u,\theta_v) & = \|\Lambda Q (e_u - e_v)\|_2^2 ~= \int_0^1 (f(\theta_u, y) - f(\theta_v, y))^2 dy \\
&\leq L^2 |\theta_u - \theta_v|^2. \label{eq:ideal.dist2.a}
\end{align}
Note that the only difference in \eqref{eq:ideal.dist1.a} and \eqref{eq:ideal.dist2.a} is the constant 
$L^2 |\lambda_1|^{2t}$ versus $L^2$. It follows by a similar argument that with probability at least $1 -  n \exp\left(- \frac{\delta^2 (n-1) \sqrt{\eta'}}{2 L}\right)$, for any $\eta' \in (0, L^2)$, 
property \ref{ass:good_distances.3} is satisfied with $\meas(\eta')  = \frac{(1-\delta) \sqrt{\eta'}}{L}$. 

\medskip
\noindent{\em Concluding Proof of Theorem \ref{lemma:case2_assumption}.} In summary, with probability at least $1 - \alpha$ for 
\[\alpha = O\Big(n^2 \exp(-\Theta((\ln n)^{1+\rho})) \Big) + n \exp\left(- \frac{\delta^2 (n-1) \sqrt{\eta'}}{2 L}\right),\]
properties \ref{ass:good_distances.1}-\ref{ass:good_distances.3} are satisfied for the estimate $\hat{d}$ computed from \eqref{eq:dist2} with $t = \lceil \frac{\ln (0.08/p) }{\ln (0.275 np)} - r' \rceil$, and the choices of
\begin{align}
d(\theta_u, \theta_v) &= \|\Lambda Q (e_u - e_v)\|_2^2, \nonumber \\
\bias(\eta) & = 2 B \sqrt{r \eta}, \nonumber \\
\Delta & = \Theta\left( (\ln n)^{ - \frac12 (\kappa - \rho)} \right), \nonumber \\
\meas(\eta') & = \frac{(1-\delta) \sqrt{\eta'}}{L}, \label{eq:summary.2}
\end{align}
for any $\eta > 0$, $\rho \in (0, \kappa)$, $\delta \in (0,1)$ and $\eta' = \eta - \Delta \in (0, L^2)$. By substituting the expressions for $\bias$, $\meas$, and $\alpha$ into Lemma \ref{lemma:nearest_neighbor}, it follows that
\begin{align}
\MSE(\hat{F}) &\leq 4B^2 r (\eta + \Delta) + \frac{2\sigma^2 L^2}{(1-\delta)^3 p n^2 (\eta-\Delta)} + \exp\Bigg( -\frac{\delta^2 p n^2 (1-\delta)^2 (\eta-\Delta)}{4L^2}\Bigg) \nonumber \\
& \quad  + O\Big(n^2 \exp(-\Theta((\ln n)^{1+\rho})) \Big) + n \exp\left(- \frac{\delta^2 (n-1) \sqrt{\eta - \Delta}}{2 L}\right). \label{eq:mse.2.1}
\end{align}
Additionally, for any $\delta' \in (0,1)$, 
\begin{align}
\max_{(u,v)\in[n]^2} |\hat{F}(u, v) - f(\theta_u,\theta_v)| \leq 2B \sqrt{r (\eta + \Delta)} + \delta' \label{eq:high_prob_bd2}
\end{align}
with probability at least
\begin{align*}
& 1 - n^2 \exp\left(-\tfrac{\delta^2 (1-\delta)^2 p n^2  (\eta-\Delta)}{4 L^{2}} \right)- n^2 \exp\left(-\tfrac{\delta'^2 (1 - \delta)^3 p n^2 (\eta-\Delta)}{ L^{2}}\right) \\
& \qquad - O\Big(n^2 \exp(-\Theta((\ln n)^{1+\rho})) \Big) - n \exp\left(- \frac{\delta^2 (n-1) \sqrt{\eta-\Delta}}{2  L}\right).
\end{align*}
By selecting $\eta = \Theta\left( \left(\frac{\ln^{1+\rho} n}{np}\right)^{1/2}\right) = \Theta\Big((\ln n)^{ - \frac12 (\kappa - \rho)} \Big)$ with a large enough constant, it follows that 
\begin{align*}
\eta \pm \Delta &= \Theta(\eta) = \Theta(\Delta), \\
pn^2 \eta &= \Omega(n), \\
n \sqrt{\eta} &= \omega(\sqrt{n}).
\end{align*}
By substituting this choice of $\eta$ and $\delta = \frac12$ into \eqref{eq:mse.2.1} it follows that 
\begin{align}
\MSE(\hat{F}) & = O(\eta) ~=~ O\Big((\ln n)^{ - \frac12 (\kappa - \rho)} \Big).
\end{align}
By choosing $\delta' = \Theta(\sqrt{\eta})$, it follows that $\delta'^2 pn^2 \eta = \omega(\sqrt{n})$. Therefore, by substituting into \eqref{eq:high_prob_bd2}, it follows that with probability $1 - o(1)$, 
\begin{align}
\max_{(u,v)\in[n]^2} |\hat{F}(u, v) - f(\theta_u,\theta_v)| & = O(\sqrt{\eta}) ~=~ O\Big((\ln n)^{ - \frac14 (\kappa - \rho)}\Big).
\end{align}
This completes the proof of Theorem \ref{lemma:case2_assumption}. \hfill \Halmos

\section{Proving distance estimates are close when $f$ has rank $r$} \label{sec:distance_estimates_proof}

This section is dedicated to establishing that the distance estimates \eqref{eq:dist1} and \eqref{eq:dist2} are good
approximations of the desired ideal distances as claimed in the statements of Lemmas \ref{lemma:dist_1} and \ref{lemma:dist_2} when $f$ has rank $r$. We start by establishing key auxiliary concentration results which will lead to their proofs. 

\subsection{Regular enough growth of bread-first-search tree}

Recall that we grow the neighborhood of each $u \in [n]$ in $\cG=([n], \cE')$ and use associated observations in 
$\Mp$ as well as $\Mpp$ to compute the distance estimates $\hat{d}$. By the assumed Bernoulli sampling model, any tuple $(a, b) \in [n]^2$ is independently included in $\cE'$ with probability $p/4$. Therefore, the expected number of immediate neighbors of $u$ (not including itself) is $(n-1)p/4 \approx np/4$. The expected number of nodes at distance $s \geq 1$ from a given $u$ scales as $(np/4)^s$. 
We define some necessary notation before we present the formal statement of this event. 
Given $\delta \in (0,1)$, define 
\begin{align}
\phi(\delta) & = 1 - \left(\frac{1-\delta}{1-\delta\sqrt{2/3}}\right)^{1/2} ~<~1.
\end{align}
For any $p = \omega\big(\frac1n\big)$ and $p = o(1)$, 
\begin{align}\label{eq:sdelta}
s^*(\delta, p, n) & = \sup\Big\{s \geq 1: \frac{p}{8} \left(\frac{(1 + \delta) np}{4} \right)^{s-1} \leq \phi(\delta)\Big\}. 
\end{align}

For any given $\delta$, $s^*(\delta, p, n)$ is well defined for $n$ large enough since $p =o(1)$.
\begin{lemma} \label{lemma:nbrhd_growth}
Let $\omega(\frac{1}{n}) \leq p \leq o(1)$, $\delta \in (0,1)$. For $1\leq s \leq s^*(\delta, p, n)$,  
\begin{align*}
\Prob{\bigcup_{h=1}^{s} \left\{|\cS_{u,h}| \notin \left[\left(\frac{(1 - \delta) n p}{4}\right)^h, \left(\frac{(1 + \delta) n p}{4}\right)^h\right] \right\} }
&\leq 4 \exp\left(-\frac{\delta^2 ((1 - \delta) n p)}{12(1-\delta\sqrt{2/3})} \right)
\end{align*}
\end{lemma}

The proof of Lemma \ref{lemma:nbrhd_growth} follows from standard argument using repeated application of Chernoff's bound
and is well known in the literature in various forms. For completeness, we have included it in the Appendix. Lemma \ref{lemma:nbrhd_growth} suggests definition of events that will hold with high probability. Specifically, for any $u \in [n]$
and $h \geq 1$, define 
\begin{align}\label{eq:event.1}
\cA^1_{u,h}(\delta) & = \left\{|\cS_{u,h}| \in \left[\left(\frac{(1 - \delta) n p}{4}\right)^h, \left(\frac{(1 + \delta) n p}{4}\right)^h\right]\right\}.
\end{align}
We note that by event $\cA^1_{u,h}(\delta)$ we simply require that the number of nodes at distance $h$ from a given node $u \in [n]$ is nearly $(np/4)^h$. However, it does not impose any restrictions on how the nodes are connected or the latent parameters associated with the nodes themselves.

\subsection{Concentration of a Quadratic Form One}

The event $\cap_{h=1}^{s+\ell} \cA^1_{u,h}(\delta)$ implies that the size of $|\cS_{u,h}|$ grows regularly as expected 
for $h \leq s+\ell  \leq s^*(\delta, p, n)$. Conditioned on this event, we prove that a specific quadratic form concentrates around its mean. This will 
be used as the key property to eventually establish that the distance estimates are a good approximation to the ideal distances. 
\begin{lemma} \label{lemma:nhbrhd_vectors}
Let $\omega(\frac{1}{n}) \leq p \leq o(1)$, $\delta \in (0,1)$,  $s \geq 0, \ell\geq 1$ for $s+\ell \leq s^*(\delta, p, n)$. Then
\begin{align*}
& \Prob{|e_k^T Q \tN_{u,s+\ell} - e_k^T \Lambda^{\ell} Q \tN_{u,s}| \geq \lambda_k^{\ell} ((1-\delta) np/4)^{-(s+1)/2} x ~\Big|~ \cap_{h=1}^{s+\ell} \cA^1_{u,h}(\delta)} \\
&\qquad \leq 2 \exp\left(- \frac{x^2 \lambda_k^2}{4}\right),
\end{align*}
as long as $x < \frac{2 ((1-\delta) np/4)^{(s+1)/2}}{B |\lambda_k|(1+|\lambda_k|)}$.
\end{lemma}

\paragraph{Proof of Lemma \ref{lemma:nhbrhd_vectors}.}
Recall that conditioning on event $\cap_{h=1}^{s+\ell} \cA^1_{u,h}(\delta)$ simply imposes the restriction that the neighborhood
of $u \in [n]$ grows at a specific rate, i.e. number of nodes at distances $h \leq s+\ell$ is within
$((1\pm \delta) np/4)^h$. However, this event is independent from latent parameters $\{\theta_i\}_{i \in [n]}$ and 
the realization of observations $M(i,j) = Z(i,j)$ for $(i,j) \in [n]\times[n]$. Consider any realization of the tree $\cT_u^{s+\ell}$ satisfying $\cap_{h=1}^{s+\ell} \cA^1_{u,h}(\delta)$; the tree contains information regarding the depth $s +\ell$ neighborhood of $u$. Given such a realization, let $\cF_{u,h}$ for $0\leq h \leq s+\ell$ denote the sigma-algebra containing information about the latent
parameters, edges and the values associated with $\cT_u^h$, i.e. the depth $h$ BFS tree rooted at $u$. Specifically, 
$\cF_{u, 0}$ contains information about latent parameter $\theta_u$ associated with $u \in [n]$; $\cF_{u,s}$ contains information about latent parameters $\cup_{h=1}^s \{\theta_i\}_{i \in \cS_{u,h}}$ and all edges and observations involved in the depth $h$ BFS tree, i.e. $\{M(i,j)\}_{(i,j) \in \cT_u^h}$. This implies that 
 $ \cF_{u,0} \subset \cF_{u, 1} \subset \cF_{u,2}$, etc.

We shall consider a specific martingale sequence with respect to the filtration $\cF_{u, h}$ that will help establish the desired concentration of $e_k^T Q \tN_{u,s+\ell} - e_k^T \Lambda^{\ell} Q \tN_{u,s}$. For $s+1 \leq h \leq s+\ell$, define 
\begin{align*}
Y_{u,h} &= e_k^T \Lambda^{s+\ell-h} Q \tN_{u,h} \\
D_{u,h} &= Y_{u,h} - Y_{u,h-1} \\
Y_{u,s+\ell} - Y_{u,s} &= e_k^T Q \tN_{u,s+\ell} - e_k^T \Lambda^{\ell} Q \tN_{u,s} \\
&= \sum_{h=s+1}^{s+\ell} D_{u,h}
\end{align*}
Note that $Y_{u,h}$ is measurable with respect to $\cF_{u,h}$ because $e_k^T \Lambda^{s+\ell-h} Q \tN_{u,h}$ only depends on observations in $\cT_u^h$ and latent variables associated to vertices in $\cS_{u,h}$. We will show that $Y_{u,h}$ is martingale with
finite mean with respect to $\cF_{u,h}$ for $s+1\leq h \leq s+\ell$,
\begin{align}\label{eq:martingale}
\E[Y_{u,h} - Y_{u,h-1}~|~ \cF_{u,h-1}] = 0 \text{ and } \E[|D_{u,h}|] & < \infty.
\end{align}
For any $s+1\leq h \leq s+\ell$, 
\begin{align*}
D_{u,h} &= Y_{u,h} - Y_{u,h-1} \\
&= \lambda_k^{s+\ell-h} \left(e_k^T Q \tN_{u,h} - \lambda_k e_k^T Q \tN_{u,h-1}\right) \\
&= \lambda_k^{s+\ell-h} \left(\frac{1}{|\cS_{u,h}|} e_k^T Q N_{u,h} - \lambda_k e_k^T Q \tN_{u,h-1}\right) \\
&= \lambda_k^{s+\ell-h} \left(\frac{1}{|\cS_{u,h}|} \sum_{i \in \cS_{u,h}} N_{u,h}(i) q_k(\theta_i) - \lambda_k e_k^T Q \tN_{u,h-1}\right) \\
&= \sum_{i \in \cS_{u,h}} X_i,
\end{align*}
where for $i \in \cS_{u,h}$, we define 
\begin{align}\label{eq:xi}
X_i &\triangleq \frac{\lambda_k^{s+\ell-h}}{|\cS_{u,h}|} \left(N_{u,h}(i) q_k(\theta_i) - \lambda_k e_k^T Q \tN_{u,h-1}\right).
\end{align}
By definition, 
\begin{align}
N_{u,h}(i) & = \sum_{j \in S_{u, h-1}} \Ind((i,j) \in \cE') M(i,j) N_{u,h-1}(j).
\end{align}
Conditioned on $\cF_{u,h-1}$, $N_{u,h-1}(j)$ for $j \in S_{u,h-1}$ is determined and so is $\theta_j$. However, 
$\theta_i$ is conditionally independent random variable. Also, given the construction of the breadth-first-search
tree, for any given $i \in S_{u, h}$ any of the $j \in S_{u, h-1}$ is equally likely to be its parent with probability $1/|S_{u, h-1}|$.
Therefore, we have that $X_i, ~i \in \cS_{u,h}$ are independent and 
\begin{align}
& \E\Big[X_i | \cF_{u, h-1}\big]   \label{eq:xi.mean} \\
& \quad = \frac{\lambda_k^{s+\ell-h}}{|\cS_{u,h}|} 
\Big(\sum_{j \in S_{u, h-1}} \frac{1}{|S_{u, h-1}|} \E[f(\theta_i, \theta_j) q_k(\theta_i) | \theta_j] N_{u, h-1}(j) 
-  \lambda_k e_k^T Q \tN_{u,h-1}\Big). \nonumber
\end{align}
Now $N_{u, h-1}(j)/|S_{u, h-1}| = \tN_{u, h-1}(j)$. And 
 \begin{align*}
 \E[f(\theta_i, \theta_j) q_k(\theta_i) | \theta_j] 
 & = \sum_{k'=1}^r \lambda_{k'} \E[q_{k'}(\theta_i) q_{k'}(\theta_j) q_k(\theta_i) | \theta_j] \\
 & = \sum_{k'=1}^r \lambda_{k'} q_{k'}(\theta_j) \E[q_{k'}(\theta_i)  q_k(\theta_i)] \\
 & = \lambda_k q_k(\theta_j), 
\end{align*}
where we use the orthonormality of $q_{k'}, ~k' \in [r]$. Therefore, 
\begin{align*}
\sum_{j \in S_{u, h-1}} \frac{1}{|S_{u, h-1}|} \E[f(\theta_i, \theta_j) q_k(\theta_i) | \theta_j] N_{u, h-1}(j) 
& = \sum_{j \in S_{u, h-1}}  \lambda_k q_k(\theta_j) \tN_{u, h-1}(j) \\
& = \lambda_k e_k^T Q \tN_{u,h-1}.
\end{align*}
Therefore, we conclude that for $i \in S_{u, h}$
\begin{align}\label{eq:zero.mean}
\E\Big[X_i | \cF_{u, h-1}\big] & = 0.
\end{align}
That is, $\E[Y_{u,h} - Y_{u, h-1} | \cF_{u,h-1}] = 0$. By definition, we have $N_{u,h}(i) \in [0,1]$ 
for any $i \in \cS_{u, h}$ and $\|q_k\|_\infty \leq B$. Therefore, it follows that for any $i \in  S_{u,h}$, 
\begin{align}
|X_i| & \leq \frac{B (1 + |\lambda_k|) |\lambda_k|^{s+\ell-h}}{|\cS_{u,h}|}.
\end{align}
Therefore, it follows that 
\begin{align}
|D_{u,h}|  & \leq B (1 + |\lambda_k|) |\lambda_k|^{s+\ell-h}.
\end{align}
Thus, we have $\{ (D_{u,h}, \cF_{u,h}): s+1\leq h \leq s+\ell \}$ as a martingale difference sequence with differences being
uniformly bounded. Now we wish to establish its concentration. To that end, consider $X_i$ for $i \in \cS_{u,h}$ as defined in 
\eqref{eq:xi}. Its variance is bounded as 
\begin{align*}
& \Var[X_i ~|~ \cF_{u, h-1}]  \\
& \quad = \frac{\lambda_k^{2(s+\ell-h)}}{|\cS_{u,h}|^2} \Var\Big[\sum_{j \in S_{u, h-1}} \Ind((i,j) \in \cE') M(i,j) N_{u,h-1}(j) q_k(\theta_i)  ~|~ \cF_{u, h-1} \Big].
\end{align*}
Since $\Var[Z] \leq \E[Z^2]$ for any $Z$, we can upper bound the variance expression by the second moment, additionally using the fact that $\Ind((i,j) \in \cE')$ only takes value 1 for a single $j \in S_{u, h-1}$ and otherwise takes value $0$,
\begin{align*}
& \Var[X_i ~|~ \cF_{u, h-1}]  \\
& \quad = \frac{\lambda_k^{2(s+\ell-h)}}{|\cS_{u,h}|^2} \E\Big[\sum_{j \in S_{u, h-1}} \Ind((i,j) \in \cE') M(i,j)^2 N^2_{u,h-1}(j) q^2_k(\theta_i)  ~|~ \cF_{u, h-1} \Big].
\end{align*}
We use the fact that $M(i,j)^2 \leq 1$, $\E[q_k^2(\theta_i)] = 1$ due to orthonormality assumptions on $q_k$, for $i \in \cS_{u,h}$ it holds that $\E[\Ind((i,j) \in \cE' ~|~ \cF_{u, h-1}] = \frac{1}{|\cS_{u,h-1}|}$, so that 
\begin{align*}
\Var[X_i ~|~ \cF_{u, h-1}] & \leq  \frac{\lambda_k^{2(s+\ell-h)}}{|\cS_{u,h}|^2} \frac{\|N_{u,h-1}\|_2^2}{|\cS_{u,h-1}|} 
\stackrel{(a)}{\leq} \frac{\lambda_k^{2(s+\ell-h)}}{|\cS_{u,h}|^2}
\end{align*}
where (a) follows from the assumption that $N_{u,h-1}$ has sparsity $\cS_{u,h-1}$ and has entries bounded in $[0,1]$.
It follows that $X_i$ conditioned on $\cF_{u,h-1}$ is sub-exponential with parameters 
\[\left(\frac{\lambda_k^{(s+\ell-h)}}{|\cS_{u,h}|}, \frac{B (1 + |\lambda_k|) |\lambda_k|^{s+\ell-h}}{|\cS_{u,h}|}\right).\] 
Now $D_{u, h}$ is sum of such $X_i$ for $i \in \cS_{u, h}$ which are independent of each other conditioned on $\cF_{u,h-1}$. Therefore, it follows that conditioned on $\cF_{u, h-1}$, $D_{u, h}$ is sub-exponential with parameters 
\[
\Big(\frac{\lambda_k^{(s+\ell-h)}}{\sqrt{|\cS_{u,h}|}}, \frac{B (1 + |\lambda_k|) |\lambda_k|^{s+\ell-h}}{|\cS_{u,h}|}\Big).
\]
Since $\{ (D_{u,h}, \cF_{u,h}): s+1\leq h \leq s+\ell \}$ is a martingale difference sequence, 
$\sum_{h=s+1}^{s+\ell} D_{u,h}$ conditioned on $\cF_{u,s}$ is sub-exponential with parameters
\[
\Bigg(\sqrt{\sum_{h=s+1}^{s+\ell} \frac{\lambda_k^{2(s+\ell-h)}}{|\cS_{u,h}|}}, \max_{h \in [s+1,s+\ell]} \frac{B (1 + |\lambda_k|) |\lambda_k|^{s+\ell-h}}{|\cS_{u,h}|}\Bigg).
\] 
Under event $\cap_{h=1}^{s+\ell} \cA^1_{u,h}(\delta)$, for any realization of the breadth-first-search tree of $u$, 
$|\cS_{u,h}| \in [((1 - \delta) n p/4)^h, ((1 + \delta) n p/4)^h]$ for all $h \in [s+\ell]$. Therefore, we can bound 
the sub-exponential parameters of $\sum_{h=s+1}^{s+\ell} D_{u,h}$ conditioned on $\cF_{u,s}$ using
the property $p = \omega(1/n)$ or $np = \omega(1)$ as 
\[
\Bigg(\lambda_k^{\ell-1} \sqrt{2} \left(\frac{(1 - \delta) n p}{4}\right)^{-(s+1)/2}, B (1 + |\lambda_k|) |\lambda_k|^{\ell-1}\left(\frac{(1 - \delta) n p}{4}\right)^{-(s+1)}\Bigg).
\]
By Azuma's concentration inequality, for $0 < x < \frac{2 ((1-\delta) np/4)^{(s+1)/2}}{B |\lambda_k|(1+|\lambda_k|)}$,
\begin{align*}
&\Prob{|e_k^T Q \tN_{u,s+\ell} - e_k^T \Lambda^{\ell} Q \tN_{u,s}| \geq \lambda_k^{\ell} ((1-\delta) np/4)^{-(s+1)/2} x ~|~ \cap_{h=1}^{s+\ell} \cA^1_{u,h}(\delta), \cF_{u,s}} \\
&\leq 2 \exp\left(- \min \left(\frac{x^2 \lambda_k^2}{4 }, \frac{x |\lambda_k| ((1 - \delta) n p/4)^{(s+1)/2}}{2 B (1 + |\lambda_k|)}\right) \right) \\
&\leq 2 \exp\left(- \frac{x^2 \lambda_k^2}{4 }\right).
\end{align*}
This completes the proof of Lemma \ref{lemma:nhbrhd_vectors}.
 \hfill \Halmos

Lemma \ref{lemma:nhbrhd_vectors} suggests the following high probability events: for any $u \in [n], k \in [r], x > 0, s \geq 0, \ell \geq 1$, $\delta \in (0,1)$, 
define 
\[
\cA^2_{u,k,s,\ell}(x, \delta) = \Big\{|e_k^T Q \tN_{u,s+\ell} - e_k^T \Lambda^{\ell} Q \tN_{u,s}| \leq \lambda_k^{\ell} ((1-\delta) np/4)^{-(s+1)/2} x\Big\}.
\]

\subsection{Concentration of a Quadratic Form Two}

We state a useful concentration that builds on Lemma \ref{lemma:nhbrhd_vectors} towards establishing 
Lemma \ref{lemma:dist_1}. 
\begin{lemma} \label{lemma:NFN_bound}
Let $\omega(\frac{1}{n}) \leq p \leq o(1)$, $\delta \in (0,1)$,  $s \geq 0, \ell\geq 1$ with $s+\ell \leq s^*(\delta, p, n)$,
and $x \leq B ((1-\delta) np/4)^{1/2}$. Consider any $u, v \in [n]$. Then, conditioned on event 
$\cap_{k=1}^r (\cA^2_{u,k,0,s}(x, \delta) \cap \cA^2_{v,k,0,s+\ell}(x, \delta))$, we have
\begin{align*}
\big|\tN_{u,s}^T F \tN_{v, s + \ell} - e_u^T Q^T \Lambda^{2s + \ell + 1} Q e_v \big| 
&\leq \frac{3 B x }{((1-\delta) np/4)^{1/2}} \Big(\sum_{k=1}^r |\lambda_k |^{2s+\ell+1} \Big).
\end{align*}
and
\begin{align*}
\big|\tN_{u,s}^T F \tN_{v, s + \ell} \big| 
\leq 4B^2 \Big(\sum_{k=1}^r |\lambda_k|^{2s+\ell+1} \Big).
\end{align*}
\end{lemma}

\paragraph{Proof of Lemma \ref{lemma:NFN_bound}.}
Assuming event $\cap_{k=1}^r (\cA^2_{u,k,0,s}(x, \delta) \cap \cA^2_{v,k,0,s+\ell}(x, \delta))$ holds,
and using the fact that $F = Q^T \Lambda Q$, it follows that
\begin{align}
&|\tN_{u,s}^T F \tN_{v, s + \ell} - e_u^T Q^T \Lambda^{2s + \ell + 1} Q e_v| \nonumber \\
&\leq |(\tN_{u,s}^T Q^T - e_u^T Q^T \Lambda^s) (\Lambda Q \tN_{v,s + \ell} - \Lambda^{s+\ell+1} Q e_v)| \nonumber \\
&\qquad + |(\tN_{u,s}^T Q^T - e_u^T Q^T \Lambda^s) \Lambda^{s+\ell+1} Q e_v|
+ |e_u^T Q^T \Lambda^{s + 1} (Q \tN_{v, s + \ell} - \Lambda^{s + \ell} Q e_v)| \nonumber\\
&\leq \Big|\sum_{k=1}^r (e_k ^T Q \tN_{u,s} - e_k^T \Lambda^s Q e_u) (e_k^T \Lambda Q \tN_{v,s + \ell} - e_k^T \Lambda^{s + \ell+1} Q e_v)\Big|\nonumber \\
&\qquad + \Big|\sum_{k=1}^r (e_k ^T Q \tN_{u,s} - e_k^T \Lambda^s Q e_u) e_k^T \Lambda^{s + \ell+1} Q e_v \Big| \nonumber\\
& \qquad  + \Big|\sum_{k=1}^r (e_k^T \Lambda^{s + 1} Q e_u) (e_k^T Q \tN_{v, s + \ell} - e_k^T \Lambda^{s + \ell} Q e_v)\Big| \nonumber\\
&\leq \frac{x}{((1-\delta) np/4)^{1/2}} \Bigg(\frac{x}{((1-\delta) np/4)^{1/2}} + 2B\Bigg)\Bigg(\sum_{k=1}^r |\lambda_k |^{2 s+\ell+1}\Bigg) \nonumber\\
& \leq \frac{3 B x}{((1-\delta) np/4)^{1/2}} \Bigg(\sum_{k=1}^r |\lambda_k |^{2s+\ell+1}\Bigg), \label{eq.NFN_bound.eq2}
\end{align}
where we have used the conditioned event $\cap_{k=1}^r (\cA^2_{u,k,0,s}(x) \cap \cA^2_{v,k,0,s+\ell}(x))$, the model assumption that $\|Q\|_\infty \leq B$, 
and the fact that $x \leq B ((1-\delta) np/4)^{1/2}$ for $n$ sufficiently large. 
%
\noindent From \eqref{eq.NFN_bound.eq2}, it follows that
\begin{align*}
&|\tN_{u,t}^T F \tN_{v,t + \ell}| \\
&\leq |e_u^T Q^T \Lambda^{2t + \ell+1} Q e_v | + |\tN_{u,t}^T F \tN_{v,t + \ell} - e_u^T Q^T \Lambda^{2t + \ell + 1} Q e_v| \\
&\leq (B^2 + 3B^2) \Big(\sum_{k=1}^r |\lambda_k|^{2t+\ell+1}\Big).
\end{align*}
 \hfill \Halmos

\subsection{Concentration of a Quadratic Form Three} 
We establish a final concentration that will lead us to the proof of good distance function property. 
\begin{lemma} \label{lemma:NMN_conc}
Let $\omega(\frac{1}{n}) \leq p \leq o(1)$, $\delta \in (0,1)$,  $s \geq 0, \ell\geq 1$ with $s+\ell \leq s^*(\delta, p, n)$ and 
$0 < x \leq B ((1-\delta) np/4)^{1/2}$. Let $u,v \in [n]$. 
Define event 
$$\Ap(u, v, s, \ell)(x) = \cap_{k=1}^r (\cA^2_{u,k,0,s}(x) \cap \cA^2_{v,k,0,s+\ell}(x)) \cap \cA^1_{u,s} \cap \cA^1_{v,s+\ell}.$$
For $0 < z \leq 4B^2 \sqrt{\Big(\sum_{k=1}^r |\lambda_k|^{2s+\ell+1}\Big) \times \pp ((1-\delta) np/4)^{2s + \ell}}$, conditioned on the event $\Ap(u, v, s, \ell)(x)$, with probability at least 
\[1 - 2\exp\left(- \frac{z^2 }{8 B^2} \right) - \exp\left(-\Theta\left( \pp \left(\frac{(1 - \delta) n p}{4|\lambda_r|^{-1}}\right)^{2s + \ell - \frac12} \right)\right),\]
it holds that  
\begin{align*}
|\frac{1}{\pp} \tN_{u,s} \big( \Mpp + \Mpind \big) \tN_{v,s + \ell} - \tN_{u,s} F \tN_{v,s + \ell}| 
&\leq \frac{|\lambda_r|^{2s}}{(pn)^{1/2}} + z \sqrt{\frac{\sum_{k=1}^r |\lambda_k|^{2s+\ell+1}}{\pp ((1-\delta) np/4)^{2s + \ell}}}.
\end{align*}
\end{lemma}

\paragraph{Proof of Lemma \ref{lemma:NMN_conc}.}
We establish this result by arguing that conditioned on the event $\Ap(u, v, s, \ell)(x)$, the matrix $\Mpp + \Mpind$ is statistically very similar to a freshly sampled dataset with density $\pp$. Recall that $\cEpind$ was constructed so that conditioned on $\cEp$, the set $\cEpind \cup \cEpp$ is distributed according to a Bernoulli($\pp$) sampling model, where each $(u,v) \in [n]^2$ are included in $\cEpind \cup \cEpp$ independently with probability $\pp$. The event $\Ap(u, v, s, \ell)(x)$ depends on $\cEp$ and the values $M(i,j)$ such that $(i,j) \in \cT^s_u\cup \cT^{s+\ell}_v$. Therefore datapoints $M(i,j) = Z(i,j)$ for tuples $(i,j) \notin \cT^s_u\cup \cT^{s+\ell}_v$ are independent from the event $\Ap(u, v, s, \ell)(x)$.

Let us define $\Mppind = [\Mppind(i,j)]$ where
\begin{align*}
\Mppind(i,j) = \begin{cases}
M(i,j) = Z(i,j) &\text{ if }(i,j) \in (\cEpp \cup \cEpind) \text{ and }(i,j) \notin \cT^s_u\cup \cT^{s+\ell}_v \\
Z_{\ind}(i,j) &\text{ if }(i,j) \in \cEpind \text{ and }(i,j) \in \cT^s_u\cup \cT^{s+\ell}_v
\end{cases},
\end{align*}
and $Z_{\ind}(i,j)$ is a freshly sampled observation for edge $(i,j)$, distributed equivalently to $Z(i,j)$. Conditioned on $\cEp$ 
and the event $\Ap(u, v, s, \ell)(x)$, $\Mppind$ has sparsity pattern $\cEpp \cup \cEpind$, which is distributed according to a 
Bernoulli$(\pp)$ sampling model where each $(i,j) \in [n]^2$ is included in $\cEpp \cup \cEpind$ with probability $\pp$. 
Furthermore, conditioned on $\Ap(u, v, s, \ell)$, for each $(i,j) \in \cEpp \cup \cEpind$ with probability $\pp$, the datapoint 
$\Mppind(i,j)$ is independent of all observations used to compute $\tN_{u,s}$ and $\tN_{v,s+\ell}$. As a result, 
$\Mppind(i,j)$ is a fresh independent signal of $F(i,j)$, distributed according to $Z(i,j)$.
First we will argue that 
\[\big(\tfrac{1}{\pp}\big) \tN_{u,s}^T (\Mpp + \Mpind) \tN_{v,s + 1} \approx \big(\tfrac{1}{\pp}\big) \tN_{u,s}^T \Mppind \tN_{v,s + 1}.\]
By construction, $\Mppind$ differs from $\Mpp + \Mpind$ only for indices $(i,j) \in \cEpind \cap (\cT^s_u \cup \cT^{s+\ell}_v)$. 
Therefore, it follows that
\begin{align*}
& |N_{u,s} \Mppind N_{v,s + \ell} - N_{u,s} \big( \Mpp + \Mpind \big) N_{v,s + \ell}| \\
& \leq \sum_{i,j} \Ind((i,j) \in \cEpind  \cap (\cT^s_u \cup \cT^{s+\ell}_v)) |Z_{\ind}(i,j) - Z(i,j)| N_{u,s}(i) N_{v, s+\ell}(j).
\end{align*}
By the boundedness assumption, $|Z_{\ind}(i,j) - Z(i,j)| \leq 1$. 
Furthermore, $N_{u,s}(i) N_{v,s+\ell}(j) \in [0,1]$ is only nonzero for $(i,j) \in \cS_{u,s} \times \cS_{v,s+\ell}$. Therefore,
\begin{align*}
& |N_{u,s} \Mppind N_{v,s + \ell} - N_{u,s} \big( \Mpp + \Mpind \big) N_{v,s + \ell}| \\
& \leq \sum_{i,j} \Ind((i,j) \in \cEpind  \cap (\cT^s_u \cup \cT^{s+\ell}_v)) \Ind((i,j) \in \cS_{u,s} \times \cS_{v,s+\ell}) \\
&= |\{(i,j) \in \cEpind \cap (\cT^s_u \cup \cT^{s+\ell}_v) \cap (\cS_{u,s} \times \cS_{v,s+\ell})\}| =: X.
\end{align*}
Conditioned on $\cEp$ and the event $\Ap(u, v, s, \ell)(x)$, the quantity above, denoted as $X$, is distributed as a Binomial random variable, where each pair $(i,j) \in (\cT^s_u \cup \cT^{s+\ell}_v) \cap (\cS_{u,s} \times \cS_{v,s+\ell})$ is included in the set $\cEpind$ independently with probability $\pp$.
The number of tuples in $(\cT^s_u \cup \cT^{s+\ell}_v) \cap (\cS_{u,s} \times \cS_{v,s+\ell})$ is bounded above by $|\cS_{u,s}| + |\cS_{v, s+\ell}|$, since the only edges in $\cT^s_u \cup \cT^{s+\ell}_v$ that intersect with $\cS_{u,s} \times \cS_{v,s+\ell}$ must be at the last layer of $\cT^s_u$ or $\cT^{s+\ell}_v$. 
By construction, the number of edges in tree $\cT^s_u$ at depth $s$ is equal to $|\cS_{u,s}|$. 
For sufficiently large $n$, by event $\Ap(u, v, s, \ell)(x)$, it follows that $|\cS_{u,s}| \leq |\cS_{v, s+\ell}|$. 
Therefore the random variable $X$ is stochastically dominated by a Binomial$(2|\cS_{v, s+\ell}|, \pp)$ random 
variable. For sufficiently large $n$, conditioned on $\cEp$ and the event $\Ap(u, v, s, \ell)(x)$, by Chernoff's bound,
\begin{align*}
&\Prob{X \geq \frac{\pp |\cS_{u,s}| |\cS_{v,s+\ell}|}{|\lambda_r|^{-2s}(pn)^{1/2}}} \\
&\leq \exp\left(-\frac13 \left(\frac{\pp |\cS_{u,s}| |\cS_{v,s+\ell}|}{|\lambda_r|^{-2s}(pn)^{1/2}} - 2 \pp |\cS_{v, s+\ell}| \right) \right) \\
&= \exp\left(-\frac23 \pp |\cS_{v, s+\ell}| \left(\frac{|\cS_{u,s}|}{2 |\lambda_r|^{-2s}(pn)^{1/2}} - 1  \right) \right) \\
&\leq \exp\left(-\frac23 \pp \left(\frac{(1 - \delta) n p}{4}\right)^{s + \ell} \left(\frac{(1-\delta)^{1/2}}{4|\lambda_r|^{-1}}\left(\frac{(1 - \delta) n p}{4|\lambda_r|^{-2}}\right)^{s-\frac12} - 1  \right) \right) \\
&= \exp\left(-\Theta\left( \pp \left(\frac{(1 - \delta) n p}{4|\lambda_r|^{-1}}\right)^{2s + \ell - \frac12} \right)\right)
\end{align*}
 It follows that conditioned on event $\Ap(u, v, s, \ell)(x)$, with probability at least 
 $1 - \exp\left(-\Theta\left( \pp \left(\frac{(1 - \delta) n p}{4|\lambda_r|^{-1}}\right)^{2s + \ell - \frac12} \right)\right)$, 
\begin{align}
|\frac{1}{\pp} \tN_{u,s} \Mppind \tN_{v, s + \ell} - \frac{1}{\pp} \tN_{u,s} \big( \Mpp + \Mpind \big) \tN_{v,s + \ell} | 
& \leq \frac{X}{\pp |\cS_{u,s}| |\cS_{v,s+\ell}|} \nonumber \\
&\leq \frac{|\lambda_r|^{2s}}{(pn)^{1/2}}. \label{eq:Mppind_diff_bd}
\end{align}
Next, we prove that with high probability,
\[\big(\tfrac{1}{\pp}\big) \big(\tN_{u,s} - \tN_{v,s}\big)^T \Mppind \big(\tN_{u,s + 1} - \tN_{v,s + 1}\big) \approx \tN_{u,s}^T F \tN_{v, s + \ell}.\]
Let $\cF(u, v, s, \ell, x)$ denote all the information related to $\cT_u^s$ and $\cT_v^{s+\ell}$, including the node latent 
parameters and observations in $\Mp$ that are associated to edges in $\cT^s_u \cup \cT^{s+\ell}_v$. Furthermore, let $\cF(u, v,s, \ell, x)$ be conditioned on the event that $\Ap(u, v, s, \ell)(x)$ holds, which is fully determined by the realization of edges and weights in 
$\cT_u^s$ and $\cT_v^{s+\ell}$. We establish concentration of $N_{u,s}^T \Mppind N_{v,s+\ell}$ by showing that the expression 
can be written as a sum of independent random variables conditioned on $\cF(u, v, s, \ell, x)$,
\begin{align*}
N_{u,s}^T \Mppind N_{v,s+\ell}
& = \sum_{i, j} \Ind((i,j) \in \cEpp \cup \cEpind) \Mppind(i,j) N_{u,s}(i) N_{v,s+\ell}(j),
\end{align*}
where each term of the summation is bounded in $[0,1]$ due to the fact that all observed entries are bounded in $[0,1]$. 
Let 
\begin{align*}
\phi(i,j) & = \Ind((i,j) \in \cEpp \cup \cEpind) \Mppind(i,j) N_{u,s}(i) N_{v, s+\ell}(j).
\end{align*} 
By construction, $\{\phi(i,j)\}_{(i,j) \in [n]^2}$ are independent random variables conditioned on $\cF(u, v, s, \ell, x)$, because $N_{u,s}$ and $N_{v,s+\ell}$ are measurable with respect to $\cF(u, v, s, \ell, x)$, and conditioned on $\cE'$, $\cEpp \cup \cEpind$ is distributed according to the Bernoulli$(\pp)$ sampling model, and the corresponding observations in $\Mppind$ are constructed to be independent due to resampling observations $Z_{\ind}(i,j)$ for $(i,j) \in \cT^s_u \cup \cT^{s+\ell}_v$.
We can verify that
\begin{align*}
& \E[\phi(i,j) | \cF(u, v, s, \ell, x) ]  = \pp F(i, j) N_{u,s}(i) N_{v,s+\ell}(j), \quad \text{and} \\
& \Var[\phi(i,j) | \cF(u, v, s, \ell, x)]  \\
& \quad = (N_{u,s}(i) N_{v,s+\ell}(j))^2 \E[\Ind((i,j) \in \cEpp \cup \cEpind) \Mppind(i,j)^2  ~|~ \cF(u, v, s, \ell, x) ] \\
& \quad \stackrel{(a)}{\leq} N_{u,s}(i) N_{v,s+\ell}(j) \E[\Ind((i,j) \in \cEpp \cup \cEpind) \Mppind(i,j)  ~|~ \cF(u, v, s, \ell, x) ] \\
&\quad \leq \pp N_{u,s}(i) N_{v, s+\ell}(j) F(i,j)
\end{align*}
where inequality $(a)$ follows from the assumption that observed entries are within $[0,1]$. Therefore, 
\begin{align}
\E[N_{u,s}^T \Mppind N_{v,s+\ell}  | \cF(u, v, s, \ell, x) ] 
&= \pp N_{u,s}^T F N_{v,s+\ell}, \label{eq:lem.mean.1}
\end{align}
and 
\begin{align}
\Var[N_{u,s}^T \Mppind N_{v,s+\ell}  | \cF(u, v, s, \ell, x)  ] 
& \leq \pp N_{u,s}^T F N_{v,s+\ell} \nonumber \\
& \leq 4 \pp |\cS_{u,s}| |\cS_{v,s+\ell}| B^2 \Big(\sum_{k=1}^r |\lambda_k|^{2s+\ell+1}\Big). \label{eq:lem.var.1}
\end{align}
The last inequality follows from Lemma \ref{lemma:NFN_bound}.
By an application of Bernstein's inequality, for $z \leq 
 4B^2 \Big(\sum_{k=1}^r |\lambda_k|^{2s+\ell+1}\Big)$, 
\begin{align*}
&\bP\left( \big|\frac{1}{\pp} \tN_{u,s} \Mppind \tN_{v,s + \ell} - \tN_{u,s} F \tN_{v,s + \ell}| > z ~|~\cF(u, v, s, \ell, x)  \right) \\
&= \bP\left( \left| N_{u,s}^T \Mppind N_{v,s+\ell} - \pp N_{u,s}^T F N_{v,s+\ell} \right| > \pp |\cS_{u,s}| |\cS_{v,s+\ell}| z ~|~\cF(u, v, s, \ell, x)  \right) \\
&\qquad \leq 2\exp\left(- \min\left(\frac{z^2 \pp |\cS_{u,s}| |\cS_{v,s+\ell}|}{8 B^2 \Big(\sum_{k=1}^r |\lambda_k|^{2s+\ell+1}\Big)}, \frac{z \pp |\cS_{u,s}| |\cS_{v,s+\ell}|}{2} \right)\right) \\
&\qquad\leq 2\exp\left(- \frac{\pp |\cS_{u,s}| |\cS_{v,s+\ell}| z^2}{8 B^2 \Big(\sum_{k=1}^r |\lambda_k|^{2s+\ell+1}\Big)| } \right). 
\end{align*}
Conditioned on the event $\Ap(u, v, s, \ell)(x)$, $|\cS_{u,s}|$ and $|\cS_{v,s+\ell}|$ are lower bounded by $((1-\delta)np/4)^s$ and $((1-\delta)np/4)^{s+\ell}$.
By reparametrizing $z \to z \sqrt{\frac{\sum_{k=1}^r |\lambda_k|^{2s+\ell+1}}{\pp ((1-\delta) np/4)^{2s + \ell}}}$, we conclude that  
\begin{align*}
&\bP\left( \left. \big|\frac{1}{\pp} \tN_{u,s} \Mppind \tN_{v,s + \ell} - \tN_{u,s} F \tN_{v,s + \ell}| > z \sqrt{\frac{\sum_{k=1}^r |\lambda_k|^{2s+\ell+1}}{\pp ((1-\delta) np/4)^{2s + \ell}}} ~\right|~\cF(u, v, s, \ell, x)  \right) \\
& \leq 2\exp\left(- \frac{z^2 }{8 B^2} \right),
\end{align*}
for $0 < z \leq 4 B^2 \sqrt{\Big(\sum_{k=1}^r |\lambda_k|^{2s+\ell+1}\Big) \times \pp ((1-\delta) np/4)^{2s + \ell}}$. 
The final step in the proof is to combine the above probability bound with the inequality stated in \eqref{eq:Mppind_diff_bd}.
 \hfill \Halmos

Define event 
\begin{align} 
\cA^3_{u, v, s, \ell}(z, \delta) & =  \Bigg\{\big|\frac{1}{\pp} \tN_{u,s} \big( \Mpp + \Mpind \big) \tN_{v,s+ \ell} - \tN_{u,s} F \tN_{v,s + \ell}| \leq \\
& \qquad \qquad\qquad \frac{|\lambda_r|^{2s}}{(pn)^{1/2}} + z \sqrt{\frac{\sum_{k=1}^r |\lambda_k|^{2s+\ell+1}}{\pp ((1-\delta) np/4)^{2s + \ell}}} \Bigg\}. \nonumber 
\end{align}


%
%
%

\subsection{Proof of Lemma \ref{lemma:dist_1}}
By statement of Lemma \ref{lemma:dist_1}, we have $t = \lfloor\frac{\ln(1/p)}{\ln(np)}\rfloor$ with $p = n^{-1 + \kappa}$ where
$1/\kappa$ is not an integer. We wish to establish that distance $\hat{d}$, as defined in \eqref{eq:dist1} is a good proxy of
distance $d$ as defined in \eqref{eq:ideal.dist1}. We shall establish this result under event $\cA$ where  
\begin{align} 
\cA & = \cA^1(0.1) \cap \cA^2(n^{\rho/2}, 0.1) \cap \cA^3(n^{\rho/2}, 0.1), \label{eq:lem1.events} 
\end{align}
where 
\begin{align*} 
\cA^3(n^{\rho/2}, 0.1) & = \cap_{u, v \in [n]} \cA^3_{u, v, t, 1}(n^{\rho/2}, 0.1), \\
\cA^2(n^{\rho/2}, 0.1) & =  \cap_{u\in[n]} \cap_{k \in [r]} (\cA^2_{u,k,0,t}(n^{\rho/2}, 0.1)  ~\cap~ \cA^2_{u,k,0,t+1}(n^{\rho/2}, 0.1)), \\
\cA^1(0.1) & = \cap_{u\in[n]} \cap_{s=1}^{t+1} \cA^1_{u,s}(0.1).
\end{align*}
We shall use Lemmas \ref{lemma:nbrhd_growth}, \ref{lemma:nhbrhd_vectors}, \ref{lemma:NFN_bound} and \ref{lemma:NMN_conc} to conclude
the desired result. To that end, we verify that appropriate conditions required in the statement of these Lemmas are satisfied. 

A crucial condition is that $t+1 \leq s^*(n, p, \delta)$ originally imposed by Lemma \ref{lemma:nbrhd_growth}. By definition of $s^*(n, p, \delta)$, it is sufficient to establish that 
\begin{align}
\frac{p}{8} \left(\frac{(1 + \delta) np}{4} \right)^{t} \leq \phi(\delta) \label{eq:phi_ineq}
\end{align}
where recall $\phi(\delta) = 1 - \left(\frac{1-\delta}{1-\delta\sqrt{2/3}}\right)^{1/2}$. We shall fix $\delta = 0.1$ for the convenience 
through the remainder of the proof. To that end, it can be checked that $\phi(0.1) > 0.01$. Therefore, it is sufficient to have 
\[t \leq  \frac{\ln(0.08 /p)}{\ln(0.275 np)} ~<~\frac{\ln(8 \phi(0.1)/p)}{\ln(0.275 np)}.\]
We have chosen $t =  \lfloor\frac{\ln(1/p)}{\ln(np)}\rfloor$. That is,
\begin{align*}
t & = \Big\lfloor \frac{(1-\kappa) \ln n}{\kappa \ln n} \Big\rfloor ~=~ \Big\lfloor \frac{(1-\kappa) }{\kappa} \Big\rfloor~<~\frac{1-\kappa}{\kappa},
\end{align*}
since $1/\kappa$ is not an integer. And, 
\begin{align*}
\frac{\ln(8 \phi(0.1)/p)}{\ln(0.275 np)} & \geq \frac{\ln 0.08 + (1-\kappa)\ln n}{\ln 0.275 + \kappa \ln n} ~\to~\frac{1-\kappa}{\kappa} \\
& >  \Big\lfloor \frac{(1-\kappa) }{\kappa} \Big\rfloor ~=~ t.
\end{align*}
for $n$ large enough. That is, for all $n$ large enough, $t+1 \leq s^*(n, p, 0.1)$. Since $1/\kappa$ is not an integer, 
for some $\gamma \in (0,1)$
\begin{align*}
t & = \Big\lfloor \frac{(1-\kappa) }{\kappa} \Big\rfloor ~=~ \frac{1-\kappa}{\kappa} - \gamma.
\end{align*}
That is, 
\begin{align}\label{eq:lem1.f0}
\kappa(t+2) - 1 & = \kappa ( 1- \gamma) > 0. 
\end{align}
For $\rho \in (0, \kappa)$, we use $x = n^{\rho/2}$ in statement of Lemmas \ref{lemma:nhbrhd_vectors}, \ref{lemma:NFN_bound} and \ref{lemma:NMN_conc}, and $z =  n^{\rho/2}$ in statement of Lemma \ref{lemma:NMN_conc}. We need to verify condition on $x$ and
$z$. Note that $\delta, B, |\lambda_k|, r, t$ are all constant with respect to $n$. Lemma \ref{lemma:nhbrhd_vectors} requires 
\[x < \frac{2 ((1-\delta) np/4)^{1/2}}{B |\lambda_k|(1+|\lambda_k|)} = \Theta((np)^{1/2})\]
and Lemma \ref{lemma:NFN_bound} requires 
\[x < B((1-\delta) np/4)^{1/2}  = \Theta((np)^{1/2}).\]
Since $np = n^\kappa$ and $x = n^{\rho/2}$ with $\rho < \kappa$, both of the above conditions are satisfied for sufficiently large $n$. For Lemma 
\ref{lemma:NMN_conc}, we require 
\[z < 4B^2 \big(\pp ((1-\delta) np/4)^{2t+1} \times (\sum_{k=1}^r |\lambda_k|^{2t+2})\big)^{1/2} \big) = \Theta((\pp (np)^{2t+1})^{1/2}).\]
Now $p' (np/4)^{2t+1} = \Theta(n^{2\kappa (t+1) - 1})$. By \eqref{eq:lem1.f0}, $2\kappa(t+1) - 1 = \kappa(t+2) - 1 + \kappa t > \kappa t \geq \kappa$. By choice, $z = n^{\rho/2}$ for $\rho < \kappa \leq 2\kappa(t+1) - 1$. Therefore, for sufficiently large $n$, the above condition is also satisfied.

Now we are ready to bound the difference between $d(u, v)$ and $\hat{d}(u,v)$ for any $u, v \in [n]$. Recall, 
\begin{align}\label{eq:lem1.f1}
d(\theta_u,\theta_v) & = \| \Lambda^{t+1} Q (e_u - e_v)\|^2 ~=~ (e_u - e_v)^T Q^T \Lambda^{2t+2} Q (e_u - e_v) \\
& = e_u^T Q^T \Lambda^{2t+2} Q e_u + e_v^T Q^T \Lambda^{2t+2} Q e_v -  e_u^T Q^T \Lambda^{2t+2} Q e_v -  e_v^T Q^T \Lambda^{2t+2} Q e_u. \nonumber
\end{align}
Recall, that according to \eqref{eq:dist1},
\begin{align}\label{eq:lem1.f2}
\hat{d}(u,v) & =  \big(\tfrac{1}{\pp}\big) \big(\tN_{u,t} - \tN_{v,t}\big)^T (\Mpp + \Mpind) \big(\tN_{u,t + 1} - \tN_{v,t + 1}\big), \\
& = \frac{1}{\pp} \tN_{u,t}^T (\Mpp + \Mpind) \tN_{u,t+1} + \frac{1}{\pp} \tN_{v,t}^T (\Mpp +  \Mpind )  \tN_{v,t+1}  \nonumber \\
& ~~~ - \frac{1}{\pp} \tN_{u,t}^T (\Mpp + \Mpind) \tN_{v,t+1}  - \frac{1}{\pp} \tN_{v,t}^T (\Mpp +  \Mpind) \tN_{u,t+1}. \nonumber
\end{align}
Under event $\cA$ as defined in \eqref{eq:lem1.events}, by Lemmas \ref{lemma:NFN_bound} and \ref{lemma:NMN_conc}, 
\begin{align}
& \big|\frac{1}{\pp} \tN_{u,t}^T (\Mpp + \Mpind) \tN_{u,t+1} - e_u^T Q^T \Lambda^{2t+2} Q e_u\big| \nonumber \\
&\leq \frac{3 B x }{((1-\delta) np/4)^{1/2}} \Big(\sum_{k=1}^r |\lambda_k |^{2t+2} \Big) + \frac{|\lambda_r|^{2t}}{(pn)^{1/2}} + z \sqrt{\frac{\sum_{k=1}^r |\lambda_k|^{2t+2}}{\pp ((1-\delta) np/4)^{2t + 1}}} \nonumber \\
& \leq \frac{3 B n^{\rho/2} }{(0.225 np)^{1/2}} \Big(\sum_{k=1}^r |\lambda_k |^{2t+2} \Big) 
+ \frac{|\lambda_r|^{2t}}{(pn)^{1/2}} + n^{\rho/2} \sqrt{\frac{\sum_{k=1}^r |\lambda_k|^{2t+2}}{\pp (0.225 np)^{2t + 1}}} \nonumber \\
& = O(B r |\lambda_1|^{2t+2} n^{- (\kappa - \rho)/2}) + O(|\lambda_r|^{2t} n^{-\kappa/2}) + O((r |\lambda_1|^{2t+2})^{1/2} n^{-(2\kappa(t+1) - 1 - \rho)/2}) \nonumber \\
& = O\Big(B r |\lambda_1|^{2t+2} n^{-\frac12 (\kappa - \rho)}\Big), \nonumber
\end{align}
where the last equality follows from observing that the first term asymptotically dominates with respect to $n$ as $\rho < \kappa \leq 2\kappa(t+1) - 1$. Similarly, all other three terms on the right hand side in \eqref{eq:lem1.f1} and \eqref{eq:lem1.f2} can be bounded by same quantities. 
Therefore, we conclude that for any $u, v \in [n]$
\begin{align}\label{final.claim.lem1}
\left| d(\theta_u,\theta_v) - \hat{d}(u,v) \right| & =  O\Big(B r |\lambda_1|^{2/\kappa} n^{-\frac12 (\kappa - \rho)}\Big),
\end{align}
where we used $t < \frac{1-\kappa}{\kappa}$.

To conclude the proof, we need to argue that event $\cA$ holds with high enough probability. To that end, through 
union bound and Lemmas \ref{lemma:nbrhd_growth}, \ref{lemma:nhbrhd_vectors}, and \ref{lemma:NMN_conc}, we have
\begin{align*}
\Prob{\neg\cA} & \leq \Prob{\neg\cA^3(n^{\rho/2}, 0.1) ~|~\cA^1(0.1) \cap \cA^2(n^{\rho/2}, 0.1)} + \\
& \quad \Prob{\neg\cA^2(n^{\rho/2}, 0.1) ~|~\cA^1(0.1)} + \Prob{\neg\cA^1(0.1)}.
\end{align*}
By union bound and Lemma \ref{lemma:NMN_conc}, we have that 
\begin{align*}
&\Prob{\neg\cA^3(n^{\rho/2}, 0.1) ~|~\cA^1(0.1) \cap \cA^2(n^{\rho/2}, 0.1)} \\
& \leq O\left(n^2 \exp\big(-\Theta(n^\rho)\big) + n^2 \exp\left(-\Theta\left( \pp \left(\frac{(1 - \delta) n p}{4|\lambda_r|^{-1}}\right)^{2t + \frac12} \right)\right)\right) \\
&\stackrel{(a)}{\leq} O\left(n^2 \exp\big(-\Theta(n^\rho)\big) + n^2 \exp\left(-\Theta\left( (n p)^{t - \frac12} \right)\right)\right) \\
&\leq O\left(n^2 \exp\big(-\Theta(n^\rho)\big) + n^2 \exp\left(-\Theta\left( n^{\kappa/2} \right)\right)\right).
\end{align*}
where the inequality $(a)$ follows from the choice of $t$, and the fact that $\delta$ and $t$ are constant with respect to $n$.
By union bound and Lemma \ref{lemma:nhbrhd_vectors}, we have that 
\begin{align*}
 \Prob{\neg\cA^2(n^{\rho/2}, 0.1) ~|~\cA^1(0.1)} & \leq O\big(n r \exp\big(-\Theta(n^\rho)\big)\big).
\end{align*}
By union bound and Lemma \ref{lemma:nbrhd_growth}, we have that 
\begin{align*}
 \Prob{\neg\cA^1(0.1)} & \leq O\big(n \exp\big(-\Theta(n^\kappa)\big)\big).
\end{align*}
In summary, \eqref{final.claim.lem1} holds with probability $1 - O\big(n^2 \exp\big(-\Theta(n^{\min(\rho, \kappa(t - \frac12))})\big)\big)$. 
This completes the proof of Lemma \ref{lemma:dist_1}. \hfill \Halmos

%
%
%
%
%
%
%

\subsection{Concentration in The Sparser Regime} We state consequence of earlier results that will help establish
Lemma \ref{lemma:dist_2}. 
\begin{lemma}\label{lemma:zNMN_bound}
Fix $\delta = 0.1$, $p = n^{-1} \ln^{1+\kappa} n$ for some $\kappa > 0$. Let
$$ t = \Bigg\lceil \frac{\ln (0.08/p) }{\ln (0.275 np)} - r'  \Bigg \rceil.$$
Let $\rho \in (0, \kappa)$. Suppose the events, 
$\cap_{k=1}^r (\cA^2_{u,k,0,t}(\ln^{(1+\rho)/2}(n), \delta) \cap \cA^2_{v,k,0,t}(\ln^{(1+\rho)/2}(n), \delta))$,
$\cap_{k \in [r]} \cap_{\ell=1}^{r'} \cA^2_{v,k,t,\ell}(\ln^{(1+\rho)/2}(n), \delta)$,  
$\cap_{\ell=1}^{r'} \cA^3_{u,v,t, \ell}(\ln^{(1+\rho)/2}(n), \delta)$
and $\cap_{s=1}^{t+r'} \big(\cA^1_{u, s}(\delta) \cap \cA^1_{v, s}(\delta)\big)$ hold. Then, 
\begin{align*}
\Big| \textstyle\sum_{\kp \in [r']} z_{\kp} \big(\tfrac{1}{\pp}\big) \tN_{u,t}^T (\Mpp +\Mpind) \tN_{v,t + \kp} - e_u^T Q^T \Lambda^{2} Q e_v \Big| 
&\leq c \ln^{-\frac{(\kappa - \rho)}{2}} n
\end{align*}
for some constant $c= c(\lambda_1, \lambda_r, \lambdag, r, B)$, independent of $n$ with $\lambdag = \min_{1\leq s < s' \leq r} |\lambda_s - \lambda_{s'}|$.
\end{lemma}

\paragraph{Proof of Lemma \ref{lemma:zNMN_bound}.}
By choice of $t$, we have that 
\begin{align}\label{eq.t.1}
\frac{\ln (0.08/p) }{\ln (0.275 np)} - r'  & \leq t ~ < \frac{\ln (0.08/p) }{\ln (0.275 np)} - r' + 1.
\end{align}
We would like to verify that $t+r' \leq s^*(\delta, p, n)$ for $\delta = 0.1$. By definition of $s^*(n, p, \delta)$, it is sufficient to establish that 
\begin{align*}
\frac18 p \left(\frac{(1 + \delta) np}{4} \right)^{t+r' -1} \leq \phi(\delta)
\end{align*}
where recall $\phi(\delta) = 1 - \left(\frac{1-\delta}{1-\delta\sqrt{2/3}}\right)^{1/2}$. For $\delta = 0.1$, it can be verified that
$\phi(0.1) > 0.01$. Therefore, it is sufficient to have 
\[t + r' - 1 \leq  \frac{\ln(0.08 /p)}{\ln(0.275 np)}, \]
which is implied by \eqref{eq.t.1}. 

For $p = n^{-1} \ln^{1+\kappa} n$, $\ln np = \ln \ln^{1+\kappa} n = (1+\kappa) \ln \ln n$. We choose $\rho \in (0, \kappa)$, which implies $\rho \in (0, \frac{\ln(np)}{\ln \ln n} - 1)$. Throughout the proof, we will denote $x = \ln^{(1+\rho)/2} n = \omega(1)$. It follows that for sufficiently large $n$,
\begin{align}\label{eq.x.1}
x^2 ((1-\delta) np/4)^{-1} & = 4 (1-\delta)^{-1} (\ln n)^{-(\kappa - \rho)} = o(1).
\end{align}

Next, we verify properties of $z$. Recall that $z$ is a vector that satisfies $\Lambda^{2 t + 2} \tLambda z = \Lambda^2 \bOne$. 
That is, for any $k \in [r]$, 
\begin{align}\label{eq.z.id}
\sum_{\kp \in [r']} z_{\kp} \lambda_k^{\kp -1} = \lambda_k^{-2t}.
\end{align}
Therefore, 
\begin{align}
\sum_{\kp \in [r']} z_{\kp}  e_u^T Q^T \Lambda^{2t+\kp+1} Q e_v = e_u^T Q^T \Lambda^{2} Q e_v.
\end{align}

Let $L$ be the $r' \times r'$ diagonal matrix containing only the distinct eigenvalues amongst $\{\lambda_k\}_{k \in [r]}$, such that $L_{hh}$ denotes the $h$-th distinct eigenvalue. Let $\tilde{L}$ denote the associated $r' \times r'$  Vandermonde matrix containing only the distinct eigenvalues, i.e. if $\tilde{L}_{ab}$ takes the value of the $a$-th distinct eigenvalue raised to the $(b-1)$-th power. Note that $\Lambda^{2 t + 2} \tLambda z = \Lambda^2 \bOne$ is satisfied whenever 
\[L^{2 t + 2} \tilde{L} z = L \bOne\]
is satisfied. 
Let us define a diagonal matrix $D$ with $D_{bb} = |\lambda_1|^{-(b-1)}$.
Therefore the explicit expression for $z$ is given by
\[z = D (\tilde{L} D)^{-1} L^{-2t} \bOne,\]
such that for $\ell \in [r']$,
\begin{align}\label{eq.z.id.1}
z_{\ell} &= \sum_{h \in [r']} |\lambda_1|^{-(h-1)} (\tilde{L} D)^{-1}_{\ell h} L_{hh}^{-2t}.
\end{align}


Theorem 1 of \cite{Gautschi62} provides bounds on the sum of entries of the inverse of a Vandermonde matrix. It states that for a $N \times N$ Vandermonde matrix $V$ such that $V_{ab} =\lambda_a^{b-1}$, if $V^{-1}$ denotes the inverse of $V$, then
\begin{align*}
\max_{j \in [N]} \sum_{i \in [N]} |(V^{-1})_{ij}|
&\leq \max_{j \in [N]} \prod_{i \neq j} \frac{1 + |\lambda_i|}{|\lambda_i - \lambda_j|}.
\end{align*}
Using this result, we obtain 
%
%
%
\begin{align}
\sum_{j \in [r']} \sum_{i \in [r']} |(\tilde{L} D)^{-1}_{ij}| 
&\leq \sum_{j \in [r']} \prod_{i \neq j} \left(\frac{1 + |L_{ii}|/|\lambda_1|}{|L_{ii} - L_{jj}|/|\lambda_1|}\right) \nonumber \\
&\leq r' \left(\frac{|\lambda_1| + |\lambda_1|}{\min_{i,j}|L_{ii} - L_{jj}|}\right)^{r'-1} \nonumber \\
&= r' \left(\frac{2|\lambda_1|}{\lambdag}\right)^{r'-1}, \label{eq:vandermonde_norm_bd}
\end{align}
where $\lambdag$ is the minimum gap between eigenvalues only amongst the distinct eigenvalues,
$$ \lambdag = \min_{i,j}|L_i - L_j| = \min_{i,j: \lambda_i \neq \lambda_j}|\lambda_i - \lambda_j|.$$ 

Our interest is in bounding 
\begin{align}
&| \textstyle\sum_{\kp \in [r']} z_{\kp} \big(\tfrac{1}{\pp}\big) \tN_{u,t}^T (\Mpp +\Mpind) \tN_{v,t + \kp} - e_u^T Q^T \Lambda^{2} Q e_v | \nonumber \\
&\leq |\sum_{\kp \in [r']} z_{\kp} \left(\big(\tfrac{1}{\pp}\big) \tN_{u,t}^T (\Mpp +\Mpind) \tN_{v,t + \kp} - \tN_{u,t}^T F \tN_{v,t + \kp} \right) | \label{eq:term1} \\
&+ |\sum_{\kp \in [r']} z_{\kp} \left(\tN_{u,t}^T Q^T \Lambda Q \tN_{v,t + \kp} - \tN_{u,t}^T Q^T \Lambda^{\kp+1} Q \tN_{v,t} \right)| \label{eq:term2} \\
&+ |\sum_{\kp \in [r']} z_{\kp} \left(\tN_{u,t}^T Q^T \Lambda^{\kp+1} Q \tN_{v,t} - e_u^T Q^T \Lambda^{2t+\kp+1} Q e_v \right) | \label{eq:term3} 
\end{align}
Conditioned on events $\cap_{k=1}^r (\cA^2_{u,k,0,t}(x, \delta) \cap \cA^2_{v,k,0,t}(x, \delta))$ and given
that all conditions of Lemma \ref{lemma:NFN_bound} are satisfied, it follows that
\begin{align*}
|\eqref{eq:term3}|
&=  \Bigg| \sum_{k\in[r]} \lambda_k^2 \left((e_k^T Q \tN_{u,t}) (e_k^T Q \tN_{v,t}) - (e_k^T \Lambda^{t} Q e_u) (e_k^T \Lambda^{t} Q e_v) \right) \Big(\sum_{\kp \in [r']} z_{\kp} \lambda_k^{\kp-1}\Big)   \Bigg| \\
& \stackrel{(a)}{=}  \Bigg| \sum_{k\in[r]} \lambda_k^{2-2t} \left(e_k^T Q \tN_{u,t} - e_k^T \Lambda^{t} Q e_u \right) \left(e_k^T Q \tN_{v,t} -  e_k^T \Lambda^{t} Q e_v + e_k^T \Lambda^{t} Q e_v \right)  \Bigg|\\
&\qquad + \sum_{k\in[r]} \lambda_k^{2-2t}  e_k^T \Lambda^{t} Q e_u \left( e_k^T Q \tN_{v,t} - e_k^T \Lambda^{t} Q e_v \right) \\
&\leq \sum_{k\in[r]} |\lambda_k|^{2-2t} \left(|\lambda_k|^{2t} x^2 \left(\frac{(1-\delta) np}{4}\right)^{-1} + 2 B |\lambda_k|^{2t} x \left(\frac{(1-\delta) np}{4}\right)^{-1/2} \right) \\
&\leq  x \left(\frac{(1-\delta) np}{4}\right)^{-1/2} \left(x \left(\frac{(1-\delta) np}{4}\right)^{-1/2} + 2 B \right) \sum_{k\in[r]} |\lambda_k|^{2},
\end{align*}
where (a) follows from \eqref{eq.z.id}. 

Similarly, conditioned on events 
$\cap_{k=1}^r \cap_{\ell=1}^{r'}  (\cA^2_{u,k,t,\ell}(x, \delta) \cap \cA^2_{v,k,t,\ell}(x, \delta))$ with $x = \ln^{(1+\rho)/2} n$ and $\delta = 0.1$, we have
\begin{align*}
|\eqref{eq:term2}|
&\leq \sum_{\kp \in [r']} z_{\kp} \Bigg|\sum_k \lambda_k (e_k^T Q \tN_{u,t}) \left(e_k^T Q \tN_{v,t + \kp} - e_k^T \Lambda^{\kp} Q \tN_{v,t}  \right)\Bigg| \\
&\stackrel{(a)}{\leq} \sum_{\kp \in [r']} z_{\kp} B x \left(\frac{(1-\delta) np}{4}\right)^{-(t+1)/2} \sum_{k\in[r]} |\lambda_k|^{\kp+1} \\
&\stackrel{(b)}{=} \sum_{\kp \in [r']} \sum_{h \in [r']}  |\lambda_1|^{-\kp+1} (\tilde{L} D)^{-1}_{\kp h} L_{hh}^{-2t} B x \left(\frac{(1-\delta) np}{4}\right)^{-(t+1)/2} \big(\sum_{k\in[r]} |\lambda_k|^{\kp+1}\big) \\
&\stackrel{(c)}{\leq} |\lambda_1|^2 |\lambda_r|^2 B r x \left(\frac{(1-\delta) |\lambda_r|^{4} np}{4}\right)^{-(t+1)/2} \Big(\sum_{\kp \in [r']} \sum_{h \in [r']} (\tilde{L} D)^{-1}_{\kp h} \Big) \\
&\stackrel{(d)}{\leq} |\lambda_1|^2 |\lambda_r|^2 B r x \left(\frac{(1-\delta) |\lambda_r|^{4} np}{4}\right)^{-(t+1)/2} r' \left(\frac{2|\lambda_1|}{\lambdag}\right)^{r'-1},
\end{align*}
where (a) follows from events $\cap_{k=1}^r \cap_{\ell=1}^{r'}  (\cA^2_{u,k,t,\ell}(x, \delta) \cap \cA^2_{v,k,t,\ell}(x, \delta))$ and showing that $e_k^T Q \tN_{u,t} \leq B$ due to the boundedness of $Q$ and $\|\tN_{u,t}\|_1 \leq 1$ by normalization; (b) follows from \eqref{eq.z.id.1}; (c) follows from $|\lambda_k| \leq |\lambda_1|$ and $|L_{hh}^{-1}| \leq |\lambda_r|^{-1}$;(d) follows from \eqref{eq:vandermonde_norm_bd}. 

Conditioned on the event $\cap_{\ell=1}^{r'} \cA^3_{u,v,t, \ell}(\ln^{(1+\rho)/2}(n), \delta)$ and Lemma \ref{lemma:NFN_bound}, $x = \ln^{(1+\rho)/}2 n$ and $\delta = 0.1$ it follows that
\begin{align*}
|\eqref{eq:term1}| 
&\leq \sum_{\kp \in [r']} z_{\kp} \left(x \left(\frac{\sum_{k=1}^r |\lambda_k|^{2t+\kp+1}}{\pp ((1-\delta) np/4)^{2t+\kp}}\right)^{1/2} + \frac{|\lambda_r|^{2t}}{(pn)^{1/2}} \right) \\
&\stackrel{(a)}{\leq} \sum_{ \kp \in [r']} \sum_{h \in [r']} L_{hh}^{-2t} (\tilde{L} D)^{-1}_{\kp h} |\lambda_1|^{-\kp+1} \left(x\left(\frac{\sum_{k=1}^r |\lambda_k|^{2t+\kp+1}}{\pp ((1-\delta) np/4)^{2t+\kp}}\right)^{1/2} + \frac{|\lambda_r|^{2t}}{(pn)^{1/2}}\right) \\
&\leq |\lambda_r|^{-2t} \left(x\left(\frac{r |\lambda_1|^{2t+2}}{\pp ((1-\delta) np/4)^{2t+1}}\right)^{1/2} + \frac{\max(1, |\lambda_1|^{-r' + 1})}{(pn)^{1/2}}\right) \Big( \sum_{\kp \in [r']} \sum_{h \in [r']} (\tilde{L} D)^{-1}_{\kp h}\Big) \\
&\stackrel{(b)}{\leq}  \left(\left(\frac{x^2 r |\lambda_r|^2 |\lambda_1|}{\pp ((1-\delta) |\lambda_r|^2 |\lambda_1|^{-1} np/4)^{2t+1}}\right)^{1/2} + \frac{\max(1, |\lambda_1|^{-r' + 1})}{(pn)^{1/2}}\right) r' \left(\frac{2|\lambda_1|}{\lambdag}\right)^{r'-1}
\end{align*}
where (a) follows using \eqref{eq.z.id.1} as well as the fact that $np = \omega(1)$ and hence for $n$ sufficiently large, 
$((1-\delta)np/4 )^{-t} \geq ((1-\delta)np/4)^{-t-\kp}$ for any $\kp \geq 0$; 
(b) follows using \eqref{eq:vandermonde_norm_bd}.  
%

In summary, we conclude
\begin{align}
&| \textstyle\sum_{\kp \in [r']} z_{\kp} \big(\tfrac{1}{\pp}\big) \tN_{u,t}^T (\Mpp +\Mpind) \tN_{v,t + \kp} - e_u^T Q^T \Lambda^{2} Q e_v | \nonumber \\
&\leq \left(\left(\frac{x^2 r |\lambda_r|^2 |\lambda_1|}{\pp ((1-\delta) |\lambda_r|^2 |\lambda_1|^{-1} np/4)^{2t+1}}\right)^{1/2} + \frac{\max(1, |\lambda_1|^{-r' + 1})}{(pn)^{1/2}}\right) r' \left(\frac{2|\lambda_1|}{\lambdag}\right)^{r'-1} \label{eq.fin.1}\\
& + |\lambda_1|^2 |\lambda_r|^2 B r x \left(\frac{(1-\delta) |\lambda_r|^{4} np}{4}\right)^{-(t+1)/2} r' \left(\frac{2|\lambda_1|}{\lambdag}\right)^{r'-1} \label{eq.fin.2}\\
& + x \left(\frac{(1-\delta) np}{4}\right)^{-1/2} \left(x \left(\frac{(1-\delta) np}{4}\right)^{-1/2} + 2 B \right) \sum_{k\in[r]} |\lambda_k|^{2}. \label{eq.fin.3}
\end{align}
Observe that due to \eqref{eq.x.1}, $x ((1-\delta) np/4)^{-1/2} = o(1)$ and $t = \Theta(\ln n/ \ln \ln n) = \omega(1)$, hence there exists some constant $c_1 = c_1(\lambda_1, \lambda_r, \lambdag, r, B)$, 
independent of $n$, such that 
\begin{align}
& | {\sf term} \eqref{eq.fin.2} + {\sf term} \eqref{eq.fin.3}| + \frac{\max(1, |\lambda_1|^{-r' + 1})}{(pn)^{1/2}} r' \left(\frac{2|\lambda_1|}{\lambdag}\right)^{r'-1} \nonumber \\
& \qquad \leq c_1 x (np)^{-\frac12}. \label{eq.fin.4}
\end{align}

Recall that we chose $t$ such that by \eqref{eq.t.1},
\begin{align*}
\ln(\pp) &= \ln(p) - \ln(4-p) \\
&= \ln(0.08/(4-p)) - \ln(0.08/p) \\
&\geq \ln(0.08/(4-p)) - (t+r') \ln(0.275 np).
\end{align*}
It follows by $t = \Theta(\frac{\ln(1/p)}{\ln(np)}) = \Theta(\frac{\ln(n)}{\ln \ln n}) = \omega(1)$ that,
\begin{align}
&\ln (\pp ((1-\delta) |\lambda_r|^2 |\lambda_1|^{-1} np/4)^{2t}) \nonumber \\
&\geq \ln(0.08/(4-p)) - (t + r') \ln(0.275 np) + 2t (\ln(\frac{(1-\delta)|\lambda_r|^2}{4|\lambda_1|}) + \ln np) \nonumber \\
&= t \ln(np) + \ln(0.08/(4-p)) - r' \ln(0.275 np) + t \left(2 \ln(\frac{(1-\delta)|\lambda_r|^2}{4|\lambda_1|}) - \ln(0.275)\right) \nonumber \\
&= \Theta(t \ln(np)) = \Theta( \ln(n)) = \omega(1). \label{eq:p_np_t_bd}
\end{align}
This implies that for some constant $c_2 = c_2(\lambda_1, \lambda_r, \lambdag, r, B)$, the square of the first term in \eqref{eq.fin.1} satisfies
\begin{align}
\frac{x^2 r |\lambda_r|^2 |\lambda_1|}{\pp ((1-\delta) |\lambda_r|^2 |\lambda_1|^{-1} np/4)^{2t+1}} (r')^2 \left(\frac{2|\lambda_1|}{\lambdag}\right)^{2(r'-1)} &\leq c_2 x^2 (np)^{-1}. 
\end{align}

Putting everything together, we have that for some constant $c= c(\lambda_1, \lambda_r, \lambdag, r, B)$
\begin{align}
| \textstyle\sum_{\kp \in [r']} z_{\kp} \big(\tfrac{1}{\pp}\big) \tN_{u,t}^T (\Mpp +\Mp_{u,v,t,\kp})) \tN_{v,t + \kp} - e_u^T Q^T \Lambda^{2} Q e_v | &\leq  c x (np)^{-1/2}. 
\end{align}
Replacing $x = \ln^{(1+\rho)/2} n$, we obtain the desired result. 
 \hfill \Halmos

\subsection{Proof of Lemma \ref{lemma:dist_2}} 
The proof of Lemma \ref{lemma:dist_2} would follow from Lemma \ref{lemma:zNMN_bound} and once we verify the probability of events 
required to hold for Lemma \ref{lemma:zNMN_bound} to be applicable. To that end, given $\kappa > 0$
so that $p = n^{-1} \ln^{1+\kappa} n$, let $\rho \in (0, \kappa)$ be parameter of choice. We set 
$$ t = \Bigg\lceil \frac{\ln (0.08/p) }{\ln (0.275np)} - r'  \Bigg \rceil.$$
Define event $\cA$ where  
\begin{align} 
\cA & = \cA^1(0.1) \cap \cA^2(\ln^{(1+\rho)/2}(n), 0.1) \cap \cA^3(\ln^{(1+\rho)/2}(n), 0.1), \label{eq:lem2.events} 
\end{align}
where 
\begin{align*} 
\cA^3(\ln^{(1+\rho)/2}(n), 0.1) & = \cap_{u, v \in [n]} \cap_{\ell=1}^{r'} \cA^3_{u, v, t, \ell}(\ln^{(1+\rho)/2}(n), 0.1), \\
\cA^2(\ln^{(1+\rho)/2}(n), 0.1) & =  \cap_{u\in[n]} \cap_{k \in [r]} \cA^2_{u,k,0,t}(\ln^{(1+\rho)/2}(n), 0.1) \\
& \qquad \qquad \cap_{u\in[n]} \cap_{k \in [r]} \cap_{\ell=1}^{r'} \cA^2_{u,k,t,\ell}(\ln^{(1+\rho)/2}(n), 0.1), \\
\cA^1(0.1) & = \cap_{u\in[n]} \cap_{s=1}^{t+r'} \cA^1_{u,s}(0.1).
\end{align*}
We shall use Lemmas \ref{lemma:nbrhd_growth}, \ref{lemma:nhbrhd_vectors}, \ref{lemma:NFN_bound} and \ref{lemma:NMN_conc} to conclude
the desired result. To that end, we verify that appropriate conditions required in the statement of these Lemmas are satisfied. 

To argue that $\cA^1(0.1)$ holds with high probability, we wish to apply Lemmas \ref{lemma:nbrhd_growth} which requires verifying 
$t+r' \leq s^*(n, p, 0.1)$ which is done in proof of Lemma \ref{lemma:zNMN_bound}. To argue that $\cA^2(\ln^{(1+\rho)/2}(n), 0.1)$ 
and $\cA^3(\ln^{(1+\rho)/2}(n), 0.1)$ hold with high probability, we will utilize Lemmas \ref{lemma:nhbrhd_vectors}, \ref{lemma:NFN_bound} and 
\ref{lemma:NMN_conc} with $x = \ln^{(1+\rho)/2}(n)$ as well as $z = \ln^{(1+\rho)/2}(n)$ in statement of Lemma \ref{lemma:NMN_conc}. 
We need to verify condition on $x$ and $z$. Lemma \ref{lemma:nhbrhd_vectors} requires 
\[x \leq \frac{2 ((1-\delta) np/4)^{1/2}}{B |\lambda_k|(1+|\lambda_k|)}\]
and Lemma \ref{lemma:NFN_bound} requires 
\[x \leq B((1-\delta) np/4)^{1/2}.\]
For sufficiently large $n$ these conditions are satisfied by our choice of $x$ due to $\rho < \kappa$. For Lemma \ref{lemma:NMN_conc}, we require 
\[z  \leq 4 B^2 \big(\pp ((1-\delta) np/4)^{2t+\ell} \times (\sum_{k=1}^r |\lambda_k|^{2t+\ell + 1 })\big)^{1/2}.\]
Now $z = \ln^{(1+\rho)/2} n$ and $np = \ln^{1+\kappa} n$ and since $\rho < \kappa$ we have that $z = o((np)^{1/2})$.
By the same argument as \eqref{eq:p_np_t_bd} in the proof of Lemma \ref{lemma:zNMN_bound}, $\pp ((1-\delta) |\lambda_r| np/4)^{2t} = \omega(1)$. As a result, the right hand side of the inequality is $\omega((np)^{\ell/2})$, which implies that for sufficiently large $n$, the above condition on $z$ is satisfied. 

Conditioned on event $\cA$, by Lemma \ref{lemma:zNMN_bound} it follows immediately that for distances
defined as per \eqref{eq:ideal.dist2} and \eqref{eq:dist2}, 
\begin{align}
\max_{u, v \in [n]} | d(\theta_u,\theta_v) - \hat{d}(u,v) | & =  O\Big(\ln^{-\frac{\kappa - \rho}{2}} n\Big) ~=~O\Bigg(\sqrt{\frac{\ln^{1+\rho} n}{np}}\Bigg). 
\end{align}

To conclude the proof, we need to argue that event $\cA$ holds with high enough probability. To that end, through 
union bound and Lemmas \ref{lemma:nbrhd_growth}, \ref{lemma:nhbrhd_vectors}, and \ref{lemma:NMN_conc}, we have
\begin{align*}
\Prob{\neg\cA} & \leq \Prob{\neg\cA^3(\ln^{(1+\rho)/2}(n), 0.1) ~|~\cA^1(0.1) \cap \cA^2(\ln^{(1+\rho)/2}(n), 0.1)} + \\
& \quad \Prob{\neg\cA^2(\ln^{(1+\rho)/2}(n), 0.1) ~|~\cA^1(0.1)} + \Prob{\neg\cA^1(0.1)}.
\end{align*}
By union bound and Lemma \ref{lemma:NMN_conc}, we have that 
\begin{align}\label{eq:lem.fin.2.1}
&\Prob{\neg\cA^3(\ln^{(1+\rho)/2}(n), 0.1) ~|~\cA^1(0.1) \cap \cA^2(\ln^{(1+\rho)/2}(n), 0.1)} \\
& \leq O\Big(n^2 r' \exp\big(-\Theta(\ln^{1+\rho} n)\big)\Big)
+ O\Big(n^2 r' \exp\left(-\Theta\left( \pp \left(\frac{(1 - \delta) n p}{4|\lambda_r|^{-1}}\right)^{2t + \frac12} \right)\right)\Big).
\end{align}
By the choice of $t$ to satisfy \eqref{eq.t.1}, it follows that $p (0.275 np)^{t+r'} \geq 0.08$. Therefore,
\begin{align*}
\pp \left(\frac{(1 - \delta) n p}{4|\lambda_r|^{-1}}\right)^{2t + \ell - \frac12}
&\geq \frac{p}{4-p} (0.275 np)^{t+r'} \left(\frac{(1-\delta)}{1.1 |\lambda_r|^{-1}}\right)^{t+r'} \left(\frac{(1-\delta) np}{4 |\lambda_r|^{-1}}\right)^{t + \frac12 - r'} \\
&= \frac{0.08}{4-p} \left(\frac{(1-\delta)}{1.1 |\lambda_r|^{-1}}\right)^{2r'-\frac12} \left(\frac{(1-\delta)^2 np}{4.4 |\lambda_r|^{-2}}\right)^{t + \frac12 - r'} \\
&= \Theta\left(\left(\frac{(1-\delta)^2 np}{4.4 |\lambda_r|^{-2}}\right)^{t + \frac12 - r'}\right) \\
&= \Omega(np) = \Theta(\ln^{1+\kappa} n),
\end{align*}
where we used the fact that $\delta, |\lambda_r|, r'$ are all constants, while $t = \omega(1)$ and $np = \omega(1)$. 
By union bound and Lemma \ref{lemma:nhbrhd_vectors}, we have that 
\begin{align}\label{eq:lem.fin.2.2}
 \Prob{\neg\cA^2(\ln^{(1+\rho)/2}(n), 0.1) ~|~\cA^1(0.1)} & \leq O\Big(n r r' \exp\big(-\Theta(\ln^{1+\rho} n)\big)\Big).
\end{align}
By union bound and Lemma \ref{lemma:nbrhd_growth}, we have that 
\begin{align}\label{eq:lem.fin.2.3}
 \Prob{\neg\cA^1(0.1)} & \leq O\Big(n \exp\big(-\Theta(\ln^{1+\kappa} n)\big)\Big).
\end{align}
In summary, the desired claim holds with probability $1 - O\Big(n^2 \exp\big(-\Theta((\ln n)^{1+\rho})\big)\Big)$. 
This completes the proof of Lemma \ref{lemma:dist_2}. \hfill \Halmos

\section{Proving distance estimate is close when $f$ has $\varepsilon$-approximate rank $r$} \label{sec:distance_estimate_approx_proof}


In this section, we extend the result that distance estimate \eqref{eq:dist1} is good approximation 
of the desired ideal distance as claimed in the statement of Lemma \ref{lemma:dist_1_approx} when
$f$ has $\varepsilon$-approximate rank $r$. We will primarily establish robustness of the distance 
estimate with respect to arbitrary, additional error of magnitude at most $\varepsilon$ in each observed
entry. This will help  conclude Lemma \ref{lemma:dist_1_approx} from Lemma \ref{lemma:dist_1}. 

\subsection{Robustness of The Quadratic Form In \eqref{eq:dist1}}

When $f$ has $\eps$-approximate rank $r$, the $F = Q^T \Lambda Q + \beps$ with $\|\beps\|_{\max} \leq \eps$. In contrast, when $f$ has rank $r$, $\beps = 0$, i.e. $F = Q^T \Lambda Q$. That is, the setting of $f$ has $\eps$-approximate rank $r$ can be viewed as a perturbation of the setting with $f$ having rank $r$: each observation 
$M(i,j)$ is first generated as per rank $r$ setting and then arbitrary perturbation or adversarial noise $\eps_{ij}$ is added to it where
$|\eps_{ij}| \leq \eps$. Therefore, we shall analyze the distance estimate as defined in \eqref{eq:dist1} for
the setting of $f$ that has  $\eps$-approximate rank $r$ by bounding the perturbation (or change) 
induced in distance estimates for the setting of $f$ that is rank $r$, due to the addition of such an 
arbitrary perturbation  $\eps_{ij}$. 
\begin{lemma}\label{lem:robust}
Let $f$ have rank $r$, $\omega(\frac{1}{n}) \leq p \leq o(1)$, $\delta \in (0,1)$,  $t \geq 0$ with 
$t+1 \leq s^*(\delta, p, n)$ and $0 < x \leq B ((1-\delta) np/4)^{1/2}$. Let $u,v \in [n]$. 
As before, define event 
$$\Ap(u, v, t, 1)(x) = \cap_{k=1}^r (\cA^2_{u,k,0,t}(x) \cap \cA^2_{v,k,0,t+1}(x)) \cap \cA^1_{u,t} \cap \cA^1_{v,t+1}. $$
We condition on the event that $\Ap(u, v, t, 1)(x)$ holds. Let $\hat{d}(u,v)$ be the distance estimate computed according to \eqref{eq:dist1}. Upon adding arbitrary
$\eps_{ij} \in [-\eps, \eps]$ to $M(i,j)$ for each $(i,j) \in \cE$, with probability at least 
\[1 - \exp\left(-\Theta\left( \pp \left(\frac{(1 - \delta) n p}{4}\right)^{2t + 1} \right)\right),\]
$\hat{d}(u,v)$ changes at most by 
$O(t \eps (1+\eps)^t + t^2 \eps^2 (1+\eps)^{2t-1})$.
\end{lemma}
\paragraph{Proof of Lemma \ref{lem:robust}.} Recall that $\hat{d}(u,v)$ is the sum of four quadratic terms (see \eqref{eq:lem1.f2} for example).  For each of these terms, we shall argue that it changes by 
$O(\eps + t \eps (1+\eps)^t + t^2 \eps^2 (1+\eps)^{2t-1})$ with high probability as claimed. This will conclude the proof. To that end, let us start by considering $\frac{1}{\pp} \tN_{u,t}^T \barM \tN_{u,t + 1}$
where $\barM =\Mpp + \Mpind$; others follow in a similar manner. Specifically, consider
\begin{align*}
N_{u,t}^T \bar{M} N_{v,t+1}
& = \sum_{i, j} \Ind((i,j) \in \cEpp \cup \cEpind) \barM(i,j) N_{u,t}(i) N_{v,t+1}(j).
\end{align*}
Let $\cF(u, v, t, 1, x)$ denote all the information related to 
$\cT_u^t$ and $\cT_v^{t+1}$, including the node latent parameters and observations in $\barM$ that are associated to edges in $\cT^t_u \cup \cT^{t+1}_v$. Furthermore, let $\cF(u, v,t, 1, x)$ be conditioned on the event that 
$\Ap(u, v, t, 1)(x)$ holds, which is fully determined by the realization of edges and weights in 
$\cT_u^t$ and $\cT_v^{t+1}$. We wish to understand how $N_{u,t}^T \bar{M} N_{v,t+1}$ changes if 
we perturb each entry $M(i,j)$ by adding arbitrary $\eps_{ij}$ so that $|\eps_{ij}| \leq \eps$ for all $(i,j) \in \cE$. 
To that end, define 
\begin{align*}
\phi(i,j) & = \Ind((i,j) \in \cEpp \cup \cEpind) \barM(i,j) N_{u,t}(i) N_{v, t+1}(j).
\end{align*} 
By construction, $\{\phi(i,j)\}_{(i,j) \in [n]^2}$ in non-zero only if all four terms in its product are. Given
$\cF(u, v,t, 1, x)$ and conditioned on $\cE'$, $\Ind((i,j) \in \cEpp \cup \cEpind)$ are i.i.d. Bernoulli$(\pp)$. 
Each $\barM(i,j)$ is perturbed at most by $\eps$. By definition $N_{u,t}(i)$ is a product $t$ terms, each of which 
takes value in $[0,1]$ and is perturbed by at most $\eps$ (in absolute) value. Let 
$N_{u,t}(i) = \prod_{s=1}^t w_s$ with $|w_s| \leq 1$ for all $s \leq t$. Let $\eps_s$ be perturbation added
to $w_s$ with $|\eps_s|\leq \eps$ for all $s \leq t$. The change in $N_{u,t}(i)$ is bounded as 
\begin{align}
\big|\prod_{s=1}^t w_s - \prod_{s=1}^t (w_s + \eps_s) \big| & = \big|\sum_{S \subset [t]: S \neq \emptyset} 
\prod_{s \in S} \eps_s \prod_{s \in [t] \backslash S} w_s\big| 
~\leq \sum_{S \subset [t]: S \neq \emptyset} \prod_{s \in S} |\eps_s| \prod_{s \in [t] \backslash S} |w_s| \nonumber \\
& \leq \sum_{S \subset [t]: S \neq \emptyset} \eps^{|S|} 
~= \sum_{s=1}^t {t \choose s} \eps^s \nonumber \\
& =~\eps\Big( \sum_{s=0}^{t-1} \frac{t!}{(t-s-1)! (s+1)!} \eps^s\Big) \nonumber \\
& \leq t \eps \Big(\sum_{s=0}^{t-1}  \frac{(t-1)!}{((t-1)-s)! s!} \eps^s\Big)  ~=~t \eps \Big(\sum_{s=0}^{t-1}  {t-1 \choose s} \eps^s\Big) \nonumber \\
& = t \eps (1+\eps)^{t-1}. 
\end{align}
Similarly, the perturbation in $N_{v, t+1}(j)$ can be bounded above by $(t+1) \eps (1+\eps)^{t}$. That is, the overall perturbation in $\barM(i,j) N_{u,t}(i) N_{v, t+1}(j)$ is bounded above as 
$O(t \eps (1+\eps)^t + t^2 \eps^2 (1+\eps)^{2t-1})$ since each of the 
$\barM(i,j), N_{u,t}(i), N_{v, t+1}(j)$ are $O(1)$. Therefore, the overall perturbation in 
$N_{u,t}^T \bar{M} N_{v,t+1}$ is bounded above by $O(t \eps (1+\eps)^t + t^2 \eps^2 (1+\eps)^{2t-1})$ times the number of $(i,j)$ such that $N_{u,t}(i), N_{v,t+1}(j)$ are non-zero and $\Ind((i,j) \in \cEpp \cup \cEpind)$.
Given $\cF(u, v,t, 1, x)$, this is precisely Binomial($|\cS_{u,t}| |\cS_{v,t+1}|, \pp$). Therefore, by Chernoff's bound, 
it follows that $\sum_{i, j \in [n]} \Ind((i,j) \in \cEpp \cup \cEpind)$ is at most $2 |\cS_{u,t}| |\cS_{v,t+1}| \pp$
with probability at least $1- \exp\big(-|\cS_{u,t}| |\cS_{v,t+1}| \pp/3\big)$. That is, perturbation in 
$N_{u,t}^T \bar{M} N_{v,t+1}$ is bounded above by $O(|\cS_{u,t}| |\cS_{v,t+1}| \pp) \times O( t \eps (1+\eps)^t + t^2 \eps^2 (1+\eps)^{2t-1})$ with probability at least 
$1- \exp\big(-|\cS_{u,t}| |\cS_{v,t+1}| \pp/3\big)$.  Conditioned on the event 
$\Ap(u, v, t, 1)(x)$, $|\cS_{v,s+\ell}|$ are lower bounded by $((1-\delta)np/4)^t$ and 
$((1-\delta)np/4)^{t+1}$. That is, the above claim holds with probability at least
$1 - \exp\big(-(1-\delta)np/4)^{2t+1} \pp/3\big)$. Recall that $\tN_{u,t} = N_{u,t}/|\cS_{u,t}|$. It follows that the perturbation in 
$\frac{1}{\pp} \tN_{u,t}^T \barM \tN_{u,t + 1}$ is bounded above by 
$O(t \eps (1+\eps)^t + t^2 \eps^2 (1+\eps)^{2t-1})$ with probability at least 
$1 - \exp\big(-(1-\delta)np/4)^{2t+1} \pp/3\big)$. Using an identical argument, the same
conclusion holds for perturbation induced in the other three terms in distance
estimate \eqref{eq:dist1}. This completes the proof of Lemma \ref{lem:robust}. 
 \hfill \Halmos

\subsection{Proof of Lemma \ref{lemma:dist_1_approx}}

Using Lemma \ref{lem:robust} and Lemma \ref{lemma:dist_1}, we establish the proof of Lemma 
\ref{lemma:dist_1_approx}. As argued in the proof of Lemma \ref{lemma:dist_1}, for choice of
$t = \lfloor\frac{\ln(1/p)}{\ln(np)}\rfloor$ with $p = n^{-1 + \kappa}$ where
$1/\kappa$ is not an integer and $\delta = 0.1$, we have that $t+1 \leq s^*(n, p, 0.1)$ 
for $n$ large enough. Further, $np = n^\kappa$ and $p' (np/4)^{2t+1} = \Theta(n^{2\kappa (t+1) - 1})$ 
with $\kappa \leq 2\kappa(t+1) - 1$. As in Lemma \ref{lemma:dist_1}, we choose 
$x = n^{\rho/2}$ for $\rho \in (0,\kappa)$ in Lemma \ref{lem:robust}. By this selection, we have 
$x \leq (np/4)^{\frac12}$ for $n$ large enough. As in Lemma \ref{lemma:dist_1}, the event
$\cA$ (recall definition from \eqref{eq:lem1.events}) holds with probability at least 
$1 - O\big(n^2 \exp\big(-\Theta(n^{\min(\rho, \kappa(t - \frac12))})\big)\big)$. Indeed, $\cA$
implies the condition required for Lemma \ref{lem:robust} to hold with $x = n^{\rho/2}$ for all
$u \neq v \in [n]$. Finally, given this, the conclusion of Lemma  \ref{lem:robust} holds for all
$u\neq v \in [n]$ with probability at least $1 - \exp\Big(n^2 \exp\big(-\Theta(n^\kappa)\big)\Big)$. 
In summary, from Lemma \ref{lem:robust} and Lemma \ref{lemma:dist_1}, it follows that 
\begin{align}
|d(u,v) - \hat{d}(u, v)| & \leq O\Big(B r |\lambda_1|^{2/\kappa} n^{-\frac12 (\kappa - \rho)}\Big) + O\Big(t \eps (1+\eps)^t + t^2 \eps^2 (1+\eps)^{2t-1}\Big), 
\end{align}
holds with probability at least $1 - O\big(n^2 \exp\big(-\Theta(n^{\min(\rho, \kappa(t - \frac12))})\big)\big)$. 
This completes the proof of Lemma \ref{lemma:dist_1_approx}. 
 \hfill \Halmos

      \section*{Acknowledgements}
  We gratefully acknowledge funding from the NSF under grants CCF-1948256, CNS-1955997, CMMI-1462158, CMMI-1634259, and a TRIPODS phase I project.

    \bibliographystyle{plain}
    \bibliography{bibliography}
		
    
    \appendix
	\section{Proof of Extra Lemmas}


\begin{lemma} \label{lemma:inequalities}
We use two simple inequalities to argue when a summation is dominated by the single largest term.
For any $\rho \geq 2$,
\begin{align*}
\sum_{s = 1}^r \rho^s \leq 2 \rho^r
\end{align*}
For any $\rho \geq r^{1/(r-1)}$, it holds that $\rho^s \geq s\rho$ for all $s \leq r$. If additionally $\exp(-a \rho) \leq \frac12$,
\begin{align*}
\sum_{s = 1}^r \exp(-a \rho^s) \leq 2 \exp(-a \rho)
\end{align*}
\end{lemma}

Recall the definitions of $\phi$ and $s^*$,
\begin{align}
\phi(\delta) & = 1 - \left(\frac{1-\delta}{1-\delta\sqrt{2/3}}\right)^{1/2} ~<~1.
\end{align}
For any $p = \omega\big(\frac1n\big)$ and $p = o(1)$, 
\begin{align}
s^*(\delta, p, n) & = \sup\Big\{s \geq 1: \frac{p}{8} \left(\frac{(1 + \delta) np}{4} \right)^{s-1} \leq \phi(\delta)\Big\}. 
\end{align}
For any given $\delta$, $s^*(\delta, p, n)$ is well defined for $n$ large enough since $p =o(1)$. Event $\cA^1_{u,s}(\delta)$ is defined as
\[\cA^1_{u,s}(\delta) := \left\{|\cS_{u,s}| \in \left[\left(\frac{(1 - \delta) n p}{4}\right)^s, \left(\frac{(1 + \delta) n p}{4}\right)^s\right]\right\}.\]

\begin{lemma} \label{lemma:A1_prob_bd}
Let $\omega(\frac{1}{n}) \leq p \leq o(1)$, $\delta \in (0,1)$. For $1\leq s \leq s^*(\delta, p, n)$,  
\begin{align*}
\Prob{\neg\cA^1_{u,s}(\delta) ~|~ \cap_{h=1}^{s-1} \cA^1_{u,h}(\delta)} &\leq 2 \exp\left(-\frac{\delta^2}{3(1-\delta\sqrt{2/3})} \left(\frac{(1 - \delta) n p}{4}\right)^{s} \right).
\end{align*}
It follows that for $t + \ell \leq s^*(\delta, p, n)$,
\begin{align*}
\Prob{\cup_{s=1}^{t+\ell} \neg\cA^1_{u,s}(\delta)}
&\leq 4 \exp\left(-\frac{\delta^2 ((1 - \delta) n p)}{12(1-\delta\sqrt{2/3})} \right).
\end{align*}
\end{lemma}

\paragraph{Proof of \ref{lemma:A1_prob_bd}.}
By definition, $s \leq s^*(\delta, p, n)$ implies that
\begin{align}
\frac18 p \left(\frac{(1 + \delta) np}{4} \right)^{s-1} \leq 1 - \left(\frac{1-\delta}{1-\delta\sqrt{2/3}}\right)^{1/2} =: \phi(\delta),\label{eq:ass_UB_t}
\end{align}

Let us denote $\cB_{u,s-1} = \cup_{h=1}^{s-1} \cS_{u,h}$. 
Conditioned on $\cap_{h=1}^{s-1} \cA^1_{u,h}(\delta)$, we can upper bound $|\cB_{u,s-1}|$ by
\begin{align*}
|\cB_{u,s-1}| 
&= 1 + \sum_{h=1}^{s-1} |\cS_{u,h}| ~\leq~ 1 + \sum_{h=1}^{s-1} (\frac{(1 + \delta) np}{4} )^h ~\leq~ 1 + 2 (\frac{(1 + \delta) np}{4} )^{s-1},
\end{align*}
where the last step follows from Lemma \ref{lemma:inequalities} showing that the summation is dominated by the largest term for sufficiently large $n$. By assuming $s \leq s^*(\delta, p, n)$, it follows that for sufficiently large $n$, because $np = \omega(1)$,
\begin{align*}
|\cB_{u,s-1}| &\leq 1 + \frac{16 \phi(\delta) n}{np} \leq \phi(\delta) n.
\end{align*}

Conditioned on the set $\cB_{u,s-1}$ and the set $\cS_{u,s-1}$, any vertex $i \in [n] \setminus \cB_{u,s-1}$ is in $\cS_{u,s}$ independently with probability $(1-(1-\frac{p}{4})^{|\cS_{u,s-1}|})$. Thus the number of vertices in $\cS_{u,s}$ is distributed as a binomial random variable.
By Chernoff's bound,
\begin{align*}
&\Prob{\left.|\cS_{u,s}| > (1 + \delta) (n - |\cB_{u,s-1}|) \left(1-\left(1-\frac{p}{4}\right)^{|\cS_{u,s-1}|}\right) ~\right|~ \cB_{u,s-1}, \cS_{u,s-1}, \cA^1_{u,s-1}} \\
&\leq \exp\left(-\frac13 \delta^2 (n - |\cB_{u,s-1}|) \left(1-\left(1-\frac{p}{4}\right)^{|\cS_{u,s-1}|}\right)\right) \\
&\stackrel{(a)}{\leq} \exp\left(-\frac{1}{3} \delta ^2 (n - |\cB_{u,s-1}|) \left(\frac{p|\cS_{u,s-1}|}{4} \right) \left(1-\frac18 p|\cS_{u,s-1}|\right)\right) \\
&\stackrel{(b)}{\leq} \exp\left(-\frac{1}{12} \delta^2 n p (1 - \phi(\delta)) \left(\frac{(1 - \delta) n p}{4}\right)^{s-1} (1-\phi(\delta))\right)\\
&= \exp\left(-\frac{1}{3} \delta^2 \frac{(1 - \phi(\delta))^2}{1-\delta} (\frac{(1 - \delta) n p}{4})^{s} \right) \\
&\stackrel{(c)}{=} \exp\left(-\frac{\delta^2}{3(1-\delta\sqrt{2/3})} \left(\frac{(1 - \delta) n p}{4}\right)^{s} \right),
\end{align*}
where inequality $(a)$ follows from $(1 - (1-x)^y) \geq xy (1 - \frac12 xy)$ for $x \in (0,1)$ and $y \in \IntegersP$, inequality $(b)$ follows from the event $\cA^1_{u,s-1}$ and the assumption $s \leq s^*(\delta,p,n)$, and equality $(c)$ follows from the fact that we constructed $\phi$ such that $(1 - \delta\sqrt{2/3}) (1-\phi(\delta))^2 = (1-\delta)$.
We obtain a lower bound on $|\cS_{u,s}|$ by a similar argument using Chernoff's bound,
\begin{align*}
&\Prob{\left. |\cS_{u,s}| < (1 - \delta\sqrt{2/3}) (n - |\cB_{u,s-1}|) \left(1-\left(1-\frac{p}{4}\right)^{|\cS_{u,s-1}|}\right) ~\right|~ \cB_{u,s-1}, \cS_{u,s-1}, \cA^1_{u,s-1}} \\
&\leq \exp\left(-\frac12 (\delta \sqrt{2/3})^2 (n - |\cB_{u,s-1}|) \left(1-\left(1-\frac{p}{4}\right)^{|\cS_{u,s-1}|}\right)\right) \\
&\leq \exp\left(-\frac13 \delta^2 (n - |\cB_{u,s-1}|) \left(\frac{p|\cS_{u,s-1}|}{4}\right) \left(1-\frac18 p|\cS_{u,s-1}|\right)\right) \\
&\leq \exp\left(-\frac{\delta^2}{3(1-\delta\sqrt{2/3})} \left(\frac{(1 - \delta) n p}{4}\right)^{s}\right).
\end{align*}
Conditioned on $\cA^1_{u,s-1}$, the above two inequalities show that $\cA^1_{u,s}$ holds with high probability. The upper bound follows from
\begin{align*}
|\cS_{u,s}| 
&\leq (1 + \delta) (n - |\cB_{u,s-1}|) \left(1-\left(1-\frac{p}{4}\right)^{|\cS_{u,s-1}|}\right) \\
&\leq (1 + \delta) \frac{n p}{4} |\cS_{u,s-1}| ~\leq~ \left(\frac{(1 + \delta) np}{4} \right)^s
\end{align*}
and the lower bound follows from
\begin{align*}
|\cS_{u,s}|
&\geq (1 - \delta\sqrt{2/3}) (n - |\cB_{u,s-1}|) \left(1-\left(1-\frac{p}{4}\right)^{|\cS_{u,s-1}|}\right) \\
&\geq (1 - \delta\sqrt{2/3}) n (1 - \phi(\delta)) \frac{p|\cS_{u,s-1}|}{4} \left(1-\frac18 p|\cS_{u,s-1}|\right) \\
&\geq (1 - \delta\sqrt{2/3}) \frac{n p}{4} (1 - \phi(\delta)) |\cS_{u,s-1}| \left(1-\frac18 p \left(\frac{(1 + \delta) np}{4} \right)^{s-1}\right) \\
&\geq (1 - \delta\sqrt{2/3}) \frac{np}{4} (1 - \phi(\delta)) |\cS_{u,s-1}| (1-\phi(\delta)) \\
&= (1 - \delta\sqrt{2/3}) \frac{np}{4} |\cS_{u,s-1}| (1-\phi(\delta))^2 \\
&\stackrel{(b)}{=} \frac{(1 - \delta) np}{4} |\cS_{u,s-1}| ~\geq~ \left(\frac{(1 - \delta) np}{4} \right)^s.
\end{align*}
where equality $(b)$ follows from the fact that we constructed $\phi$ such that $(1 - \delta\sqrt{2/3}) (1-\phi(\delta))^2 = (1-\delta)$.

We finally lower bound the probability of event $\cap_{s=1}^{t+\ell} \cA^1_{u,s}$, by a repeated application of Chernoff's bound for all $s \in [t + \ell]$,
\begin{align*}
\Prob{\cup_{s=1}^{t+\ell} \neg\cA^1_{u,s}(\delta)}
&= \sum_{s=1}^{t+\ell} \Prob{\neg\cA^1_{u,s}(\delta) ~|~ \cap_{h=1}^{s-1} \cA^1_{u,h}(\delta)} \\
&\leq \sum_{s=1}^{t+\ell} 2 \exp\left(-\frac{\delta^2}{3(1-\delta\sqrt{2/3})} \left(\frac{(1 - \delta) n p}{4}\right)^{s} \right) \\
&\stackrel{(a)}{\leq} 4 \exp\left(-\frac{\delta^2 ((1 - \delta) n p)}{12(1-\delta\sqrt{2/3})} \right),
\end{align*}
where inequality $(a)$ follows from the assumption that $pn = \omega(1)$ such that the largest term in the summation dominates. \hfill \Halmos

\end{document}